\documentclass[11pt,letterpaper]{amsart}
\usepackage[hmargin=2.5cm,vmargin=2.5cm]{geometry}

\usepackage{amssymb}
\usepackage{amsmath}
\usepackage{amsfonts}
\usepackage{amsthm}
\usepackage{stmaryrd}
\usepackage[all]{xy}
\usepackage{mathrsfs}
\usepackage{graphicx}
\usepackage{hyperref}
\usepackage{color}
\usepackage{multirow}
\usepackage{extarrows}
\usepackage{amscd}
\usepackage{scalerel}
\usepackage{stackengine}
\usepackage{bbm}
\usepackage{mathtools}
\usepackage{enumerate}
\usepackage{cleveref}
\usepackage{cite}
\usepackage{pictexwd,dcpic}
\usepackage{tikz} 
\usepackage{comment}

\usepackage[titletoc]{appendix}

\numberwithin{equation}{section}
\newtheorem{theorem}{Theorem}[section]

\newtheorem{conjecture}[theorem]{Conjecture}
\newtheorem{corollary}[theorem]{Corollary}
\newtheorem{lemma}[theorem]{Lemma}
\newtheorem{proposition}[theorem]{Proposition}

\newtheorem{assumption}[theorem]{Assumption}

\theoremstyle{definition}

\newtheorem{remark}[theorem]{Remark}

\newcommand{\ra}{\rightarrow}
\newcommand{\lra}{\longrightarrow}

\def\AAA{\mathbb{A}}

\def\CC{\mathbb{C}}

\def\FF{\mathbb{F}}

\def\QQ{\mathbb{Q}}
\def\RR{\mathbb{R}}

\def\ZZ{\mathbb{Z}}

\def\calA{\mathcal{A}}

\def\calD{\mathcal{D}}

\def\calH{\mathcal{H}}

\def\calJ{\mathcal{J}}

\def\calL{\mathcal{L}}

\def\calO{\mathcal{O}}

\def\calS{\mathcal{S}}
\def\calT{\mathcal{T}}

\def\calV{\mathcal{V}}

\def\gotha{\mathfrak{a}}

\def\scrV{\mathscr{V}}
\def\scrW{\mathscr{W}}

\def\wt{\widetilde}
\def\wh{\widehat}
\def\Lgp{\prescript{L}{}} %Langlands dual group
\def\alg_k{\AAA\mathbbm{l}\mathbbm{g}_{/k}}

\DeclareMathOperator{\Gal}{Gal}

\DeclareMathOperator{\im}{Im}

\DeclareMathOperator{\Ad}{Ad}
\DeclareMathOperator{\Tr}{Tr}% trace
\DeclareMathOperator{\Nm}{Nm}% norm

\DeclareMathOperator{\Hom}{Hom}

\DeclareMathOperator{\Frob}{Frob}

\DeclareMathOperator{\Ind}{Ind}

\DeclareMathOperator{\GL}{GL}

\DeclareMathOperator{\SL}{SL}
\DeclareMathOperator{\SO}{SO}

\DeclareMathOperator{\Irr}{Irr}

\DeclareMathOperator{\disc}{disc}
\DeclareMathOperator{\cusp}{cusp}
\DeclareMathOperator{\sgn}{sgn}
\DeclareMathOperator{\cts}{cts}
\DeclareMathOperator{\GSp}{GSp}

\DeclareMathOperator{\Mp}{Mp}
\DeclareMathOperator{\U}{U}
\DeclareMathOperator{\As}{As}

\newcommand{\BIGboxplus}{\mathop{\mathchoice%
{\raise-0.35em\hbox{\huge $\boxplus$}}%
{\raise-0.15em\hbox{\Large $\boxplus$}}{\hbox{\large $\boxtimes$}}{\boxtimes}}}

\newcommand{\BIGboxtimes}{\mathop{\mathchoice%
{\raise-0.35em\hbox{\huge $\boxtimes$}}%
{\raise-0.15em\hbox{\Large $\boxtimes$}}{\hbox{\large $\boxtimes$}}{\boxtimes}}}

\newcommand{\trivial}[2][]{\if\relax\detokenize{#1}\relax
{\color{red} \vspace{0em} {[} #2 {]}}
\else 
\ifx#1h
\ifcsname showtrivial\endcsname 
{\color{orange} \vspace{0em} {[}  #2 {]}}
\fi
\else {\red Wrong argument!} \fi
\fi
}

\begin{document}

\title[Arthur's Multiplicity Formula for $\rm O_{2n}$ and $\U_{n}$]
{Arthur's Multiplicity Formula for Even Orthogonal and Unitary Groups}

\author{Rui Chen}
\address{Institute for Advanced Study in Mathematics, Zhejiang University, China}
\email{rchenmat@zju.edu.cn}

\author{Jialiang Zou}
\address{
Department of Mathematical, University of Michigan, Ann Arbor, U.S.  
}
%\email{zoujialiang@u.nus.edu}
\email{jlzou@umich.edu}

%\tableofcontents
%This version is without the Appendix A part. 
\begin{abstract}
Let $F$ be a number field and $G$ an even orthogonal or unitary group over a number field. Based on the same method used in \cite{MR3866889}, we prove Arthur's multiplicity formula for the generic part of the automorphic discrete spectrum of $G$ by using the theta lift. Enhancing this method, we also obtain a description of the full automorphic discrete spectrum of $G$ of $F$-rank less than or equal to one.

%We also consider a class of non-generic $A$-parameters and obtain a multiplicity formula in this case. In particular, we obtain a description of the full automorphic discrete spectrum of even orthogonal or unitary groups with Witt index less or equal to one. 
\end{abstract}

\maketitle

\section{Introduction}
Let $F$ be a number field, $\AAA$ the adele ring of $F$, and $G$ a reductive group over $F$. A central question in representation theory is to determine completely the spectral decomposition of $L^2(G(F)\backslash G(\AAA))$ as a unitary representation of $G(\AAA)$. By some results in number theory and functional analysis, we have a decomposition
\begin{equation*}
L^2(G(F)\backslash G(\AAA))=L^2_{\disc}(G)\oplus L^2_{\cts}(G). 
\end{equation*}
Here $L^2_{\disc}(G)$ is called the ``discrete spectrum'', which is called because it decomposes discretely, and $L^2_{\cts}(G)$ is called the ``continuous spectrum''. Furthermore, by Langlands \cite{MR0419366}, the continuous spectrum of $G$ can be described in terms of the discrete spectrum of Levi subgroups of $G$ using Eisenstein series. Therefore, the question is reduced to studying the discrete spectrum $L^2_{\disc}(G)$. In his monumental book \cite{MR3135650}, Arthur obtained a description of the discrete spectrum $L^2_{\disc}(G)$ for quasi-split special orthogonal and symplectic groups $G$. Roughly speaking, the classification can be divided into two steps:

$\bullet$ STEP I: %Assignment of $A$-parameters. 
Decompose $L^2_{\disc}(G)$ into a direct sum of so-called ``near equivalence classes'' (``NEC'' for short), and show that each NEC can be represented by an elliptic $A$-parameter (in the sense of weak transfer to certain general linear group). For an elliptic $A$-parameter $\psi$ of $G$, we denote by $L^2_\psi(G)$ the summand of $L^2_{\disc}(G)$ represented by $\psi$.

$\bullet$ STEP II: Establish the Arthur's multiplicity formula (``AMF'' for short), which gives a further decomposition of $L^2_\psi(G)$ for each $\psi$.
%For each local place $v$ of $F$, define the local $A$-packet $\Pi_{\psi_v}(G_v)$. Then use the local global compatibility to parametrize $L^2_\psi(G)$. 

Following Arthur's work, many works have been done. The following is a list of them. 
\begin{enumerate}
	\item In \cite{MR3338302}, Mok established the AMF for quasi-split unitary groups;
	\item In \cite{kaletha2014endoscopic}, Kaletha-M\'{\i}nguez-Shin-White studied the case of inner forms of unitary groups, and obtained the AMF for the generic part of the automorphic discrete spectrum of these groups;
	\item In \cite{MR3991897}, Gee-Ta\"ibi proved the AMF for $\GSp_4$; 
	\item In \cite{MR3908767}, Ta\"ibi also studied certain inner forms of classical groups. For such a group $G$, he proved the AMF for automorphic representations of $G$ with algebraic regular infinitestimal character at Archimedean places;
	\item In \cite{MR3747484} \cite{xu2021global}, B. Xu established the AMF for the generic part of the automorphic discrete spectrum of quasi-split similitude symplectic and similitude even orthogonal groups.
 \item In \cite{ishimoto2024endoscopic}, Ishimoto established the AMF for the generic part of the automorphic discrete spectrum of non-quasi-split odd special orthogonal groups.
\end{enumerate}
All these works use the stable trace formula as the main tool. 

On the other hand, the theta correspondence provides us a way to transfer the AMF from one group to another group. Let $(G,H)$ be a reductive dual pair over $F$ such that $(G,H)$ is in the stable range and $G$ is the smaller group. For an unitary representation $\pi$ of $G(\AAA)$, we denote by $m_{\disc}(\pi)$ the multiplicity of $\pi$ in $\calA^2(G)$. Likewise, for an unitary representation $\sigma$ of $H(\AAA)$, we can define $m_{\disc}(\sigma)$. In \cite{MR1448215}, J-S. Li proved that 
\[
	m_{\disc}(\pi) \leq m_{\disc}\left(\theta^{abs}(\pi)\right),
\]
where $\theta^{abs}(\pi)$ is the (place-wise) theta lift of $\pi$. In \cite[Sect.4]{MR3866889}, Gan-Ichino observed that when $\pi$ belongs to a generic NEC, J-S. Li's inequality is indeed an equality. Based on this observation, they established the AMF for the generic part of $L^2_{\disc}(\Mp_{2n})$ by transferring it from $L^2_{\disc}(\SO_{2r+1})$ with $r$ sufficiently large. Using the same idea, but combining with some other results, they are also able to describe the full automorphic discrete spectrum of $\Mp_4$, which is carried out in \cite{MR4206596}. 

In this paper we follow Gan-Ichino's method and try to deduce some results on the AMF for even orthogonal or unitary groups. Our goal is two-fold. First, we prove the AMF for the generic part of $L^2_{\disc}(G)$ when $G$ is an even orthogonal or unitary group by transferring it from $L^2_{\disc}(H)$, where $H$ is a large symplectic or quasi-split unitary group. This part is almost parallel to \cite{MR3866889}. Second, we enhance this method using a new observation and deduce a description of the full automorphic discrete spectrum of even orthogonal or unitary groups of $F$-rank less than or equal to one. A case worth noting is when $G$ is an unitary group and $G_v\simeq \U_{1,n-1}$ at one real place $v$. In this case, the description of $L^2_{\disc}(G)$ might have some arithmetic applications to Shimura varieties of type $\U_{1,n-1}$. 

Since some of the results in this article have already been proved elsewhere and we also rely on part of these results, here we list the results we used in this work and compare the results we obtained with others.
\begin{itemize}
\item   Our work relies on Arthur \cite{MR3135650} and Mok \cite{MR3338302} for the description of the discrete spectrum \(L^2_{\disc}(G)\) for quasi-split symplectic and unitary groups. 
However, certain key statements such as the twisted weighted fundamental lemma and the local intertwining relation (``LIR'' for short) are not proved in \cite{MR3135650} and \cite{MR3338302}. The LIR are now filled by the joint work of Atobe, Gan, Ichino, Kaletha, M{\'\i}nguez and Shin \cite{atobe2024local}. 
\item  Our results for unitary groups are independent of the work in \cite{kaletha2014endoscopic}. We obtained the same results as \cite{kaletha2014endoscopic} on the AMF for the generic part of the automorphic discrete spectrum of pure inner forms of unitary groups. However, we also obtained an AMF for the full automorphic discrete spectrum of unitary groups with \(F\)-rank less than or equal to one, which is not covered in \cite{kaletha2014endoscopic}.
\item Our results are also independent of Ta\"ibi's \cite{MR3908767}, and they partially overlap with the results in \cite{MR3908767}.
\end{itemize}

We now give a summary of each section. In the first part of this paper (Section 2 -- 6), we follow the structure of \cite{MR3866889}.
\begin{itemize}
	\item In Section 2, we first recall some basic notions and results for even orthogonal and unitary groups, and then formulate two of our main theorems. The first (Theorem \ref{ThmA}) concerns the existence of weak transfers to certain general linear groups; it is in the full generality and does not have any restrictive conditions. The second (Theorem \ref{ThmB}) is the AMF for the generic part of the automorphic discrete spectrum of even orthogonal groups or unitary groups. For the even orthogonal group case, this result is new; for the unitary group, this provide an alternative approach to the result of \cite{kaletha2014endoscopic}.
\item In Section 3, after recalling some basic notions on the theta correspondence, we review the endoscopic classification for quasi-split groups, which will serve as the input of our later proofs. We also highlight several works due to Howe and J-S. Li. It is their works suggest the possibility of transfering the AMF. 
\item In Section 4, we prove Theorem \ref{ThmA}. The proof is actually the same as \cite[Thm.A]{MR3866889}: combining the local unramified calculations and some results on the partial $L$-functions, one can show that the (abstract) theta lift realizes certain functoriality. Then the desired theorem for our target group $G$ follows from the same result for the auxiliary quasi-split group $H$.
\item Section 5 is the core of this paper. In this section, we first recall Gan-Ichino's observation, and then illustrate how to use their key equality to transfer the AMF from the auxiliary quasi-split group $H$ to $G$. Although Gan-Ichino's observation is only for generic NEC, we shall work in a general setting. Let $\psi$ be an elliptic $A$-parameter for $G$. Locally, we define certain packets $\Pi_{\psi_v}^\theta(G_v)$ of $G$ at each place $v$ of $F$ using the theta lift between $(G,H)$. Globally, we ``glue'' these local packets using the canonical sign character $\epsilon_\psi$ defined by Arthur (Section 5.3), and define the global packet $\Pi_{\psi}^\theta(G,\epsilon_\psi)$. These local and global packets are possible candidates for (conjectural) $A$-packets for our target group $G$, and we would like to call them ``theta packets'' to distinguish various notions of packets. By the key equality of Gan-Ichino, when $\psi=\phi$ is generic, $L^2_\phi(G)$ will decompose according to the associated global theta packet $\Pi_{\phi}^\theta(G,\epsilon_\phi)$ (Proposition 5.6). 
\item Section 6 is the most technical section in this paper. In this section, we carry out the study of those theta packets when $\psi=\phi$ is generic. In this case, we proved local theta packets $\Pi^\theta_{\phi_v}(G_v)$ equal to local $L$-packets $\Pi^L_{\phi_v}(G_v)$. The proof of this local comparison result is also the main difference from \cite{MR3866889}. In Gan-Ichino's paper, they proved similar local comparison results mainly using the global method. Since they assumed the AMF for both quasi-split and non quasi-split special odd orthogonal groups, they can use theta lifts to two Witt towers to obtain the desired information. In our case, we only assume the AMF holds for symplectic groups and quasi-split unitary groups, so essentially we only have one Witt tower to do theta lifts and can not simply apply Gan-Ichino's arguments. We overcome this difficulty by combining local and global arguments. Firstly we use Prasad's conjecture and the induction principle to prove the comparison result for a large class of parameters and representations. With these special cases at hand, we then appeal to the global method to prove the comparison result for the remaining cases.
\end{itemize}
At this point, we have already complete the proofs of two theorems stated in Section 2. The second part of this paper (Section 7) is devoted to generalizing some results in the non-generic case. This part is short but perhaps more novel than the first part.
\begin{itemize}
	\item In Section 7 we try to enhance Gan-Ichino's method by using some other observations. As suggested above, as long as one has the multiplicity preservation
	\[
	m_{\disc}(\pi) = m_{\disc}\left(\theta^{abs}(\pi)\right),
	\]
	then one can play the same game as Section 5.3. We venture to conjecture that all representations in global theta packets only have square-integrable automorphic realizations (Conjecture 7.1). A trivial but worth noting observation is that, when $G$ is anisotropic, all automorphic forms on $G$ are automatically cuspidal. This implies that our conjecture hold in this case. Using the square-integrability criterion, we also prove the conjecture when $G$ is of $F$-rank one. With this conjecture at hand, the desired multiplicity preservation is then granted. Thus, we get a description of the full automorphic discrete spectrum of $G$ of $F$-rank less than or equal to one (Theorem \ref{Thm-C}).\\
\end{itemize} 
In Appendix A, we define an LLC for real full even orthogonal groups and discuss the so-called Prasad's conjecture in this case. In Appendix B, we prove some result on the irreducibility of certain parabolic induction. In Appendix C and D, we prove some results on irreducible self-dual or conjugate self-dual Galois representations. Finally in Appendix E, we discuss the existence of certain number fields. These appendices will serve as supplementary results for the main part of the paper.

We end up this introduction with a remark on a companion work. A disadvantage of Gan-Ichino's method is that, a priori, the AMF for $G$ we get might depend on the auxiliary group $H$. One would also like to prove that the AMF we get is actually independent of the choice of $H$. By the local-global structure of the AMF, the problem can be reduced to studying local theta packets $\Pi_{\psi_v}^\theta(G_v)$ of $G$ at each place $v$. It is predicted by Adams' conjecture \cite[Sect.4]{MR1021501} that these theta packets $\Pi_{\psi_v}^\theta(G_v)$ should be exactly the (conjectural) local $A$-packets $\Pi_{\psi_v}^A(G_v)$ of $G_v$. In \cite{chen2021theta}, we will prove that these packets $\Pi_{\psi_v}^\theta(G_v)$ are independent of the choice of $H$, and compare them with local $A$-packets $\Pi_{\psi_v}^A(G_v)$ when $G_v$ is quasi-split. Indeed, a large part of these have been studied by M{\oe}glin \cite{MR2906916}, see Remark 7.9.

\section*{Acknowledgments}
We would like to thank our supervisor Wee Teck Gan for many useful advice. We also thank Hiraku Atobe, Atsushi Ichino, Wen-Wei Li, and Sug Woo Shin for helpful conversations during the conference ``Workshop on Shimura varieties, representation theory and related topics, 2019'' in Hokkaido University. We thank Caihua Luo, Jiajun Ma, Xiaolei Wan, Chuijia Wang, and Liyang Yang for helpful discussions. We especially thank Hirotaka Kakuhama for several email correspondences on the real Prasad's conjecture. 

\section{Statement of results in generic case}\label{state.main.results}
In this section, we formulate two of our main results. Let $F$ be a local or global field of charactersitic zero, and $E$ either $F$ itself or a quadratic field extension of $F$. Let
\[
c=\begin{cases}
	\textit{the identity of }F\quad &\textit{if }E=F;\\
	\textit{the non-trivial element in }\Gal(E/F)\quad &\textit{if }[E:F]=2.
\end{cases}
\]
For convenience, we denote by $C_F$ (resp. $C_E$) the multiplicative group $F^\times$ (resp. $E^\times$) or $F^\times\backslash\AAA^\times$ (resp. $E^\times\backslash\AAA_E^\times$), depending on the case that $F$ is local or global. In the case $[E:F]=2$, we denote by $\omega_{E/F}$ the quadratic character of $C_F$ by class field theory. Let $V=V_{(n)}$ be a finite dimensional vector space over $E$ equipped with a non-degenerate Hermitian $c$-sesquilinear form
\[
\langle\cdot,\cdot\rangle_V:V\times V\lra E.
\]
We consider the following three cases:
\begin{equation*}%\label{dimV}
	\begin{cases}
		\textit{Case ${\rm O}$: } &\textit{$E=F$ and $\dim V=2n$};\\
		\textit{Case $\U_0$: } &\textit{$[E:F]=2$ and $\dim V=2n$};\\
		\textit{Case $\U_1$: } &\textit{$[E:F]=2$ and $\dim V=2n+1$}.\\
	\end{cases}
\end{equation*}
where $n\geq 0$ is an integer. Sometimes, when we want to deal with Case $\U_0$ and Case $\U_1$ at the same time, we shall simply write ``Case U''. Let $G=G(V)$ be the group of elements $g$ in $\GL(V)$ such that
\[
\langle gv,gw\rangle_V=\langle v,w\rangle_V\quad\textit{for }v,w\in V.
\]
If $\dim V=0$, we interpret $G=G(V)$ and its pure inner form as the trivial group. Now assume that $\dim V>0$. Let $\disc V=(-1)^n \cdot \det(V)$ be the discriminant of $V$. In Case O, we let
\begin{equation}\label{chiV}
\chi_V:C_F\lra\CC^\times
\end{equation}
be the quadratic character associated to $\disc V$ by class field theory, and $\epsilon(V)$ be the (normalized) Hasse-Witt invariant of $V$ (cf. \cite[p.80-81]{MR770063}). In Case U, we define the sign $\epsilon(V)=\omega_{E/F}(\disc V)$.

Sometimes we also need to consider pure inner forms of $G=G(V)$. It is well known that all pure inner forms of $G$ arise in the forms $G'=G(V')$ for some spaces $V'$. We briefly describe the classification of $V'$ in both local and global situations. If $F$ is a local field, then all these spaces $V'$ are classified by some invariants.

\noindent\underline{\textit{When $F$ is non-Archimedean:}}
\begin{itemize}
	\item In Case O, these $V'$ are orthogonal spaces of the same dimension and discriminant as $V$. There are at most two such spaces, distinguished by their (normalized) Hasse-Witt invariant $\epsilon(V)$. We shall denote by $V^+$ the one with Hasse-Witt invariant $+1$ (which always exists), and by $V^-$ the one with Hasse-Witt invariant $-1$ (which exists unless $n=1$ and $\chi_V$ is trivial). Since $V^+$ has the maximal possible Witt index, $V^+$ must be isometric to
\begin{equation*}%\label{V^+}
  V^+\simeq V_{(d,c)}+\calH^{n-1}
\end{equation*}
for some $d,c\in F^\times$, where 
\[
  V_{(d,c)}=F[X]/(X^2-d)
\]
is an $2$-dimensional vector space over $F$ equipped with the quadratic form
\[
  a+bX\longmapsto c\cdot(a^2-b^2d),
\]
and $\calH$ is the (orthogonal) hyperbolic plane. We fix such a tuple $(d,c)$ and the isometry, and we shall say that $V^+$ is of type $(d,c)$. Notice that the choice of the tuple $(d,c)$ is not unique. 

	\item In Case U, these $V'$ are Hermitian spaces of the same dimension as $V$. There are exactly two such spaces, distinguished by their sign $\epsilon(V)=\omega_{E/F}(\disc V)$. we shall denote by $V^+$ the one of sign $+1$, and by $V^-$ the one of sign $-1$.
\end{itemize}
\underline{\textit{When $F$ is real:}}
\begin{itemize}
	\item In Case O, the space $V$ is determined by its signature $(p,q)$. In this case, these spaces $V'$ are classified by their signatures $(p',q')$ such that 
	\[
	p'+q'=2n \quad \textit{and} \quad p'\equiv p \bmod 2.
	\]
	We shall denote by $V^+$ the one with Hasse-Witt invariant $+1$ and has maximal Witt index.

	\item In Case U, these $V'$ are classified by their signatures $(p',q')$ satisfying $p'+q'=\dim V$. We shall denote by $V^+$ the one of sign $+1$ and has maximal Witt index. 
\end{itemize}
\underline{\textit{When $F$ is complex:}}
\begin{itemize}
	\item There is only one such space up to isometry with given dimension, and we shall denote it by $V^+$.
\end{itemize}
With these local classifications at hand, we can now describe the classification $V'$ when $F$ is a global field.
\begin{itemize}
	\item In Case O, these $V'$ are orthogonal spaces of the same dimension and discriminant as $V$. Let $d=\disc(V)$. The local-global principle for orthogonal spaces (cf. \cite[p.225 Thm.6.10]{MR770063}) implies that, whenever we are given a collection of local orthogonal spaces $\{V'_v\}_v$ over $F_v$ for all places $v$ of $F$, such that $\dim V'_v=2n$, $\disc(V'_v)=d_v$, $\epsilon (V'_v)=1$ for almost all places and
	\begin{align*}%\label{local-global principle for quadratic spaces}
		\prod_v \epsilon (V'_v)=1,
	\end{align*}
	there exists a global orthogonal space $V'$ with these localizations. Moreover, these spaces $V'$ are classified by such coherent data $\{V'_v\}_v$.

	\item In Case U, these $V'$ are Hermitian spaces of the same dimension as $V$. The local-global principle for Hermitian spaces (cf. \cite[p.377 Thm.6.9]{MR770063}) implies that, whenever we are given a collection of local  Hermitian spaces $\{V'_v\}_v$ over $E_v$ for all places $v$ of $F$, such that $\dim V'_v=\dim V$, $\epsilon (V'_v)=1$ for almost all places and
	\begin{align*}%\label{local-global principle for Hermitian spaces}
		\prod_v \epsilon (V'_v)=1,
	\end{align*}
	there exists a global Hermitian space $V'$ with these localizations. Moreover, these spaces $V'$ are classified by such coherent data $\{V'_v\}_v$. 
\end{itemize}
Given $V$ and $G=G(V)$, we let $V^+$ be the space such that for each place $v$ of $F$, $V_v^+$ is (isometry to) the space we have defined in local situations, i.e. $(V^+)_v\simeq (V_v)^+$. In all the cases above, $G^*=G(V^+)$ is quasi-split, and we refer to it as the quasi-split pure inner form of $G$.

\subsection{Near equivalence classes and Arthur parameters}\label{NEC&APara}
Let $F$ be a number field. We first describe the decomposition of $L^2_{\disc}(G)$ into near equivalence classes of representations. We say two irreducible representations $\pi=\otimes_v\pi_v$ and $\pi'=\otimes_v\pi'_v$ of $G(\AAA)$ are nearly equivalent if $\pi_v$ and $\pi'_v$ are equivalent for almost all places $v$ of $F$. The decomposition into near equivalence classes will be expressed in terms of elliptic $A$-parameters. Recall that a (global) $A$-parameter for $G$ is nothing but a formal finite sum
\begin{equation}\label{paradecomp}
	\psi=\sum_i\phi_i\boxtimes S_{d_i},
\end{equation}
where
\begin{itemize}
	\item $\phi_i$ is an irreducible (conjugate) self-dual cuspidal automorphic representation of $\GL_{n_i}(\AAA_E)$;
	\item $S_{d_i}$ is the $d_i$-dimensional irreducible representation of $\SL_2(\CC)$;
	\item $\sum_i n_id_i=\dim V$;
	\item If $d_i$ is odd, then $\phi_i$ is 
	\[
	\begin{cases}
		\textit{orthogonal}\quad &\textit{Case ${\rm O}$};\\
		\textit{conjugate symplectic }\quad &\textit{Case $\U_0$};\\
		\textit{conjugate orthogonal }\quad &\textit{Case $\U_1$}.
	\end{cases}
	\]
	\item If $d_i$ is even, then $\phi_i$ is 
	\[
	\begin{cases}
		\textit{symplectic}\quad &\textit{Case ${\rm O}$};\\
		\textit{conjugate orthogonal }\quad &\textit{Case $\U_0$};\\
		\textit{conjugate symplectic }\quad &\textit{Case $\U_1$}.
	\end{cases}
	\]
	\item In Case O, if we denote the central character of $\phi_i$ by $\omega_i$, then
	\[
	\prod_i\omega_i^{d_i}=\chi_V.
	\]
\end{itemize}
Following Arthur, we make the following definitions:
\begin{itemize}
	\item If there is only one term in the summation $(\ref{paradecomp})$, i.e. $\psi=\phi_1\boxtimes S_{d_1}$, then we say that $\psi$ is simple;
	\item If we have $(\phi_i,d_i)\neq(\phi_j,d_j)$ for all $i\neq j$, then we say that $\psi$ is elliptic;
	\item If we have $d_i=1$ for all $i$, then we say that $\psi$ is generic (or tempered).
\end{itemize}	
Note that an (elliptic) $A$-parameter for $G$ is also an (elliptic) $A$-parameter for $G^*$, and vice versa. We denote the set of all elliptic $A$-parameters of $G$ by $\Psi_{ell}(G^*)$. One can formally associate to $\psi$ a free $\mathbb Z/2\mathbb Z$-module 
\begin{equation*}\label{Glcom}
	\calS_\psi=\prod_i (\mathbb Z/2\mathbb Z)e_i 
\end{equation*}
with a canonical basis $\{e_i\}_i$, where each $e_i$ corresponds to the summand $\phi_i\boxtimes S_{d_i}$. We shall call $\calS_{\psi}$ the global component group of $\psi$. 

For each place $v$ of $F$, One can also define local $A$-parameters and local component groups for $G_v$. First suppose that we are either in Case O, or in Case U and $v$ is not split in $E$. Then there is a unique place of $E$ above $v$, which we shall still denote by $v$. Let
\begin{equation*}
	L_{E_v}=\begin{cases}
		\textit{the Weil group of }E_v\quad&\textit{if }v\textit{ is Archimedean};\\
		\textit{the Weil-Deligne group of }E_v\quad&\textit{if }v\textit{ is non-Archimedean}.\\
	\end{cases}
\end{equation*}
A local $A$-parameter for $G_v$ is a representation
\begin{equation}\label{E:localAPara}
	\psi_v=\sum_i m_i\cdot\phi_{i,v}\boxtimes S_{d_i}
\end{equation}
of $L_{E_v}\times \SL_2(\CC)$, where
\begin{itemize}
	\item $\phi_{i,v}\boxtimes S_{d_i}$ are pairwise inequivalent irreducible (conjugate) self-dual representations of $L_{E_v}\times\SL_2(\CC)$ with multiplicities $m_i$;
	\item As a representation of $L_{E_v}\times \SL_2(\CC)$, $\psi_v$ is  
\[
\begin{cases}
	\textit{orthogonal} \quad &\textit{Case ${\rm O}$};\\
	\textit{conjugate symplectic }\quad &\textit{Case $\U_0$};\\
	\textit{conjugate orthogonal }\quad &\textit{Case $\U_1$}.
\end{cases}
\]  
	\item In Case O, $\det(\psi_v)=\chi_{V,v}$.
\end{itemize}
By \cite[Sect.8]{MR3202556}, the component group $\calS_{\psi_v}$ has an explicit description of the form 
\[
\calS_{\psi_v}=\prod_j (\mathbb Z/2\mathbb Z) e_j
\]
%\zjl{Here I change the notation $a_j$ to $e_j$. Please check if I made all the changes accordingly below}
with a canonical basis $\{e_j\}$, where the product ranges over all $j$ so that $\phi_{j,v}\boxtimes S_{d_j}$ is (conjugate) self-dual of the same parity as $\psi_v$. For $e=e_{j_1}+\cdots+e_{j_r}\in \calS_{\psi}$, we put 
\[
\psi_v^e=\phi_{j_1,v}\boxtimes S_{d_{j_1}}+ \cdots + \phi_{j_r,v}\boxtimes S_{d_{j_r}}.
\]
Next suppose that we are in Case U, and $v$ is split into two places $\{w,\overline{w}\}$ in $E$. In this case the local $A$-parameter for $G_v$ can be similarly defined as a formal sum as in (\ref{E:localAPara}), but now each $\phi_{i,v}$ is an irreducible conjugate self-dual representation of some $\GL_{n_i}(E_v)\simeq \GL_{n_i}(E_w)\times \GL_{n_i}(E_{\overline{w}})$. Indeed, if we identify $F_v\simeq E_w\simeq E_{\overline{w}}$, then the conjugate self-duality of each $\phi_{i,v}$ will imply that
\[
    \phi_{i,v} \simeq \phi_{i,w}\boxtimes \phi_{i,w}^\vee
\]
for some irreducible representation $\phi_{i,w}$ of $\GL_{n_i}(E_w)$. Regarding these $\phi_{i,w}$ as representations of $L_{E_w}$ by using the local Langlands correspondence for general linear groups (see \cite{MR1011897},\cite{MR1876802},\cite{MR1738446},\cite{MR3049932}), we get a local $A$-parameter for $G_v\simeq \GL(V_w)$ 
\[
    \psi_w = \sum_i m_i\cdot\phi_{i,w}\boxtimes S_{d_i}
\]
in the usual sense. In this case the component group $\calS_{\psi_v}$ is trivial. In both cases, if further that $d_i=1$ for all $i$, i.e. the restriction of $\psi_v$ to $\SL_2(\CC)$ is trivial, then we say that $\psi_v$ is an $L$-parameter for $G_v$. Again following Arthur, we shall use $\Psi(G_v)$ (resp. $\Phi(G_v)$) to denote the set of $A$- (resp. $L$-)parameters of $G_v$ with bounded image on the Weil group, and also $\Psi^+(G_v)$ (resp. $\Phi^+(G_v)$) the set of $A$- (resp. $L$-)parameters of $G_v$.

Now given an elliptic $A$-parameter $\psi=\sum_i \phi_i\boxtimes S_{d_i}$ of $G$, let $\psi_v=\sum_i \phi_{i,v}\boxtimes S_{d_i}$ be the localization of $\psi$ at $v$. Here each $\phi_{i,v}$ is an irreducible representation of $\GL_{n_i}(E_v)$, and we shall also regard it as an $L$-parameter via the LLC for general linear groups. Then by definitions, $\psi_v$ gives rise to a $A$-parameter for $G_v$. We associate to $\psi_v$ an $L$-parameter $\phi_{\psi_v}$ by the formula
\begin{equation}\label{AparameterassLpara}
	\phi_{\psi_v}(w)=\psi_v\left(w,\left(\begin{array}{cc}
		{|w|^\frac{1}{2}} & {} \\
		{} & {|w|^{-\frac{1}{2}}}
	\end{array}\right)\right).
\end{equation}
Our first theorem shows that each NEC has a weak transfer to certain general linear group. 
\begin{theorem}\label{ThmA}
	There exists a decomposition
	\begin{equation*}
		L^2_{\disc}(G)=\bigoplus_{\psi\in\Psi_{ell}(G^*)} L^2_\psi(G),
	\end{equation*}
	where $L^2_\psi(G)$ is a full near equivalence class of irreducible representations $\pi$ in $L^2_{\disc}(G)$ such that the $L$-parameter of $\pi_v$ is $\phi_{\psi_v}$ for almost all places $v$ of $F$.
\end{theorem}
Note that at this point, $L^2_\psi(G)$ could be zero for some $\psi\in\Psi_{ell}(G^*)$.

\subsection{Local Langlands correspondence}\label{SS:LLC}
Our next goal is to describe the near equivalence class $L^2_\psi(G)$ when $\psi$ is generic. For this purpose, we need to make use of the (Vogan version) local Langlands correspondence (``LLC'' for short), which provides us a bijection between irreducible representations of $G_v$ and enhanced $L$-parameters. There are some existing results. 

\noindent\underline{\textit{When $v$ is Archimedean:}}
the LLC was proved in \cite{MR1011897} for connected reductive groups. However since we need to deal with the disconnected group ${\rm O}_{2n}$ in Case O, we need to extend the LLC for special even orthogonal groups to full even orthogonal groups. When $F$ is complex, there is only one orthogonal space of dimension $2n$, and the corresponding even orthogonal group is quasi-split. In this case such an LLC is already provided by Arthur's results, see the remark below. When $F$ is real, we provide an LLC for full even orthogonal groups in Appendix \ref{App.LLC4O} using theta lifts.

\noindent\underline{\textit{When $v$ is non-Archimedean:}} 
the LLC was proved in \cite{MR3135650} \cite{MR3708200} in the case that $G$ is a quasi-split even orthogonal group; in \cite{MR3338302} in the case that $G$ is a quasi-split unitary group; and in \cite{kaletha2014endoscopic} in the case that $G$ is an inner form of an unitary group. Also, in our previous papers \cite{zou2020local, chen2020local}, we established the LLC for pure inner forms of even orthogonal groups and unitary groups over non-Archimedean fields using theta lifts. 

\begin{remark}
%\begin{enumerate}
Indeed, Arthur also established a weak version of the LLC for quasi-split special even orthogonal groups over Archimedean fields in \cite{MR3135650}. According to \cite[Thm.3.10]{MR3708200}, Arthur's results also implicitly imply an LLC for quasi-split full even orthogonal groups, though he didn't highlight it. We can show that our extension of the LLC for real full even orthogonal groups coincide with Arthur's when the group is quasi-split by appealing to the global method. We sketch the proof at the end of Appendix \ref{App.LLC4O}. 
%By mimicking the argument of \cite[Prop.4.10]{MR3708200}, we can show that our extension of the LLC for full even orthogonal groups coincide with Arthur's when the group is quasi-split.
%	\item Recently, Kaletha \cite{kaletha2022local} has proposed a conjectural LLC for certain disconnected groups. In an ongoing work with Paul Mezo, they are trying to established the LLC for disconnected real groups. Unfortunately we don't know whether our extension of the LLC for full even orthogonal groups coincide with theirs. \zjl{discuss here whether we need this}
 %classical groups by computing some endoscopy transfer factors.  
%\end{enumerate}
\end{remark}

Now we briefly recall above results. Assume $F$ is local for a moment. There is a canonical finite to one surjection
\[
\calL:\bigsqcup_{G'}\Irr \left(G'\right)\lra\Phi^+(G),
\]
where the disjoint union runs over all pure inner forms of $G$. For each $L$-parameter $\phi$, we denote
\[
\Pi_\phi^L(G')=\calL^{-1}(\phi)\cap\Irr \left(G'\right)
\]
and we call it the $L$-packet of $G'$ associated to $\phi$. There is a canonical bijection (depends on the choice of a Whittaker datum $\scrW$ of $G^*$)
\begin{equation}\label{LLCforG}
	\calJ_\scrW^L:\bigsqcup_{G'}\Pi_\phi^L(G')\lra\wh{\calS_\phi},
\end{equation}
where the disjoint union again runs over all pure inner forms of $G$. Furthermore, the bijection $\calJ_\scrW^L$ is compatible with the Kottwitz isomorphism \cite[Thm.1.2]{MR858284}, and this property characterizes the image of $\Pi_\phi^L(G')$ under $\calJ_\scrW^L$. We shall denote by $\pi(\phi,\eta)$ the irreducible representation of some $G'$ with $L$-parameter $\phi$ and corresponding to $\eta$. We may also regard the $L$-packet $\Pi_\phi^L(G')$ as a representation of $\calS_\phi\times G'$ by letting
\[
\Pi_\phi^L(G')=\bigoplus_{\pi}\calJ_\scrW^L(\pi)\boxtimes\pi,
\]
where the summation on the RHS runs over all irreducible representations in $\Pi_\phi^L(G')$. Sometimes we will adapt to this point of view. 
\begin{remark}\label{LLC}
Let $\kappa_\phi$ be the character of $\calS_\phi$ defined by the formula
\begin{equation}\label{determitwis}
	\kappa_\phi(e)=(-1)^{\dim\phi^e}
\end{equation}
for $e\in\calS_\phi$.
	\begin{enumerate}
		\item In Case O, for any irreducible representation $\pi=\pi(\phi,\eta)$ of $G$, we have
		\[
			\pi\otimes\det=\pi(\phi,\eta\cdot\kappa_\phi).
		\]
		When $F$ is non-Archimedean, this property is proved for example in \cite[Thm.4.4]{zou2020local}. When $F$ is real, see Remark \ref{R:CompatibleRealEvenOvsSO}. When $F$ is complex, this follows from the compatibility of the LLC for full even orthogonal groups with the LLC for special even orthogonal groups, see \cite[Desideratum 3.9(8)]{MR3708200}.
		\item In Case $\U_1$, if $F$ is non-Archimedean, we can take $V^-=a\cdot V^+$ for some $a\in F^\times\backslash \Nm_{E/F}(E^\times)$. Then $G(V^+)$ and $G(V^-)$ are physically equal as subgroups of $\GL(V^+)$. For any irreducible representation $\pi=\pi(\phi,\eta)$ of $G(V^+)$, if we consider it as a representation of $G(V^-)$, then we have
		\[
			\pi=\pi(\phi,\eta\cdot\kappa_\phi).
		\]
		Readers can consult \cite[Thm.2.5.5]{chen2020local} for a more detailed discussion on this property.
	\end{enumerate}
\end{remark}

\begin{remark}
In Case U, we simply use the LLC for general linear groups in split places.
\end{remark}

\subsection{Multiplicity formula}
We now assume that $F$ is a number field, and $\psi=\phi$ a generic elliptic $A$-parameter of $G$, i.e. $\phi$ is a multiplicity-free sum  
\begin{equation*}
	\phi=\sum_i \phi_i
\end{equation*}
of irreducible (conjugate) self-dual cuspidal automorphic representation $\phi_i$ of $\GL_{n_i}(\mathbb A_E)$ with appropriate parity. Fix a global Whittaker datum $\scrW$ of $G^*$. At each place $v$ of $F$, we have a localization map
\[
\calS_\phi\lra\calS_{\phi_v}.
\]
We define a global packet
\begin{align*}
	\Pi_\phi(G)&=\otimes'_v\Pi_{\phi_v}^L(G_v)\\
	&=\left\{\pi=\otimes'_v\pi_v~|~\pi_v\in\Pi_{\phi_v}^L(G_v),~\pi_v\textit{ unramified with $L$-parameter $\phi_v$ for almost all }v\right\}.
\end{align*}
We then have a map
\begin{align*}
	\calJ_\scrW:\Pi_\phi(G)&\lra\widehat{\calS_\phi},\\
	\pi&\longmapsto\calJ_\scrW(\pi),\\
	\calJ_\scrW(\pi)(x)&\coloneqq\prod_v\calJ_{\scrW_v}^L(\pi_v)(x_v),
\end{align*}
where $x\in\calS_\phi$ and $x_v$ is the localization of $x$ at $v$.
\begin{remark}
	According to the main local theorems of \cite{MR3135650} and \cite{MR3338302}, if $G_v$ and $\pi_v$ are both unramified, then $\calJ_{\scrW_v}^L(\pi_v)$ is the trivial character $1$. Hence $\calJ_\scrW$ is well-defined.
\end{remark} 
Let $\epsilon_\phi=1$ be the trivial character of $\calS_\phi$. We put
\begin{equation*}
	\Pi_\phi(G,\epsilon_\phi)=\{\pi\in\Pi_\phi(G)~|~\calJ_\scrW(\pi)=\epsilon_\phi\}.
\end{equation*}
Our second theorem is the following.
\begin{theorem}\label{ThmB}
	Let $\phi$ be a generic elliptic $A$-parameter for $G$. Then we have the decomposition
	\begin{equation*}
		L^2_\phi(G)=\bigoplus_{\pi\in\Pi_\phi(G,\epsilon_\phi)}\pi.
	\end{equation*}
	In particular, $L^2_\phi(G)$ is multiplicity free. 
\end{theorem}
\begin{remark}\label{AMFquasi}~
\begin{enumerate}
	\item Suppose that $V=V^+$. Then $G=G^*$ is quasi-split. Hence, by the results of \cite{MR3135650}, \cite{MR3708200} (for Case O), and \cite{MR3338302} (for Case U), Theorem \ref{ThmB} holds for $G$. Our results generalize these works to the case of (not necessarily quasi-split) pure inner forms. 

	\item We should mention that in \cite{MR3135650}, Arthur only formulated and proved his results for quasi-split special even orthogonal groups $\SO(V^+)$. His results do not distinguish between automorphic representations and their twist by the outer automorphism. Therefore, in the AMF for $\SO(V^+)$, some multiplicity two phenomenon occurs. In \cite{MR3708200}, Atobe-Gan formulate the AMF for quasi-split even orthogonal groups $\mathrm{O}(V^+)$ precisely and explicated that Arthur's results in \cite{MR3135650} already implied Theorem \ref{ThmB} for $\mathrm{O}(V^+)$. 
\end{enumerate}
\end{remark}

\subsection{A special case}
We first deal with a special case of Theorem \ref{ThmB}, which will be used in later proofs. In this subsection, suppose that we are in Case $\U_1$, and $F$ is a totally imaginary number field.

In this case,  we take $a\in F^\times$, so that $a$ is in the same $\Nm_{E/F}(E^\times)$-orbit as $\disc(V)$. Then by \cite[Cor.6.6]{MR770063}, we have $V^+ \simeq a\cdot V$. This implies that $G\simeq G^*$ as abstract groups. Therefore
\[
L_{\psi}^2(G)= L_{\psi}^2(G^*)
\] 
for any elliptic $A$-parameter $\psi$ of $G$. However, to establish Theorem \ref{ThmB} for $G$, we need to consider $G$ not only as an abstract group, but also as a pure inner form of $G^*$, i.e. the Hermitian form $V$ should also be taken into consideration. When $\psi=\phi$ is generic, we need to distinguish $\Pi_{\phi_v}^L(G_v)$ and $\Pi_{\phi_v}^L(G_v^*)$ at some places. To be more precise, let $v$ be a place of $F$, there are two cases:
\begin{itemize}
	\item If $a_v\in\Nm_{E_v/F_v}(E_v^\times)$, then $V_v\simeq V_v^+$, and hence $\Pi_{\phi_v}^L(G_v)=\Pi_{\phi_v}^L(G_v^*)$ as representations of $\calS_{\phi_v}\times G_v$. Note that all complex places satisfy this condition.
	\item If $a_v\notin\Nm_{E_v/F_v}(E_v^\times)$, then $V_v\not\simeq V_v^+$. According to Remark \ref{LLC} (2), as representations of $\calS_{\phi_v}\times G_v$, we have
	\[
	\Pi_{\phi_v}^L(G_v)=\Pi_{\phi_v}^L(G_v^*)\otimes\kappa_{\phi_v}.
	\] 
	Note that by the local-global principle for Hermitian forms, the number of places $v$ satisfying this condition is even.
\end{itemize}
Therefore, it is not hard to check that:
\[
\Pi_\phi(G,\epsilon_\phi)=\Pi_\phi(G^*,\epsilon_\phi)
\]
as sets of representations of $G(\AAA)$. We deduce that:
\begin{proposition}\label{AMF-U_1}
	Suppose we are in Case $\U_1$, and $F$ is a totally imaginary number field. Then Theorem \ref{ThmB} holds.
\end{proposition}
This proposition will be used in the later proof of Theorem \ref{ThmB}.

\section{Preliminaries}
In this section, we recall some preliminaries we will need in the proof of our main theorems.

\subsection{Theta lifts}\label{ThetaLifts}
Fix a trace zero element $\delta\in E^\times$. Let $W=W_{(r)}$ be an 
\begin{equation}\label{W-split}
	\begin{cases}
		2r\textit{-dimensional}\quad &\textit{Case ${\rm O}$};\\
		(2r+1)\textit{-dimensional}\quad &\textit{Case $\U_0$};\\
		(2r+2)\textit{-dimensional}\quad &\textit{Case $\U_1$}\\
	\end{cases}
\end{equation}
vector space over $E$ equipped with a non-degenerate skew-Hermitian $c$-sesquilinear form
\[
\langle\cdot,\cdot\rangle_W:W\times W\lra E,
\]
such that $W$ is split (in Case $\U_0$ we require that the anisotropic kernel of $W$ is the $1$-dimensional skew-Hermitian space represented by $\delta$). Let $H=H(W)$ be the group of elements $h$ in $\GL(W)$ which preserve the form
\[
\langle hv,hw\rangle_W=\langle v,w\rangle_W\quad\textit{for }v,w\in W. 
\]
Note that $H$ is quasi-split. The pair $(G,H)$ is then an example of a reductive dual-pair. We fix a pair of characters $(\chi_V,\chi_W)$ of $C_E$ as follows:
\[
\chi_V=\begin{cases}
	\textit{the quadratic character associated to }V\quad &\textit{Case ${\rm O}$};\\
	\textit{a character of }C_E\textit{ such that }\chi_V|_{C_F}=\omega_{E/F}^{\dim V}\quad &\textit{Case ${\rm U}$}.
\end{cases}
\]
\[
\chi_W=\begin{cases}
	\textit{the trivial character of }F^\times\quad &\textit{Case ${\rm O}$};\\
	\textit{a character of }C_E\textit{ such that }\chi_W|_{C_F}=\omega_{E/F}^{\dim W}\quad &\textit{Case ${\rm U}$}.
\end{cases}
\]

Assume $F$ is local for a moment. With respect to a non-trivial additive character $\psi_F$ of $F$ and the auxiliary data $(\chi_V,\chi_W)$, one can define the Weil representation $\omega$ of $G\times H$. For any irreducible representation $\pi$ of $G$, the maximal $\pi$-isotypic quotient of $\omega$ is of the form 
\begin{equation*}
	\pi\boxtimes\Theta(\pi)
\end{equation*}
for some smooth representation $\Theta(\pi)$ of $H$ of finite length. Then by the Howe duality \cite{MR1159105}, \cite{MR3502978}, \cite{MR3454380}, the maximal semi-simple quotient $\theta(\pi)$ of $\Theta(\pi)$ is either zero or irreducible. Similarly, for any irreducible representation $\sigma$ of $H$, we can define $\Theta(\sigma)$ and $\theta(\sigma)$.

Suppose next that $F$ is a number field. Fix a non-trivial additive character $\psi_F$ of $F\backslash\AAA$, and also characters $(\chi_V,\chi_W)$. We define an abstract irreducible representation of $\pi$ of $G(\AAA)$ as a tensor product of irreducible representations $\pi_v$ of $G_v$, which is at almost all places umramified. We write 
$\pi=\otimes_v\pi_v$. At each place $v$ of $F$, we can form the local theta lift $\theta(\pi_v)$ with respect to $(\psi_{F,v},\chi_{V,v},\chi_{W,v})$. Assume that they are all non-vanishing. Then $\theta(\pi_v)$ is irreducible for all $v$ and is unramified for almost all $v$. Hence, we may define an abstract irreducible representation
\begin{equation*}
	\theta^{abs}(\pi)=\bigotimes_v\theta(\pi_v)
\end{equation*}
of $H(\AAA)$. We call $\theta^{abs}(\pi)$ the abstract theta lift of $\pi$ to $H(\AAA)$. On the other hand, if $\pi$ is an irreducible cuspidal automorphic representation of $G(\AAA)$, then we can define its global theta lift $\Theta^{aut}(\pi)$ as the subspace of $\calA(H)$ spanned by all automorphic forms of the form
\begin{equation*}
	\theta(f,\varphi)(h)=\int_{G(F)\backslash G(\AAA)}\theta(f)(g,h)\overline{\varphi(g)}dg
\end{equation*}
for $f\in\omega$ and $\varphi\in\pi$. Here $\omega$ is the Weil representation of $G(\AAA)\times H(\AAA)$ and $\theta(f)$ is the theta function associated to $f$. According to \cite{MR1289491}, if $\Theta^{aut}(\pi)$ is non-zero and contained in $\calA^2(H)$, then $\Theta^{aut}(\pi)$ is irreducible and
\begin{equation*}
	\Theta^{aut}(\pi)\simeq\theta^{abs}(\pi).
\end{equation*}

\subsection{Unitary representations of low rank}
The notion of rank for unitary representations was first introduced by Howe \cite{MR777342} for symplectic groups and was extended to other classical groups by J-S. Li \cite{MR1008803}. Following \cite{MR1008803}, we say that an irreducible unitary representation of $H=H(W_{(r)})$ is of low rank if its rank is less than $r$. Such representations can be described using theta lifts as follows.

Assume $\dim V<r$, so that the reductive dual pair $(G,H)$ is in the stable range (see \cite[Def.5.1]{MR1001840}). If $F$ is local, then for any irreducible representation $\pi$ of $G$, its theta lift $\theta(\pi)$ to $H$ is non-vanishing. Moreover, if $\pi$ is unitary, then by \cite{MR1001840}, so is $\theta(\pi)$. In \cite{MR1008803}, J-S. Li showed that this theta lift provides a bijection
\begin{equation*}
	\begin{array}{c}
		\displaystyle{\bigsqcup_{V'}}\Irr_{unit}G(V')\times\Big\{\textit{Characters of }H\Big\}\\
		\Bigg\updownarrow\\
		\Big\{\textit{Irreducible unitary representations of }H\textit{ of rank }\dim V\Big\}.
	\end{array}
\end{equation*}
where the disjoint union runs over all isometry classes of vector space $V'$ over $E$ of the same dimension as $V$, and equipped with a non-degenerate Hermitian $c$-sesquilinear form. The map sends a pair $(\pi,\chi)$ in the first set to a representation $\theta(\pi)\otimes\chi$ of $H$. Note that in Case O, $\chi$ is always trivial since $H$ is simple.

This result has a global analog. Let $F$ be a number field and $\sigma=\otimes_v\sigma_v$ an irreducible unitary representation of $H(\AAA)$ that occurs as a subrepresentation of $\calA(H)$. Then, by \cite[Lem.3.2]{MR1001840}, the following are equivalent:
\begin{itemize}
	\item $\sigma$ is of rank $\dim V$;
	\item $\sigma_v$ is of rank $\dim V$ for all $v$;
	\item $\sigma_v$ is of rank $\dim V$ for some $v$.
\end{itemize}
Suppose that $\sigma$ satisfies the above equivalent conditions. Then, \cite[Prop.5.7]{MR1448215} asserts that there exists some $G=G(V')$ with $\dim V'=\dim V$, together with an abstract representation $\pi=\otimes_v\pi_v$ of $G(\AAA)$, and an automorphic character $\chi$ of $H(\AAA)$, such that
\begin{equation*}
	\sigma\simeq\theta^{abs}(\pi)\otimes\chi.
\end{equation*}

\subsection{Some inequalities}
Finally, we recall a result of J-S. Li, which allows us to lift square-integrable automorphic representations of $G(\AAA)$ to $H(\AAA)$. For any irreducible representation $\pi$ of $G(\AAA)$, we define its multiplicities $m(\pi)$ and $m_{\disc}(\pi)$ by
\begin{align*}
	m(\pi)&=\dim\Hom_{G(\AAA)}\big(\pi,\calA(G)\big);\\
	m_{\disc}(\pi)&=\dim\Hom_{G(\AAA)}\big(\pi,\calA^2(G)\big).
\end{align*}
Obviously, $m_{\disc}(\pi)\leq m(\pi)$. Similarly, if $\sigma$ is an irreducible representation of $H(\AAA)$, we have its multiplicities $m(\sigma)$ and $m_{\disc}(\sigma)$. By \cite[Theorem A]{MR1448215}, we have:
\begin{theorem}\label{J-S inequality}
	Assume that $\dim V<r$. Let $\pi$ be an irreducible unitary representation of $G(\AAA)$ and $\theta^{abs}(\pi)$ its abstract theta lift to $H(\AAA)$. Then we have
	\begin{equation*}
		m_{\disc}(\pi)\leq m_{\disc}(\theta^{abs}(\pi))\leq m(\theta^{abs}(\pi))\leq m(\pi).
	\end{equation*}
\end{theorem}

\subsection{Review of results for quasi-split groups}
We review some results in \cite{MR3135650} and \cite{MR3338302}.

First, let $F$ be a local field. Recall that by our definition, $H=H(W_{(r)})$ is either a symplectic group or a quasi-split unitary group. Similarly to Section \ref{NEC&APara}, one can also define the local $A$-parameter and the component group for $H$. Following Arthur, we use $\Psi(H)$ to denote the set of local $A$-parameters for $H$ with bounded images on the Weil group. Let $\psi\in\Psi(H)$ be a $A$-parameter for $H$. If we write 
\[
	\psi=\sum_i m_i\cdot\phi_i\boxtimes S_{d_i}
\]
with pairwise inequivalent irreducible subrepresentations $\phi_i\boxtimes S_{d_i}$ of $\psi$, then by \cite[Sect.8]{MR3202556}, we have 
\begin{align*}
	\calS_\psi= \prod_i (\ZZ/2\ZZ) e_j \quad \textit{and} \quad \overline{\calS_\psi}=\calS_\psi/\langle z_\psi\rangle,
\end{align*}
where the product ranges over all $j$ such that $\phi_{j}\boxtimes S_{d_j}$ is (conjugate) self-dual of the same parity as $\psi$, and $z_\psi=\sum_j m_j\cdot e_j$. We shall call $\calS_\psi$ or $\overline{\calS_\psi}$ the component group associated to $\psi$. To such a $\psi$, Arthur \cite{MR3135650} and Mok \cite{MR3338302} assigned a finite multi-set $\Pi_\psi(H)$ of irreducible unitary representations of $H$, together with a canonical map (after fixing a Whittaker datum $\scrW'$ of $H$)
\begin{equation*}
	\calJ_{\scrW'}:\Pi_\psi(H)\lra\widehat{\overline{\calS_\psi}}. 
\end{equation*} 
They proved that this multi-set $\Pi_\psi(H)$ and the assignment $\sigma\mapsto\calJ_{\scrW'}(\sigma)$ satisfy the following two properties:
\begin{enumerate}
	\item If both $H$ and $\sigma$ are unramified, then $\calJ_{\scrW'}(\sigma)=1$.
	\item Let $\phi_\psi$ be the $L$-parameter associated to $\psi$ as in equation  (\ref{AparameterassLpara}). Then the $A$-packet $\Pi_\psi(H)$ contains the $L$-packet $\Pi_{\phi_\psi}^L(H)$ as a subset. We have a commutative diagram
	\begin{equation*}%\label{localAcontainlocalL}
		\begin{CD}
			\Pi_{\phi_\psi}^L(H) @>\calJ_{\scrW'}^L>> \widehat{\overline{\calS_{\phi_\psi}}}\\
			@VVV @VVV\\
			\Pi_\psi(H) @>\calJ_{\scrW'}>> \widehat{\overline{\calS_{\psi}}}
		\end{CD},
	\end{equation*}
	where the left vertical arrow is the natural inclusion and the right vertical arrow is induced by the surjection $\calS_\psi\ra\calS_{\phi_\psi}$. Moreover, if $\psi=\phi$ is a tempered $L$-parameter, then $\Pi_\psi(H)$ is multiplicity free and coincide with the $L$-packet $\Pi_{\phi}^L(H)$.
\end{enumerate}
Besides these two properties, the packet $\Pi_\psi(H)$ also satisfies the so-called ``endoscopic character identity'', but we will not use this fact in this paper. 

Let $\Psi^+(H)$ to be the set of local $A$-parameters, whose elements do not necessarily have bounded images on the Weil group. Due to the potential failure of the generalized Ramanujan conjecture for general linear groups, for the purpose of the global classification, Arthur and Mok also defined the local $A$-packet $\Pi_\psi(H)$ and the canonical map $\calJ_{\scrW'}$ for $\psi\in \Psi^+(H)$ by using parabolic inductions. These local $A$-packets $\Pi_\psi(H)$ and maps $\calJ_{\scrW'}$ also satisfy the two properties listed above. In Section \ref{S:LP}, we will discuss this issue in more details.

Now we turn to the global classification, so $F$ is now a number field. Let $\psi$ be an elliptic $A$-parameter for $H$ and write 
\begin{equation*}
	\psi=\sum_i \phi_i\boxtimes S_{d_i}
\end{equation*}
as in (\ref{paradecomp}). Let 
\begin{equation*}
	\calS_{\psi}=\prod_i (\mathbb Z/2\mathbb Z)e_i \quad \textit{and} \quad \overline{\calS_\psi}=\calS_{\psi}/\langle z_\psi\rangle
\end{equation*}
be the global component groups of $\phi$, where $z_\psi=\sum_i e_i$. For each place $v$ of $F$, the localization $\psi_v$ of $\psi$ at $v$ gives rise to a local $A$-parameter for $H_v$. We also have a localization map 
\begin{align*}
	\calS_\psi\lra \calS_{\psi_v} \quad \left(\textit{or}\quad\overline{\calS_\psi}\lra \overline{\calS_{\psi_v}}\right).
\end{align*}
Fix a global Whittaker datum $\scrW'$ of $H$. Given an elliptic $A$-parameter $\psi$, we define the global packet $\Pi_\psi(H)$ associated to $\psi$ as the restricted tensor product of the local $A$-packets $\Pi_{\psi_v}(H_v)$:
\begin{align*}
	\Pi_\psi(H)&=\otimes'_v\Pi_{\psi_v}(H_v)\\
	&=\{\sigma=\otimes'_v\sigma_v~|~\sigma_v\in\Pi_{\psi_v}(H_v),~\sigma_v\textit{ unramified with $L$-parameter $\phi_{\psi_v}$ for almost all }v\}.
\end{align*}
We then have a map
\begin{align*}
	\calJ_{\scrW'}:\Pi_\psi(H)&\lra\wh{\overline{\calS_\psi}},\\
	\sigma&\longmapsto\calJ_{\scrW'}(\sigma),\\
	\calJ_{\scrW'}(\sigma)(x)&\coloneqq\prod_v\calJ_{\scrW'_v}(\sigma_v)(x_v),
\end{align*}
where $x\in\calS_\psi$ and $x_v$ is the localization of $x$ at $v$. 
\begin{remark}
	According to the main local theorems of \cite{MR3135650} and \cite{MR3338302}, if $H_v$ and $\sigma_v$ are both unramified, then $\calJ_{\scrW'_v}(\sigma_v)$ is the trivial character $1$. Hence $\calJ_{\scrW'}$ is well-defined.
\end{remark}
Let $\epsilon_\psi\in\wh{\calS_\psi}$ be the canonical sign character defined by Arthur \cite[p.47]{MR3135650} and Mok \cite[p.29]{MR3338302}. Put
\begin{equation*}
	\Pi_\psi(H,\epsilon_\psi)=\{\sigma\in\Pi_\psi(H)~|~\calJ_{\scrW'}(\sigma)=\epsilon_\psi\}.
\end{equation*}
Then the main global theorems in \cite{MR3135650} and \cite{MR3338302} assert that:
\begin{theorem}\label{H-ThmA}
	There exists a decomposition
	\begin{equation*}
		L^2_{\disc}(H)=\bigoplus_{\psi\in\Psi_{ell}(H)} L^2_\psi(H),
	\end{equation*}
	where $L^2_\psi(H)$ is a full near equivalence class of irreducible representations $\sigma$ in $L^2_{\disc}(H)$ such that the $L$-parameter of $\sigma_v$ is $\phi_{\psi_v}$ for almost all places $v$ of $F$.
\end{theorem}
\begin{theorem}\label{H-ThmB}
	Let $\psi$ be an elliptic $A$-parameter for $H$. Then we have the decomposition
	\begin{equation*}
		L^2_\psi(H)=\bigoplus_{\sigma\in\Pi_\psi(H,\epsilon_\psi)}\sigma.
	\end{equation*}
\end{theorem}
Our Theorem \ref{ThmA} is an analog of Theorem \ref{H-ThmA} for the group $G$, and our Theorem \ref{ThmB} is an analog of Theorem \ref{H-ThmB} for generic elliptic $A$-parameters of $G$. 

\subsection{Remarks on Whittaker data}\label{Whittaker.Data}
In the rest of this paper, we will prove Theorem \ref{ThmA} and Theorem \ref{ThmB} by using the theta lift between $G$ and $H$. Since the local and global classifications of both $G$ and $H$ depend on the choices of Whittaker data, and the theta lift also depends on the choice of an additive character, we need to specify the data we are using and their relation. Here we are following \cite[Conj.4.4 \& 4.6]{MR3788848} and \cite[Sect.4.4 \& 4.6]{MR3573972}, as we will use their results in later proofs (see Theorem \ref{Prasad.Conj}). We now briefly describe the way we choose these data.\\

When $F$ is a local field, firstly we fix a non-trivial additive character $\psi_F$ of $F$. Besides, we also need some auxiliary data in different cases:
\begin{itemize}
	\item in Case O, we fix an isometry
\[
  V^+\simeq V_{(d,c)}+\calH^{n-1}
\]
for some $d,c\in F^\times $, as described at the beginning of Section \ref{state.main.results};
	\item in Case U, we fix a trace zero element $\delta\in E^\times$.
\end{itemize}	
We define a Whittaker datum $\scrW=\scrW_{\psi_F}$ of $G^*=G(V^+)$ as follows.\\

\noindent\underline{\textit{Case ${\rm O}$:}}
In this case, recall that we have fixed an isometry
\[
  V^+\simeq V_{(d,c)}+\calH^{n-1},
\]
where $\calH$ is the (orthogonal) hyperbolic plane. We denote by $e,e'$ the images of $1,X\in F[X]$ in $V_{(d,c)}$ respectively. For $1\leq k\leq n-1$, we write the $k$-th hyperbolic plane $\calH=Fv_k+Fv_k^*$ with 
\[
  \langle v_k,v_k\rangle_V=\langle v_k^*,v_k^*\rangle_V=0\quad\textit{and}\quad \langle v_k,v_k^*\rangle_V=1,
\]
and we set
\[
  X_k=Fv_1+\cdots+Fv_k\quad\textit{and}\quad X_k^*=Fv_1^*+\cdots+Fv_k^*.
\]
Let $B=TU$ be the $F$-rational Borel subgroup of $G^*$ stabilizing the complete flag
\[
  X_1\subset\cdots\subset X_{n-1},
\]
where $T$ is the $F$-rational torus stabilizing the lines $Fv_i$ for $1\leq k\leq n-1$. We define a generic character $\mu_c$ of $U$ by
\begin{equation}\label{whittaker So}
  \mu_c(u)=\psi_F(\langle uv_2,v_1^*\rangle_V+\cdots+\langle uv_{n-1},v_{n-2}^*\rangle_V+\langle ue,v_{n-1}^*\rangle_V).
\end{equation}
Let $\scrW=(V^+,B,T,\mu_c)$. Here the notion of Whittaker datum is slightly different from the usual one: the datum is not only associated to the group $G^*$, but rather the orthogonal space $V^+$ is part of the datum. Moreover, this datum $\scrW$ is indeed independent of the choice of the additive character $\psi$, but only depend on the choice of $c\in F^\times$. Readers may consult \cite[Sect.2.2]{MR3708200} for a discussion of this.\\

\noindent\underline{\textit{Case $\U_0$:}}
In this case, the Witt index of $V^+$ is $n$. We choose a basis $\{v_i,v_i^*~|~i=1,\cdots,n\}$ of $V^+$ such that
\[
	\langle v_i,v_j\rangle_V=\langle v_i^*,v_j^*\rangle_V=0\quad\textit{and}\quad \langle v_i,v_j^*\rangle_V=\delta_{i,j}
\]
for $1\leq i,j\leq n$. We set
\[
	X_k=Ev_1+\cdots+Ev_k\quad\textit{and}\quad X_k^*=Ev_1^*+\cdots+Ev_k^*
\]
for $1\leq i,j\leq n$. We denote by $B=TU$ the $F$-rational Borel subgroup of $G^*$ stabilizing the complete flag
\[
	X_1\subset\cdots\subset X_{n},
\]
where $T$ is the $F$-rational torus stabilizing the lines $Ev_i$ for $1\leq k\leq n$. We define a generic character $\mu$ of $U$ by
\[
	\mu(u)=\psi_F\Big(\frac{1}{2}\Tr_{E/F}\big(\delta\cdot(\langle uv_2,v_1^*\rangle_V+\cdots+\langle uv_{n},v_{n-1}^*\rangle_V+\langle uv_{n}^*,v_{n}^*\rangle_V)\big)\Big).
\]
Let $\scrW=(U,\mu)$.\\

\noindent\underline{\textit{Case $\U_1$:}}
In this case, there is an unique Whittaker datum $\scrW$ of $G^*$.\\

Then we define a Whittaker datum $\scrW'=\scrW'_{\psi_F}$ of $H=H(W)$ as follows.\\

\noindent\underline{\textit{Case ${\rm O}$:}}
In this case $W$ is the $2r$-dimensional symplectic space. We choose a basis $\{w_i,w_i^*~|~i=1,\cdots,r\}$ of $W$ such that
\[
	\langle w_i,w_j\rangle_W=\langle w_i^*,w_j^*\rangle_W=0\quad\textit{and}\quad \langle w_i,w_j^*\rangle_W=\delta_{i,j}
\]
for $1\leq i,j\leq r$. We set
\[
	Y_k=Fw_1+\cdots+Fw_k\quad\textit{and}\quad Y_k^*=Fw_1^*+\cdots+Fw_k^*
\]
for $1\leq i,j\leq r$. We denote by $B'=T'U'$ the $F$-rational Borel subgroup of $H$ stabilizing the complete flag
\[
	Y_1\subset\cdots\subset Y_{r},
\]
where $T'$ is the $F$-rational torus stabilizing the lines $Fw_i$ for $1\leq k\leq r$. We define a generic character $\mu'$ of $U'$ by
\begin{equation}\label{Whitttaker sp}
	\mu'(u)=\psi_F\Big(\langle uw_2,w_1^*\rangle_W+\cdots+\langle uw_{r},w_{r-1}^*\rangle_W+c\langle uw_{r}^*,w_{r}^*\rangle_W\Big).
\end{equation}
%\zjl{ Atobe GGP paper page $234$. Check
%\begin{equation}\label{Whitttaker sp}
%	\mu'(u)=\psi_F\Big(\langle uw_2,w_1^*\rangle_W+\cdots+\langle uw_{r},w_{r-1}^*\rangle_W+c\langle uw_{r}^*,w_{r}^*\rangle_W\Big).
%\end{equation}
%}
Let $\scrW'=(U',\mu')$.\\

\noindent\underline{\textit{Case $\U_0$:}} 
In this case $W$ is a split $(2r+1)$-dimensional skew-Hermitian space. There is a is a unique Whittaker datum $\scrW'$ of $H$. \\

\noindent\underline{\textit{Case $\U_1$:}} 
In this case $W$ is a split $2r$-dimensional skew-Hermitian space. We choose a basis $\{w_i,w_i^*~|~i=1,\cdots,r\}$ of $W$ such that
\[
	\langle w_i,w_j\rangle_W=\langle w_i^*,w_j^*\rangle_W=0\quad\textit{and}\quad \langle w_i,w_j^*\rangle_W=\delta_{i,j}
\]
for $1\leq i,j\leq r$. We set
\[
	Y_k=Fw_1+\cdots+Fw_k\quad\textit{and}\quad Y_k^*=Fw_1^*+\cdots+Fw_k^*
\]
for $1\leq i,j\leq r$. We denote by $B'=T'U'$ the $F$-rational Borel subgroup of $H$ stabilizing the complete flag
\[
	Y_1\subset\cdots\subset Y_{r},
\]
where $T'$ is the $F$-rational torus stabilizing the lines $Fw_i$ for $1\leq k\leq r$. We define a generic character $\mu'$ of $U'$ by
\[
	\mu'(u)=\psi_F\Big(\frac{1}{2}\Tr_{E/F}\big(\langle uw_2,w_1^*\rangle_W+\cdots+\langle uw_{r},w_{r-1}^*\rangle_W+\langle uw_{r}^*,w_{r}^*\rangle_W\big)\Big).
\]
Let $\scrW'=(U',\mu')$.\\

When $F$ is a number field, after fixing an additive character $\psi_F$ of $\AAA/F$ (and also some auxiliary data as in the local case), we can define (global) Whittaker data $\scrW=\scrW_{\psi_F}$ and $\scrW'=\scrW'_{\psi_F}$ of $G^*$ and $H$ using the same formulas described above. By our definitions, these Whittaker data satisfy the local-global property, in the sense that
\[
	\scrW_v=\scrW_{\psi_{F,v}} \quad \textit{and} \quad \scrW'_v=\scrW'_{\psi_{F,v}} 
\]
for all places $v$ of $F$.

\section{Weak transfer to the general linear group}
From now on we assume that $\dim V<r$, hence the reductive dual pair $(G,H)$ is in the stable range. In this section, we prove Theorem \ref{ThmA}, which shows that each NEC in $L^2_{\disc}(G)$ has a weak transfer to the general linear group $\GL(\scrV)$, where $\scrV$ is the standard representation of the dual group $\widehat{G}$ of $G$.

\subsection{Attaching Arthur parameters}
Let $F$ be a number field. Let $C$ be a near equivalence class in $L_{\disc}^2(G)$. Then $C$ gives rise to a collection of $L$-parameters
\[
\phi_v:L_{F_v}\lra\Lgp{G}
\]
for almost all $v$, such that for any irreducible summand $\pi$ of $C$, the $L$-parameter of $\pi_v$ is $\phi_v$ for almost all $v$. We have:
\begin{proposition}\label{Weaklift}
	There exists a unique elliptic $A$-parameter $\psi$ for $G$ such that $\phi_{\psi_v}=\phi_v$ for almost all $v$.
\end{proposition}
\begin{proof}
	The strategy of the proof is the same as that of \cite[Prop.3.1]{MR3866889}. For the convenience of the reader, we sketch the proof here. Let $\pi$ be any irreducible summand of $C$, consider the theta lift between $(G,H)$. Since $m_{\disc}(\pi)\geq 1$, we deduce from J-S. Li's inequality Theorem \ref{J-S inequality} that
\[
	m_{\disc}\left(\theta^{abs}(\pi)\right)\geq 1.
\] 
Therefore, Theorem \ref{H-ThmA} attaches an elliptic $A$-parameter $\theta(\psi)$ to $\theta^{abs}(\pi)$. 

%By computing the partial $L$-function 
%\[
%	L^S(s,\theta^{abs}(\pi)\times\chi_V^{-1}),
%\]
Next we show that $\theta(\psi)$ contains $\chi_V\boxtimes S_{2r-2n+1}$ as a direct summand. 
%\iffalse%%%%%%%%%%%%%%%%%%%%%%%%%%%%%%%%%%%%%%%%%%%%%%%%%%%%%%%%%%%%%%%%%%%%%%%%%%%%%%%%%%%%%%%%%%%%%%%%%%%%%%%%%%%%%%%%%%%%%%%%
Consider the partial $L$-function $L^S\left(s,\theta^{abs}(\pi)\times\chi_V^{-1}\right)$ associated to $\theta^{abs}(\pi)$ and $\chi_V^{-1}$. Here $S$ is a sufficiently large finite set of places of $F$ such that for any place $v$ outside $S$, the $L$-parameter of $\pi_v$ is $\phi_v$, and $G_v, H_v,\psi_{F,v},\chi_{V,v},\chi_{W,v},\pi_v$ are all unramified. If we write $\theta(\psi)=\sum_i\phi_i\boxtimes S_{d_i}$ as in (\ref{paradecomp}), then
	\begin{equation}\label{ThmAeq1}
		L^S\left(s,\theta^{abs}(\pi)\times\chi_V^{-1}\right)=\prod_i\prod_{j=1}^{d_i}L^S\left(s+\frac{d_i+1}{2}-j,\phi_i\chi_V^{-1}\right).
	\end{equation}
	It follows from \cite[Lemma 4.4, Theorem 5.3]{MR618323} that $L^S\left(s,\phi_i\chi_V^{-1}\right)$ is holomorphic for $\Re(s)>1$ for all $i$, and it has a pole at $s=1$ if and only if $\phi_i=\chi_V$ is an automorphic character of $\GL_1(\AAA_E)$. On the other hand, by the local theta correspondence for unramified representations \cite[Thm.6.1]{MR658543}, for all $v\notin S$, the $L$-parameter of $\theta(\pi_v)$ is
	\begin{equation}\label{ThmAeq2}
		\phi_v\chi_{W,v}^{-1}\chi_{V,v}+\left(\bigoplus_{j=n-r}^{r-n}|\cdot|^j\right)\chi_{V,v}.
	\end{equation}
	Hence 
	\begin{equation}\label{ThmAeq3}
		L^S\left(s,\theta^{abs}(\pi)\times\chi_V^{-1}\right)=L^S\left(s,\pi\times\chi_W^{-1}\right)\prod_{j=n-r}^{r-n}\zeta^S(s+j),
	\end{equation}
	where $L^S\left(s,\pi\times\chi_W^{-1}\right)$ is the partial $L$-function associated to $\pi$ and $\chi_W^{-1}$. By \cite[Thm.9.1]{MR3211043}, we know that the partial $L$-function $L^S\left(s,\pi\times\chi_W^{-1}\right)$ is holomorphic when
	\[
		\Re(s)>\begin{cases}
			n \quad &\textit{Case ${\rm O}$};\\
			n+\frac{1}{2} \quad &\textit{Case $\U_0$};\\
			n+1 \quad &\textit{Case $\U_1$}.
		\end{cases}
	\]
	It follows from (\ref{ThmAeq3}) that $L^S\left(s,\theta^{abs}(\pi)\times\chi_V^{-1}\right)$ is holomorphic for $\Re(s)>r-n+1$ but has a simple pole at $s=r-n+1$. 
 %\trivial[]{By checking \cite[Thm.9.1]{MR3211043}, the bound for Case O is $\Re(s)>n$, correspondence to trivial representation with $A$-parameter $1\oplus 1\boxtimes S_{2n-1}$, the bound for Case $\U_0$ and $U_1$ is $\Re(s)>n+\frac{1}{2}$ and $\Re(s)>n+1$, correspondence to trivial representation with $A$-parameter $ 1\boxtimes S_{2n}$ and $ 1\boxtimes S_{2n+1}$}
 This and (\ref{ThmAeq1}) imply that $\theta(\psi)$ contains $\chi_V\boxtimes S_t$ as a direct summand for some $t$. Let $t$ be the largest integer with this property. Then $L^S\left(s,\theta^{abs}(\pi)\times\chi_V^{-1}\right)$ has a largest pole at $s=\frac{t+1}{2}$. So we have $t=2r-2n+1$. 
%\fi%%%%%%%%%%%%%%%%%%%%%%%%%%%%%%%%%%%%%%%%%%%%%%%%%%%%%%%%%%%%%%%%%%%%%%%%%%%%%%%%%%%%%%%%%%%%%%%%%%%%%%%%%%%%%%%%%%%%%%%%%%%%%%%%%
	Thus we may write
	\[
	\theta(\psi)=\psi\chi_W^{-1}\chi_V+\chi_V\boxtimes S_{2r-2n+1}
	\]
	for some elliptic $A$-parameter $\psi\in\Psi_{ell}(G^*)$. This and (\ref{ThmAeq2}) imply that $\phi_{\psi_v}=\phi_v$ for almost all $v$. The uniqueness of $\psi$ simply follows from the strong multiplicity one theorem for general linear groups \cite{MR623137}.
\end{proof}

Now we denote by $L^2_{\psi}(G)$ the near-equivalence class $C$ with the associated $A$-parameter $\psi$. Then we have a decomposition 
\[
L^2(G)=\bigoplus_{\psi\in\Psi_{ell}(G^*)} L^2_{\psi}(G).
\]
This completes the proof of Theorem \ref{ThmA}.

\section{A key equality: Generic case}\label{Generic.case}
Let $F$ be a number field, and $G$ an even orthogonal or unitary group over $F$ as in the setting of Section \ref{state.main.results}. In this section, we study the structure of $L_{\phi}^2(G)$ for a generic elliptic $A$-parameter $\phi$. 
\subsection{Gan-Ichino's observation} 
For any irreducible representation $\pi$ of $G(\AAA)$, we define the multiplicity $m_{\cusp}(\pi)$ by
\[
m_{\cusp}(\pi)=\dim\Hom_{G(\AAA)}(\pi,\calA_{\cusp}(G)).
\]
Obviously, we have
\[
m_{\cusp}(\pi)\leq m_{\disc}(\pi)\leq m(\pi).
\]
The following observation due to Gan-Ichino is the core of this work.
\begin{proposition}\label{cusp.realization}
	Let $\phi$ be a generic elliptic $A$-parameter for $G$. Let $\pi$ be an irreducible representation of $G(\AAA)$, such that the $L$-parameter of $\pi_v$ is $\phi_v$ for almost all $v$. Then we have
	\[
	m_{\cusp}(\pi)=m_{\disc}(\pi)=m(\pi).
	\]
\end{proposition}
\begin{proof}
	The proof is totally the same as that of \cite[Prop.4.1]{MR3866889}. 
%%%%%%%%%%%%%%%%%%%%%%%%%%%%%%%%%%%%%%%%%%%%%%%%%%%%%%%%%%%%%%%%%%%%%%%%%%%%%%%%%%%%%%%%%%%%%%%%%%%%%%%%%%%%%
\end{proof}

\subsection{The multiplicity preservation}
Consider the theta lift between $(G,H)$ with respect to a datum $(\psi_F,\chi_V,\chi_W)$ as described in Section \ref{ThetaLifts}. Let $\sigma$ be an irreducible representation of $H(\AAA)$. We say $\sigma$ is relevant to $G$, if there is an irreducible unitary representation $\pi$ of $G(\AAA)$ and an automorphic character $\chi$ of $H(\AAA)$, such that
\[
\sigma\simeq\theta^{abs}(\pi)\otimes\chi.
\]
As a consequence of the previous proposition, we deduce:
\begin{corollary}\label{Multi-Preserve}
	Let $\phi$ be a generic elliptic $A$-parameter for $G$. Suppose that
	\[
	L_{\phi}^2(G)=\bigoplus_\pi m_\pi\pi.
	\]
	Let $\theta(\phi)=\phi\chi_W^{-1}\chi_V+\chi_V\boxtimes S_{2r-2n+1}$ be an elliptic $A$-parameter for $H$. Then
	\[
	L_{\theta(\phi)}^2(H)=\left(\bigoplus_\pi m_\pi\theta^{abs}(\pi)\right)\oplus\left(\bigoplus_\sigma m_\sigma\sigma\right),
	\]
	where the second summation on the RHS runs over all $\sigma$ with $A$-parameter $\theta(\phi)$ and not relevant to $G$.
\end{corollary}
\begin{proof}
	It follows from Theorem \ref{J-S inequality} and the Howe duality that there is an injection 
	\[
	\bigoplus_\pi m_\pi\theta^{abs}(\pi)\hookrightarrow L_{\theta(\phi)}^2(H).
	\]
	It remains to show that for any irreducible summand $\sigma$ of $L_{\theta(\phi)}^2(H)$ relevant to $G$, there is an irreducible summand $\pi$ of $L_{\phi}^2(G)$, such that $\sigma\simeq\theta^{abs}(\pi)$ and $m_{\disc}(\pi)=m_{\disc}(\theta^{abs}(\pi))$.

	So now suppose that $\sigma$ is relevant to $G$. By definition there is an unique irreducible unitary representation $\pi$ of $G(\AAA)$ and an unique automorphic character $\chi$ of $H(\AAA)$ such that
	\[
	\sigma\simeq\theta^{abs}(\pi)\otimes\chi.
	\]
	It follows from J-S. Li's results \cite[Prop.5.7]{MR1448215} that $\pi$ and $\chi$ satisfying this condition are unique. By Theorem \ref{H-ThmA}, we know that the $L$-parameter of $\sigma_v$ is
	\begin{equation}\label{E:S5unramLsig}
	\phi_v\chi_{W,v}^{-1}\chi_{V,v}+\left(\bigoplus_{j=n-r}^{r-n}|\cdot|^j\right)\chi_{V,v}
	\end{equation}
	for almost all $v$. Then it follows from the local theta lift for unramified representations that
	\[
	\sigma_v\simeq\theta(\pi'_v)
	\]
	for almost all $v$, where $\pi'_v$ is the unramified representation of $G_v$ with the $L$-parameter $\phi_v$. Therefore, by the uniqueness of $\pi$ and $\chi$, the $L$-parameter of $\pi_v$ is $\phi_v$ and $\chi_v$ is trivial for almost all $v$. Since $\chi$ is automorphic, it must be trivial, so that $\sigma\simeq\theta^{abs}(\pi)$. Moreover, since the $L$-parameter of $\pi_v$ is $\phi_v$ for almost all $v$, we have
	\[
	m_{\disc}(\pi)=m_{\disc}(\theta^{abs}(\pi))
	\]
	thanks to Theorem \ref{J-S inequality} and Proposition \ref{cusp.realization}. This completes the proof.
\end{proof}
\begin{remark}\label{Multi-Preserve-PIF}
	As explicated at the beginning of Section \ref{state.main.results}, all pure inner forms of $G$ arise in the forms $G'=G(V')$ for some spaces $V'$. Suppose that for $G'$ we have the decomposition
	\[
	L_{\phi}^2(G')=\bigoplus_{\pi'} m_{\pi'}\pi'.
	\]
	Consider the theta lift between $(G',H)$ for all such $G'$ simultaneously. For any irreducible summand $\sigma$ of $L_{\theta(\phi)}^2(H)$, at almost all places $v$ of $F$ the localization $\sigma_v$ is unramified with the $L$-parameter as in (\ref{E:S5unramLsig}). Hence by using a same argument as in the proof of Corollary \ref{Multi-Preserve}, one can show that $\sigma$ is of the form $\theta^{abs}(\pi')$ for a unique summand $\pi'$ of $L_{\phi}^2(G')$, where $G'$ is a pure inner form of $G$. This implies that
	\[
	L_{\theta(\phi)}^2(H)=\bigoplus_{G'}\left(\bigoplus_{\pi'} m_{\pi'}\theta^{abs}(\pi')\right),
	\]
	where the first summation on the RHS runs over all pure inner forms of $G$. 
\end{remark}

\subsection{Transferring the multiplicity formula}\label{transfer.AMF}
In this subsection we define some notions and ``transfer'' the AMF from $H$ to $G$. We shall work in a general setting first, and then specialize to the generic case. 

Let $\psi=\sum_i\phi_i\boxtimes S_{d_i}$ be an elliptic $A$-parameter for $G$, and
\[
\theta(\psi)=\psi\chi_W^{-1}\chi_V+\chi_V\boxtimes S_{2r-2n+1}
\]
be an elliptic $A$-parameter for $H$. Then
\[
\calS_\psi=\prod_i(\ZZ/2\ZZ)e_i\quad\textit{and}\quad \calS_{\theta(\psi)}=\left(\prod_i(\ZZ/2\ZZ)e'_i\right)\times\ \Big(\left(\ZZ/2\ZZ\right)a\Big),
\]
where $e_i$ corresponds to $\phi_i\boxtimes S_{d_i}\subset\psi$, $e'_i$ corresponds to $\phi_i\chi_W^{-1}\chi_V\boxtimes S_{d_i}\subset\theta(\psi)$, and $a$ corresponds to $\chi_V\boxtimes S_{2r-2n+1}\subset\theta(\psi)$. The composition of the following two maps
\begin{equation*}%\label{componentGlobalmap}
	\ell:\calS_\psi\lra\calS_{\theta(\psi)}\lra\overline{\calS_{\theta(\psi)}}
\end{equation*}
is an isomorphism, where the first map sends $e_i$ to $e'_i$, and the second map is the natural projection. Recall that Arthur \cite[p.47]{MR3135650} and Mok \cite[p.29]{MR3338302} have defined the so-called canonical sign characters $\epsilon_\psi$ and $\epsilon_{\theta(\psi)}$ of $\calS_\psi$ and $\overline{\calS_{\theta(\psi)}}$. We first give an explicit description of these characters.

\begin{proposition}\label{P:epsilonfactor}
We have
\begin{equation*}\label{epsilonfactor unitary}
		\epsilon_\psi(e_i)=\prod_{j\neq i}\epsilon\left(\frac{1}{2},\phi_i\times\phi_j^\vee \right)^{\min\{d_i,d_j\}}.
\end{equation*}
Likewise, a similar formula holds for $\epsilon_{\theta(\psi)}$.
\end{proposition} 

\begin{proof}
For Case O, this is proved in \cite[Prop.-Def.8.3.7]{MR3929692}. So we will only prove this in Case U. To make the proof more streamlined, we shall make use of hypothetical global Langlands groups $\calL_F$ and $\calL_E$. However, we should mention that these two groups are not necessary, one can introduce the substitutes $\calL_\psi$ and $\calL_{\psi/E}$ as in \cite[p.19(2.4.3)]{MR3338302}, and our proof still works after some slight modifications.

By the global Langlands conjecture, there is an one to one correspondence between irreducible cuspidal representations of $\GL_k(E)$ and irreducible $k$-dimensional representations of $\calL_E$. Therefore each $\phi_i$ corresponds to an irreducible representation of $\calL_E$, which we shall still denote by $\phi_i$. For each $i$, let $\psi_i=\phi_i\boxtimes S_{d_i}$ be an irreducible representation of $\calL_E\times \SL_2(\CC)$. Then the global $A$-parameter $\psi$ can be regarded as the direct sum of these $\psi_i$, and the conjugate self-duality of $\psi$ allows us to extend it to an $L$-homomorphism
\[
	\widetilde{\psi}: \calL_F\times\SL_2(\CC) \lra {}^LG.
\] 
Using this $L$-homomorphism, the global component group $\calS_\psi$ can be identified with the centralizer of $\im(\widetilde{\psi})$ in $\widehat{G}$. To define the character $\epsilon_\psi$, we need to consider the adjoint representation of ${}^LG$ on the Lie algebra $\widehat{\mathfrak g}$ of the Langlands dual group $\widehat G$. Let  
\[
	\tau_\psi: \calS_\psi\times \calL_F\times \SL_2(\mathbb C) \longrightarrow \GL(\widehat {\mathfrak g}) 
\]
be the representation given by the formula 
\[
	\tau_\psi\left(x, g, h\right)= \Ad\left(x\cdot \widetilde{\psi}(g, h)\right)
\]
for $x\in \calS_\psi$, $g\in\calL_F$ and $h\in\SL_2(\CC)$. If this representation $\tau_\psi$ is decomposed as 
\[
	\tau_\psi= \bigoplus_\alpha \lambda_\alpha \boxtimes \mu_\alpha \boxtimes \nu_\alpha 
\]
for some irreducible representations $\lambda_\alpha$, $\mu_\alpha$, $\nu_\alpha$ of $\calS_\psi$, $\calL_F$ and $\SL_2(\mathbb C)$ respectively, then the character $\epsilon_\psi$ is define by 
\begin{equation*}
	\epsilon_\psi(e_i)=\prod_{\alpha}' \det \left(\lambda_\alpha\left(e_i\right)\right),
\end{equation*}
where $\prod'_\alpha$ denotes the product over the indices $\alpha$ such that $\mu_\alpha$ is symplectic and $\epsilon\left(\frac{1}{2}, \mu_\alpha\right)=-1$. By some elementary computations, we know that the representation $\tau_\psi$ can be decomposed as  
\[
	\tau_\psi \simeq \left(\sum_{i} \varepsilon_i^2\As^{\pm}(\psi_i)\right) \oplus \left(\sum_{i< j} \varepsilon_i\varepsilon_j\cdot\Ind_{\calL_{E}\times \SL_2(\mathbb C)}^{\calL_F\times \SL_2(\mathbb C)} \left(\psi_i\otimes \psi_j^\vee\right)\right), 
\]
where $\varepsilon_i$ is the character of $\calS_\psi$ defined by 
\[
	\varepsilon_i:\calS_\psi\lra \{\pm1\}, \quad e_j\longmapsto \begin{cases}
		+1 \quad &\textit{if } j=i;\\
		-1 \quad &\textit{if } j\neq i.
	\end{cases}
\]
Also note that 
\[
	\Ind_{\calL_{E}\times \SL_2(\mathbb C)}^{\calL_F\times \SL_2(\mathbb C)}\left(\psi_i\otimes\psi_j^\vee\right) = \sum_{k=1}^{\min\{d_i,d_j\}} \Ind_{\calL_E}^{\calL_F}\left(\phi_i\otimes\phi_j^\vee\right) \boxtimes S_{d_i+d_j-2k+1}.
\]
Therefore by the definition we have
\[
	\epsilon_\psi\left(e_i\right) = \prod_{j\neq i}\epsilon\left(\frac{1}{2},\Ind_{\calL_E}^{\calL_F}\left(\phi_i\otimes\phi_j^\vee\right)\right)^{\min\{d_i,d_j\}} = \prod_{j\neq i}\epsilon\left(\frac{1}{2},\phi_i\times\phi_j^\vee\right)^{\min\{d_i,d_j\}}
\]
Here in the last equality we use the inductivity of epsilon factors. %\textit{\color{red}need a reference for this...}.
\end{proof}

\begin{remark}
When $\psi=\phi$ is generic, for any indices $i$ and $j$, both $\phi_i$ and $\phi_j$ are (conjugate) self-dual of the same parity. By \cite[Thm.1.5.3]{MR3135650} (for Case O) and \cite[Thm.2.5.4]{MR3338302} (for Case U), we have $\epsilon\left(\frac{1}{2},\phi_i\times\phi_j^\vee\right)=1$. This implies that
\begin{equation}\label{epsilon-generic}
	\epsilon_\phi=1. 
\end{equation}
\end{remark}
%%%%%%%%%%%%%%%%%%%%%%%
\begin{comment}
We can describe the character  $\epsilon_\psi$ more explicitly as follows. 
\begin{lemma}\label{epsilon-explict}
We have 
\begin{equation}\label{epsilonfactor2}
		\epsilon_\psi(e_i)=
  \begin{cases}
  \prod_{j\neq i}\epsilon\left(\frac{1}{2},\phi_i\times\phi_j\right)^{\min\{d_i,d_j\}}\quad & \mbox{in Case O}\\
\prod_{j\neq i}\epsilon\left(\frac{1}{2},\phi_i\times\phi_j^\vee\right)^{\min\{d_i,d_j\}}\quad & \mbox{in Case U}
  \end{cases}
  \end{equation}
  where $\epsilon(s, \phi_i\times \phi_j)$ (resp. $\epsilon(s, \phi_i\times \phi_j^\vee)$) are the global $\epsilon$-factor of the pair $(\phi_i,\phi_j)$ (resp. $(\phi_i,\phi_j^\vee)$). 
\end{lemma}
\begin{proof}
The character $\epsilon_\psi$ is explicated in \cite[Prop.-Def.8.3.7]{MR3929692} in the case of orthogonal and symplectic groups. We can apply the same argument to the case of unitary groups. For the convenience of the reader, we sketch the proof here. 
\end{proof}
\end{comment}
%%%%%%%%%%%%%%%%%%%%%%%%%%%%%

Next we compare the character $\epsilon_\psi$ and the pull back of $\epsilon_{\theta(\psi)}$ along the map $\ell$ using this description.  

\begin{lemma}\label{Compare-epsilon}
	We have
	\[
	\epsilon_\psi=\ell^*\left(\epsilon_{\theta(\psi)}\right).
	\]
\end{lemma}
\begin{proof}
By Proposition \ref{P:epsilonfactor}, we have
\begin{equation}\label{epsilonfactorComputation}
\begin{split}
		\epsilon_{\theta(\psi)}(e'_i)=&\left(\prod_{j\neq i}\epsilon\left(\frac{1}{2},\phi_i\chi_V\chi_W^{-1}\times\phi_j^\vee \chi_V^{-1}\chi_W\right)^{\min\{d_i,d_j\}}\right)\times\epsilon\left(\frac{1}{2},\phi_i\chi_V\chi_W^{-1}\times\chi_V^{-1}\right)^{d_i}\\
		=&\left(\prod_{j\neq i}\epsilon\left(\frac{1}{2},\phi_i\times\phi_j\right)^{\min\{d_i,d_j\}}\right)\times\epsilon\left(\frac{1}{2},\phi_i\chi_W^{-1}\right)^{d_i}\\
  =&\epsilon_\psi\left(e_i\right)\cdot \epsilon\left(\frac{1}{2},\phi_i\chi_W^{-1}\right)^{d_i}.\\
\end{split}
	\end{equation}
If $\phi_i$ has the same parity as $\psi$, then it also has the same parity as $\chi_W$. By \cite[Thm.1.5.3]{MR3135650} (for Case O) and \cite[Thm.2.5.4]{MR3338302} (for Case U), we have $\epsilon\left(\frac{1}{2},\phi_i\chi_W^{-1}\right)=1$. On the other hand, if $\phi_i$ has different parity with $\psi$, then $d_i$ must be even.  
Therefore we always have $\epsilon\left(\frac{1}{2},\phi_i\chi_W^{-1}\right)^{d_i}=1$. The desired conclusion then follows from equation (\ref{epsilonfactorComputation}).

%%%%%%%%%%%%%%%%%Case O not used 
\begin{comment}
We first consider the Case O where $G$ is an even orthogonal group, and $H$ is a symplectic group. In this case $\chi_W$ is the trivial character. As explicated in \cite[Prop.-Def.8.3.7]{MR3929692}, we have
	\begin{equation}\label{epsilonfactor}
		\epsilon_\psi(e_i)=\prod_{j\neq i}\epsilon\left(\frac{1}{2},\phi_i\times\phi_j\right)^{\min\{d_i,d_j\}},
	\end{equation}
	and similarly
	\begin{align*}
		\epsilon_{\theta(\psi)}(a'_i)=&\left(\prod_{j\neq i}\epsilon\left(\frac{1}{2},\phi_i\chi_V\times\phi_j\chi_V\right)^{\min\{d_i,d_j\}}\right)\times\epsilon\left(\frac{1}{2},\phi_i\chi_V\times\chi_V\right)^{d_i}\\
		=&\left(\prod_{j\neq i}\epsilon\left(\frac{1}{2},\phi_i\times\phi_j\right)^{\min\{d_i,d_j\}}\right)\times\epsilon\left(\frac{1}{2},\phi_i\right)^{d_i}.
	\end{align*}
	If $\phi_i$ is orthogonal, by \cite[Thm.1.5.3]{MR3135650}, we have 
	\[
	\epsilon\left(\frac{1}{2},\phi_i\right)=1.
	\]
	On the other hand, if $\phi_i$ is symplectic, then $d_i$ must be even. Hence we always have
	\[
	\epsilon\left(\frac{1}{2},\phi_i\right)^{d_i}=1.
	\]
	This implies our assertion in Case O.  
\end{comment} 
%%%%%%%%%%%%%%%%%%%%%%%

\end{proof}

Locally, for each place $v$ of $F$, we also have the natural map of component groups
\begin{equation*}%\label{componentlocalmap}
	\ell_v: \calS_{\psi_v}\lra\calS_{\theta(\psi_v)}\lra\overline{\calS_{\theta(\psi_v)}}.
\end{equation*}
Again $\ell_v$ is an isomorphism. We regard the local $A$-packet $\Pi_{\theta(\psi_v)}(H_v)$ as a representation of $\overline{\calS_{\theta(\psi_v)}}\times H_v$ by setting
\[
\Pi_{\theta(\psi_v)}(H_v)=\bigoplus_\sigma\calJ_{\scrW'}(\sigma)\boxtimes\sigma,
\]
where the summation on the RHS runs over all irreducible unitary representations of $H_v$ in $\Pi_{\theta(\psi_v)}(H_v)$. Let
\[
\Pi_{\psi_v}^{\theta}(G_v)=\bigoplus_\sigma\ell_v^*\left(\calJ_{\scrW'}(\sigma)\right)\boxtimes\theta(\sigma).
\]
Note that $\theta(\sigma)$ could be zero. We shall discard those $\sigma$ being mapped to zero under the theta lift. We regard $\Pi_{\phi_v}^{\theta}(G_v)$ as a (multi) set equipped with a map
\begin{align*}
	\calJ_v^\theta:\Pi_{\psi_v}^{\theta}(G_v)&\lra\calS_{\psi_v},\\
	\theta(\sigma)&\longmapsto\ell_v^*\left(\calJ_{\scrW'}(\sigma)\right).
\end{align*}
Define a global packet
\begin{align*}
	\Pi_\psi^\theta(G)&=\otimes'_v\Pi_{\psi_v}^\theta(G_v)\\
	&=\{\pi=\otimes'_v\pi_v~|~\pi_v\in\Pi_{\psi_v}^\theta(G_v),~\pi_v\textit{ unramified with the L-parameter $\phi_{\psi_v}$ for almost all }v\}.
\end{align*}
We then have a map
\begin{align*}
	\calJ^\theta:\Pi_\psi^\theta(G)&\lra\widehat{\calS_\psi},\\
	\pi&\longmapsto \calJ^\theta(\pi),\\
	\calJ^\theta(\pi)(x)&\coloneqq\prod_v\calJ_v^\theta(\pi_v)(x_v),
\end{align*}
where $x\in\calS_\psi$ and $x_v$ is the localization of $x$ at $v$. It is easy to check that
\[
	\calJ^\theta(\pi)=\ell^*\left(\calJ_{\scrW'}(\theta^{abs}(\pi))\right).
\]
We put
\[
\Pi_\psi^\theta(G,\epsilon_\psi)=\{\pi\in\Pi_\psi^\theta(G)~|~\calJ^\theta(\pi)=\epsilon_\psi\}.
\]
Next we specialize to the generic case. As a direct consequence of Corollary \ref{Multi-Preserve} and Lemma \ref{Compare-epsilon}, we have:
\begin{proposition}\label{transfermulti}
	Let $\psi=\phi$ be a generic elliptic $A$-parameter for $G$. Then there is a decomposition
	\[
	L_\phi^2(G)=\bigoplus_{\pi\in\Pi_\phi^\theta(G,\epsilon_\phi)}\pi.
	\]
\end{proposition}
Hence, to complete the proof of Theorem \ref{ThmB}, it remains to describe $\Pi_{\phi_v}^{\theta}(G_v)$ in terms of the LLC for $G_v$. This will be established in the next section.

\section{Local comparison}\label{localcom}
In this section we let $F$ be a local field of characteristic $0$. We will compare the packets $\Pi_{\phi}^{\theta}(G)$ and $\Pi_{\phi}^{L}(G)$ for all almost tempered $L$-parameters $\phi$ of $G$. %(cf. Section \ref{mainlical}). Note that all localizations of global generic elliptic $A$-parameters are almost tempered.

\subsection{Main local theorem}\label{mainlical}
We first recall several terminologies. Let $\phi$ be a local $L$-parameter for $G$. As explicated in \cite[Sect.8]{MR3202556}, $\phi$ can be regarded as a (conjugate) self-dual representation of $L_E$ of certain parity. If we write $\phi$ as 
\begin{equation*}%\label{localLparameter}
	\phi= \sum_i m_i \phi_i
\end{equation*}
for some positive integers $m_i$ and some pairwise distinct irreducible representations $\phi_i$ of $L_E$, we say that:
\begin{itemize}
	\item $\phi$ is of good parity if $\phi_i$ is (conjugate) self-dual of the same parity as $\phi$ for all $i$;
	\item $\phi$ is tempered if $\phi_i$ is tempered (i.e. has bounded image on the Weil group $W_E$) for all $i$;
	\item $\phi$ is almost tempered if $\phi_i=\phi'_i|\cdot|^{s_i}$ for some tempered representation $\phi'_i$ of $L_E$ and some $s_i\in \mathbb R$ with $|s_i|<\frac{1}{2}$ for all $i$. 
\end{itemize}
Note that if $\phi$ is of good parity, then it is tempered. 
\begin{remark}\label{ramajujan}
	In fact, the Ramanujan conjecture predicts that all localizations of global generic elliptic $A$-parameters are tempered. However, we still need to deal with a larger class of $L$-parameters but not only those are tempered since the Ramanujan conjecture has not been proved in general. Thank to the work of Jacquet-Shalika \cite[Cor.2.5]{MR618323}, for our purpose, it is sufficient to do the comparison for almost tempered $L$-parameters.
\end{remark}

Let $\phi$ be an almost tempered $L$-parameter for $G=G(V)$. Assume that $\dim V<r$ and let
\begin{equation}\label{thetaphi}
	\theta(\phi)=\phi\chi_W^{-1}\chi_V+\chi_V\boxtimes S_{2r-2n+1}.
\end{equation}
Then $\theta(\phi)$ is a local $A$-parameter for $H=H\left(W_{(r)}\right)$. Recall that there is an isomorphism between component groups
\[
\ell:\calS_\phi\lra\calS_{\theta(\phi)}\lra\overline{\calS_{\theta(\phi)}}.
\]
In the light of Proposition \ref{transfermulti} and Remark \ref{ramajujan}, to complete the proof of Theorem \ref{ThmB}, all we need is the following two theorems.
\begin{theorem}\label{T:APack-MultiFree}
	The local $A$-packet $\Pi_{\theta(\phi)}(H)$ is multiplicity free. Hence we can regard $\Pi_{\theta(\phi)}(H)$ as a subset of $\Irr(H)$.
\end{theorem}

\begin{theorem}\label{Local-Compare}
There is a commutative diagram
\begin{equation}\label{Diag-LocalCompare}
	\begin{CD}
	\Pi_{\theta(\phi)}(H) @>\calJ_{\scrW'}>> \widehat{\overline{\calS_{\theta(\phi)}}}\\
	@VV\theta V @VV\ell^*V\\
	\bigsqcup\Pi_\phi^L(G') @>\calJ_\scrW^L>> \widehat{\calS_\phi}
	\end{CD},
\end{equation}
where the disjoint union runs over all pure inner forms of $G$, and the arrow $\theta$ is a bijection given by the theta lift. More precisely, for $\sigma\in \Pi_{\theta(\phi)}(H)$, we have:
\begin{enumerate}
	\item There exists an unique pure inner form $G'$ of $G$, such that the theta lift $\pi=\theta(\sigma)$ of $\sigma$ to $G'$ is non-vanishing.
	\item $\pi$ lies in the $L$-packet $\Pi_{\phi}^{L}(G')$, and 
	\[
		\calJ_\scrW^L(\pi)= \ell^*\left(\calJ_{\scrW'}(\sigma)\right).
	\]
\end{enumerate} 
Hence it follows that $\Pi_{\phi}^{\theta}(G)=\Pi_{\phi}^{L}(G)$ as sets and $\calJ_\scrW^L= \calJ^{\theta}$.
\end{theorem}
\begin{remark}\label{mogelin}~
	\begin{enumerate}
		\item Theorem \ref{T:APack-MultiFree} is largely due to M{\oe}glin \cite{MR2767522} when $F$ is non-Archimedean, to M{\oe}glin-Renard \cite{MR3666056} when $F=\mathbb C$ and to Arancibia-M{\oe}glin-Renard \cite{MR3919547}, M{\oe}glin-Renard \cite{MR3969879} \cite{MR3947270} \cite{MR4092900} when $F=\mathbb R$ and $\phi$ is of good parity. Here we will give a proof of Theorem \ref{T:APack-MultiFree} by using theta lifts.  
		\item When we are in Case U, we also need to deal with split places where $E\simeq F\times F$. In this case, $G(V)$ and $H(W)$ are isomorphic to general linear groups, and $G(V)$ has no pure inner form other than itself. We also know that the local $L$-packet $\Pi^L_{\phi}(G)$ and the local $A$-packet $\Pi_{\theta(\phi)}(H)$ are singletons. Therefore, in this case we need to prove that
		\[
		\theta\left(\pi_\phi\right)=\sigma_{\theta(\phi)}, 
		\]
		where $\pi_\phi$ and $\sigma_{\theta(\phi)}$ are the irreducible representations of $G$ and $H$ associated to $\phi$ and $\theta(\phi)$ respectively. This assertion is an analogy of Theorem \ref{Local-Compare} for general linear groups. When $F$ is non-archimedean, it has been proved by M{\'{\i}}nguez \cite{MR2504432}. When $F$ is Archimedean, it has been proved by Adams-Barbasch \cite{MR1346217} for $F=\mathbb C$ and by Adams \cite{MR1021501} for $F=\mathbb R$.
	\end{enumerate}
\end{remark}
We are going to prove these theorems in the next few subsections. 

\subsection{Some local packets}\label{S:LP}
We review some knowledge of certain local $L$-packets and $A$-packets in this subsection. Let $\phi$ be an almost tempered $L$-parameter. We may write it as
\begin{equation}\label{almosttempered}
	\phi=\varphi+\phi_0+\left(\varphi^c\right)^\vee,
\end{equation}
where:
\begin{itemize}
	\item $\phi_0$ is an $L$-parameter of good parity for $G_0^*=G(V_0^+)$; here $V_0^+$ is a $c$-Hermitian space in the Witt tower containing $V^+$;
	\item $\varphi$ is a summation 
		\[
			\varphi=\phi_1|\cdot|^{s_1}+ \cdots + \phi_r|\cdot|^{s_r}
		\] 
		for some $k_i$-dimensional irreducible tempered representation $\phi_i$ of $L_E$ and non-negative real number $s_i < 1/2$, such that for each $i\in\{1,\cdots,r\}$, either $\phi_i$ is not (conjugate) self-dual of the same parity as $\phi$, or $s_i$ is positive. 
\end{itemize}
Without loss of generality, we can rearrange the index set $\{1,\cdots,r\}$ such that
\[
	\frac{1}{2}> s_1 \geq \cdots \geq s_r \geq 0.
\]
We put $k=k_1+\cdots+k_r$. There is a natural isomorphism $\calS_{\phi_0}\simeq\calS_{\phi}$. By the inductive property of the LLC for $G$, we know that $\Pi^L_{\phi}(G)=\varnothing$ unless the $F$-rank of $G$ is greater or equal to $k$, in which case there exists a $F$-parabolic subgroup $P$ of $G$ with Levi component $\GL_{k}(E)\times G_0$, where $G_0=G(V_0)$ and $V_0$ is a $c$-Hermitian space in the Witt tower containing $V$. Let $\tau$ be the irreducible representation of $\GL_{k}(E)$ associated to $\varphi$. Then by \cite[Thm.4.4(10)]{zou2020local}, \cite[Thm.2.5.1(8)]{chen2020local} %and {\color{red}{\textit{reference for real LLC...}}},\zjl{just say that the compatibility of LLC with Langlands quotient ?} 
we know that for any $\eta\in \wh{\calS_\phi}$, $\pi=\pi(\phi,\eta)$ is the unique irreducible quotient of 
\[
	\Ind_{P}^{G}\left(\tau\boxtimes\pi_0\right),
\]
where $\pi_0=\pi\left(\phi_0,\eta\right)$, and we regard $\eta$ as a character of $\calS_{\phi_0}$ through the natural isomorphism. Indeed we have more than this: 

\begin{lemma}\label{Irr-StdMod}
	The induced representation $\Ind_P^{G}\left(\tau\boxtimes \pi_0\right)$ is irreducible for all $\pi_0 \in \Pi_{\phi_0}\left(G_0\right)$. Hence we have 
	\[
	\Pi^L_{\phi}(G)=\Big\{\Ind_P^G(\tau\boxtimes\pi_0)~\big|~\pi_0\in\Pi_{\phi_0}(G_0)\Big\}.
	\]
\end{lemma}
\begin{proof}
	When $\phi$ is tempered, the assertion in this lemma simply follows from the local intertwining relation for $G$. %{\color{red}{\textit{reference for LIR of real LLC...}}}. \zjl{just say that this follows from the LLC for G ?}
So it remains to show the lemma when $\phi$ is almost tempered but non-tempered. Let $j\in\{1,\cdots,r\}$ be the largest integer such that $s_j>0$. By induction in stages, we may rewrite $\Ind_P^{G}\left(\tau\boxtimes \pi_0\right)$ as a standard module. More precisely, we have
	\[
		\Ind_P^{G}\left(\tau\boxtimes \pi_0\right)= \Ind_{P'}^G \left(\tau'\boxtimes \pi'_0\right),
	\]
	where:
	\begin{itemize}
		\item $P'$ is a parabolic subgroup of $G$ with Levi component $\GL_{k'}(E)\times G(V'_0)$, where $V'_0$ is a $c$-Hermitian space in the Witt tower containing $V$, and $k'=k_1+\cdots+k_j$;
		\item $\tau'$ is the irreducible representation of $\GL_{k'}(E)$ associated to the $L$-parameter 
		\[
		\phi_1|\cdot|^{s_1}+ \cdots + \phi_j |\cdot|^{s_j};
		\]
		\item $\pi'_0$ is some irreducible tempered representation of $G(V'_0)$ with the $L$-parameter
		\[
			\phi'_0=\left(\phi_{j+1}+\cdots +\phi_r\right) + \phi_0 + (\left(\phi_{j+1}+\cdots +\phi_r\right)^c)^\vee.
		\] 
	\end{itemize}
	Hence it would be sufficient to show the irreducibility of $\Ind_{P'}^G \left(\tau'\boxtimes \pi'_0\right)$. If $F$ is Archimedean,then this follows from a result of Speh-Vogan \cite[Thm.1.1]{MR590291}; see also \cite[Sect.8]{MR632407}. If $F$ is non-Archimedean, this has been proved in \cite[Prop.6.7]{MR3708200} for Case O and \cite[Prop.9.1 \& App.B]{MR3573972} for Case U. 
\end{proof}

Next we consider certain local $A$-packets for $H$. Let $W_0=W_{(r-k)}$ and $H_0=H(W_0)$. Put
\begin{equation*}
	\theta\left(\phi_0\right)=\phi_0\chi_W^{-1}\chi_V+\chi_V\boxtimes S_{2r-2n+1}. 
\end{equation*}
Then $\theta\left(\phi_0\right)$ is a local $A$-packet of good parity for $H_0$. According to the decomposition (\ref{almosttempered}) of $\phi$, we have a similar decomposition of $\theta(\phi)$  
\begin{equation*}
	\theta(\phi)=\varphi\chi_W^{-1}\chi_V+\theta(\phi_0)+\left(\left(\varphi\chi_W^{-1}\chi_V\right)^c\right)^\vee,
\end{equation*}
as well as the natural isomorphism $\calS_{\theta(\phi_0)}\simeq\calS_{\theta(\phi)}$. Let $Q$ be the standard parabolic subgroup of $H$ with Levi component $\GL_k(E)\times H_0$. Then by the definition and the inductive property of local $A$-packets (see \cite[p.45(1.5.1)]{MR3135650} and \cite[p.28]{MR3338302}), we have: 
\begin{equation*}
	\Pi_{\theta(\phi)}(H)=
	\Big\{\textit{Irreducible constituents of}~\Ind_Q^H(\tau\chi_W^{-1}\chi_V\boxtimes\sigma_0)~\big|~\sigma_0\in\Pi_{\theta(\phi_0)}(H_0)\Big\}.
\end{equation*}
The following irreducibility result for these induced representations will be used in later proofs. 
\begin{lemma}\label{Irr-nonStdMod}
	The induced representation $\Ind_Q^H(\tau\chi_W^{-1}\chi_V\boxtimes\sigma_0)$ is irreducible for all irreducible unitary representation $\sigma_0$ in the $A$-packet $\Pi_{\theta(\phi_0)}(H_0)$. 
	Hence as a (multi) set, we have
	\[
	\Pi_{\theta(\phi)}(H)=\Big\{\Ind_Q^H(\tau\chi_W^{-1}\chi_V\boxtimes\sigma_0)~\big|~\sigma_0\in\Pi_{\theta(\phi_0)}(H_0)\Big\}.
	\]
\end{lemma}
\begin{proof}
	This assertion has been proved in a more general context by M{\oe}glin in \cite[Sect.3.2]{MR2767522}, and \cite[Sect.5.1]{MR2822218} when $F$ is non-Archimedean; and by M{\oe}glin-Renard in \cite[Sect.6]{MR3666056} when $F=\CC$. To show the case when $F=\RR$, one can appeal to the same argument as that in \cite[Prop.3.4]{MR3907735}; for a sketch of the proof, see Appendix \ref{App.Irr}. 
\end{proof}

\subsection{Reduction to the case of good parity}
We proceed the proof by induction on $\dim V$. 
\begin{lemma}\label{Induction.Hypothesis}
	Theorem \ref{T:APack-MultiFree} and Theorem \ref{Local-Compare} holds for $\dim V=0$, i.e. in Case ${\rm O}$ or Case $\U_0$, and $n=0$.  
\end{lemma}
\begin{proof}
	When $\dim V=0$, we have $\theta(\phi)=\chi_V\boxtimes S_{2r+1}$ (note that if we are in Case O, we have $\chi_V=\mathbbm{1}$ by our conventions), and $\Pi_{\theta(\phi)}(H)=\{\sigma_1\}$, where 
	\[
	\sigma_1= \begin{cases}
		\textit{the trivial representation} \quad & \textit{Case ${\rm O}$},\\
		\chi\circ\det \quad & \textit{Case ${\rm U}$}.
	\end{cases}
	\]
	Here in Case U, $\chi$ is the character of $E^1$ corresponding to $\chi_V$ by the LLC for torus. It is easy to see that $\sigma_1$ is also the theta lift of the trivial representation of $G$. This completes the proof. 
\end{proof}

From now on, we assume that $\dim V>0$. We first reduce the proof of Theorem \ref{T:APack-MultiFree} and Theorem \ref{Local-Compare} to the case of good parity. Given an almost tempered $L$-parameter $\phi$ of $G$, we can write it as 
\[
	\phi=\varphi+\phi_0+\left(\varphi^c\right)^\vee,
\]
where $\phi_0$ is the ``good parity part'', see (\ref{almosttempered}). 

\begin{lemma}\label{Reduction.To.GoodParity}
	Assume that Theorem \ref{T:APack-MultiFree} and Theorem \ref{Local-Compare} hold for $\phi_0$. Then they also hold for $\phi$.
\end{lemma}
\begin{proof}
	Let $\sigma\in\Pi_{\theta(\phi)}(H)$. By Lemma \ref{Irr-nonStdMod} we know that
	\[
	\sigma=\Ind_Q^H(\tau\chi_W^{-1}\chi_V\boxtimes\sigma_0)
	\]
	for some $\sigma_0\in\Pi_{\theta(\phi_0)}(H_0)$. Moreover, if we identify $\calS_{\theta(\phi)}$ with $\calS_{\theta(\phi_0)}$ via the natural isomorphism, then $\calJ_{\scrW'}(\sigma)=\calJ_{\scrW'}(\sigma_0)$. By our hypothesis, there is exactly one pure inner form $G_0'$ of $G_0^*$, such that:
	\begin{itemize}
		\item the theta lift of $\sigma_0$ to $G_0'$, denoted by $\pi_0$, is non-vanishing;
		\item $\calJ_{\scrW}^L(\pi_0)=\ell^*(\calJ_{\scrW'}(\sigma_0))$.
	\end{itemize}
	Then it follows from the induction principle of local theta correspondence \cite{MR818351}, \cite[Cor.3.21]{MR1346217}, \cite[Thm.4.5.5]{MR1658091}, \cite[Sect.5.2]{MR2175409} that there exists a non-zero equivariant map
	\[
	\omega\lra\Ind_{P'}^{G'}(\tau\boxtimes\pi_0)\boxtimes\Ind_Q^H(\tau\chi_W^{-1}\chi_V\boxtimes\sigma_0),
	\]
	where $G'$ is a pure inner form of $G$, and $P'$ is the standard parabolic subgroup of $G'$ with Levi component $\GL_k(E)\times G_0'$. Since $\Ind_{P'}^{G'}(\tau\boxtimes\pi_0)$ is irreducible by Lemma \ref{Irr-StdMod}, this implies that the theta lift of $\sigma$ to $G'$ is 
	\[
	\pi=\Ind_{P'}^{G'}(\tau\boxtimes\pi_0).
	\]
	Moreover, by the LLC for $G'$, if we identify $\calS_\phi$ with $\calS_{\phi_0}$ via the natural isomorphism, then $\calJ_{\scrW}^L(\pi)=\calJ_{\scrW}^L(\pi_0)$. Hence the theta lift between $(G',H)$ gives us the desired commutative diagram (\ref{Diag-LocalCompare}). Moreover, it follows from the combination of our hypothesis, Lemma \ref{Irr-StdMod}, and Lemma \ref{Irr-nonStdMod} that the left vertical arrow $\theta$ is a bijection. In particular, $\Pi_{\theta(\phi)}(H)$ is multiplicity-free. This completes the proof.
\end{proof}

\subsection{A subdiagram}\label{S:Asubdiagram}
So now, by the results in the previous subsection, it remains to prove Theorem \ref{T:APack-MultiFree} and Theorem \ref{Local-Compare} for $L$-parameters of good parity. Let $\phi$ be an $L$-parameter of good parity for $G$. Let $\theta(\phi)$ be the $A$-parameter of $H$ defined in (\ref{thetaphi}), and $\phi_{\theta(\phi)}$ the $L$-parameter associated to $\theta(\phi)$. Since $\phi$ is of good parity, it is also tempered. Hence we know that the $L$-packet $\Pi_{\phi_{\theta(\phi)}}^L(H)$ is a subset of the local $A$-packet $\Pi_{\theta(\phi)}(H)$, and there is a commutative diagram 
\begin{equation*}%\label{localAcontainlocalL2}
	\begin{CD}
		\Pi_{\phi_{\theta(\phi)}}^L(H) @>\calJ_{\scrW'}^L>> \widehat{\overline{\calS_{\phi_{\theta(\phi)}}}}\\
		@VVV @VVV\\
		\Pi_{\theta(\phi)}(H) @>\calJ_{\scrW'}>> \widehat{\overline{\calS_{\theta(\phi)}}} 
	\end{CD}.
\end{equation*}
By abuse of notation, we shall still denote the following composition of natural maps
\[
\calS_\phi\xrightarrow{~\ell~}\overline{\calS_{\theta(\phi)}}\lra\overline{\calS_{\phi_{\theta(\phi)}}}
\]
by the symbol $\ell$. 
\begin{remark}
	In our later proofs, we shall also use the fact that the $A$-packet $\Pi_{\theta(\phi)}(H)$ contains the $L$-packet $\Pi_{\phi_{\theta(\phi)}}^L(H)$ when $\phi$ is almost tempered but may not be tempered. We are not very sure whether this fact has been verified in \cite{MR3135650} \cite{MR3338302} or not, but this fact can be easily shown based on the tempered case, by using Lemma \ref{Irr-nonStdMod}. %and the ``exchange the induction orders'' argument used in Lemma \ref{exchange}.  
\end{remark}
In this subsection, as the first step toward the good parity case, we prove the following proposition, which will establish a ``subdiagram'' of Theorem \ref{Local-Compare}. 
\begin{proposition}\label{Sub-Diagram.Local.Compare}
	There is a commutative diagram 
	\[
	\begin{CD}
		\Pi_{\phi_{\theta(\phi)}}^L(H) @>\calJ_{\scrW'}^L>> \widehat{\overline{\calS_{\phi_{\theta(\phi)}}}}\\
		@VV\theta V @VV\ell^*V\\
		\bigsqcup\Pi_\phi^L(G') @>\calJ_\scrW^L>> \widehat{\calS_\phi}
	\end{CD},
	\]
	where the disjoint union runs over all pure inner forms of $G$, and the arrow $\theta$ is an injection given by the theta lift. More precisely, for $\sigma\in \Pi_{\phi_{\theta(\phi)}}^L(H)$, we have:
\begin{enumerate}
	\item There exists an unique pure inner form $G'$ of $G$, such that the theta lift $\pi=\theta(\sigma)$ of $\sigma$ to $G'$ is non-vanishing.
	\item $\pi$ lies in the $L$-packet $\Pi_{\phi}^{L}(G')$, and 
	\[
		\calJ_\scrW^L(\pi)= \ell^*\left(\calJ_{\scrW'}^L(\sigma)\right).
	\]
\end{enumerate}
\end{proposition}

The ingredient of the proof of this proposition is the so-called Prasad's conjecture. It describes the (almost) equal rank theta lift in terms of the LLC. We shall briefly recall it.

Let $W_0=W_{(n)}$ as defined in (\ref{W-split}), and $H_0=H\left(W_{(n)}\right)$. We shall use symbols $G'$ and $H_0'$ to denote pure inner forms of $G$ and $H_0$. Consider the theta lift between $(G',H_0')$, which is in the almost equal rank, with respect to the datum $(\psi_F,\chi_V,\chi_W)$. To distinguish notations, we use the symbol $\vartheta$ to denote the theta lift between $(G',H_0')$. As we will need to use the LLC for the group $H_0'$, it is necessary for us to specify which Whittaker datum of $H_0$ we are using. Following the description in Section \ref{Whittaker.Data}, we can pick up a Whittaker datum of $H_0$ (depending on the choice of the additive character $\psi_F$), which we shall also denote by $\scrW'$ by abuse of notation. Let
\begin{equation*}%\label{varthetaphi}
	\vartheta(\phi)=\phi\chi_W^{-1}\chi_V+\chi_V
\end{equation*}
be a tempered $L$-parameter for $H_0$. Then the Prasad's conjecture asserts that:
\begin{theorem}\label{Prasad.Conj}~
	\begin{enumerate}
		\item For any pure inner form $H'_0$ of $H_0$, and any $\sigma_0\in\Pi_{\vartheta(\phi)}^L(H'_0)$, there exists an unique pure inner form $G'$ of $G$, such that the theta lift $\pi$ of $\sigma_0$ to $G'$ is non-zero. Moreover, $\pi$ is in the $L$-packet $\Pi_\phi^L(G')$, and
		\begin{equation}\label{E:PrasadConjLabelling}
		\calJ_{\scrW}^L(\pi)=\calJ_{\scrW'}^L(\sigma_0)~\Big|_{\calS_\phi}.
		\end{equation}
		Here we regard $\calS_\phi$ as a subgroup of $\calS_{\vartheta(\phi)}$ via the natural embedding $\calS_\phi\hookrightarrow\calS_{\vartheta(\phi)}$. In particular, if we are in Case ${\rm U}$, then the map
		\[
		\vartheta:\bigsqcup_{H'_0}\Pi_{\vartheta(\phi)}^L(H'_0)\lra\bigsqcup_{G'}\Pi_\phi^L(G')
		\]
		given by the theta lift $\vartheta$ is surjective.
		\item Suppose we are in Case ${\rm O}$. Suppose further that $\chi_W\not\subset\phi$. Then for any pure inner form $G'$ of $G$, and any $\pi\in\Pi_\phi^L(G')$, both
		\[
		\begin{cases}
			\textit{the theta lift }\vartheta(\pi)\textit{ of $\pi$ to } H_0,\\
			\textit{the theta lift }\vartheta(\pi\otimes\det)\textit{ of $\pi\otimes\det$ to } H_0
		\end{cases}
		\]
		are non-zero.
		\item Suppose we are in Case ${\rm U}$, and $F$ is non-Archimedean. Suppose further that $\chi_W\not\subset\phi$. Let $W'_0$ be the unique skew-Hermitian space over $E$ of the same dimension and the opposite sign as $W_0$, and $H'_0=H(W'_0)$. Then for any pure inner form $G'$ of $G$, and any $\pi\in\Pi_\phi^L(G')$, both
		\[
		\begin{cases}
			\textit{the theta lift }\vartheta(\pi)\textit{ of $\pi$ to } H_0,\\
			\textit{the theta lift }\vartheta(\pi)\textit{ of $\pi$ to } H'_0
		\end{cases}
		\]
		are non-zero.
	\end{enumerate}
\end{theorem}
\begin{proof}
	This theorem was proved by many people in various cases. When $F$ is non-Archimedean, this is proved by Atobe-Gan \cite[Thm.4.5]{MR3708200} in Case O, and by Gan-Ichino \cite[Thm.4.4]{MR3573972} in Case U; when $F$ is real, this is essentially proved by Paul \cite[Thm.3.4]{MR1664375} \cite[Thm.15 \& Thm.18]{MR2175409}; when $F$ is complex, this is proved by Adams-Barbasch \cite[Thm.2.8 \& Thm.2.9]{MR1346217}. 
%\trivial[h]{discuss here. Atobe-Gan in the case of non-quasi-split}
\end{proof}

\begin{remark}~
\begin{enumerate}
\item In Case U, the proof of Gan-Ichino \cite[Thm.4.4]{MR3573972} uses the AMF as an input. But one can apply Atobe's method \cite[Sect.7.1]{MR3788848} to this case to avoid using any global method. Indeed, except for the equality (\ref{E:PrasadConjLabelling}), all other statements in Theorem \ref{Prasad.Conj} follows from \cite[Thm.C.5]{MR3166215}. As for the equality (\ref{E:PrasadConjLabelling}), one can consider the diagram
\begin{equation*}
\begin{CD}
\omega\otimes\Ind_Q^{\widetilde{H}'_0}\left(\tau^c\chi_W^c\boxtimes\sigma_0^\vee\right) @>{\calT}>> \Ind_P^{\widetilde{G}'}\left(\tau\chi_V\boxtimes\pi\right)\\
@V{1\otimes R(\wt{w}',\tau^c\chi_W^c\boxtimes\sigma_0^\vee)}VV @VV{R(\wt{w},\tau\chi_V\boxtimes\pi)}V\\
\omega\otimes\Ind_Q^{\widetilde{H}'_0}\left(\tau^c\chi_W^c\boxtimes\sigma_0^\vee\right) @>{\calT}>> \Ind_P^{\widetilde{G}'}\left(\tau\chi_V\boxtimes\pi\right)
\end{CD}
\end{equation*}
constructed in \cite[Sect.8.2]{MR3573972}, where $\tau$ is an irreducible discrete series representation of some general linear group; $\widetilde{G}'$, ${\widetilde{H}'_0}$ are some larger unitary groups; $P$, $Q$ are maximal parabolic subgroups of $\widetilde{G}'$, ${\widetilde{H}'_0}$; and finally $R(\wt{w},\tau\chi_V\boxtimes\pi)$, $R(\wt{w}',\tau^c\chi_W^c\boxtimes\sigma_0^\vee)$ are normalized intertwining operators. This diagram commutes up to an explicitly computable constant. Let $\tau$ runs over all discrete series of general linear groups, by the local intertwining relation one can recover $\calJ_{\scrW}^L(\pi)$ from normalized intertwining operators $R(\wt{w},\tau\chi_V\boxtimes\pi)$, whereas the latter can be computed via the diagram above. Readers may also consult our previous paper \cite[Cor.6.3.2]{chen2020local} for a similar argument (but in a slightly different setting).

\item Here is a caveat: Paul's results \cite[Thm.3.4]{MR1664375} \cite[Thm.15 \& Thm.18]{MR2175409} are written in
terms of the Harish-Chandra parameters rather than the $L$-parameters. Therefore we need a ``dictionary'', which allows us to translate the language of Harish-Chandra parameters into the language of $L$-parameters, so that we can reformulate Paul's results into the form we need. For Case U, such a dictionary has been provided by \cite[App.A]{MR4054051} and \cite[Sect.3.2]{MR4568696}. For Case O, the situation is more complicated. The LLC for real even orthogonal groups has not been established in the literature so far, due to their disconnectedness. 
%By mimicking \cite[App.A]{MR4054051}, we will first provide a dictionary for symplectic groups and special even orthogonal groups in Appendix \ref{App.HCvsL}. Once this dictionary has been provided, we can reformulate Paul's results into the form of Theorem \ref{weakver.Prasad.Conj}, which is a weaker version of Theorem \ref{Prasad.Conj} (see also \cite[Conj.4.6]{MR3788848}). Finally 
In Appendix \ref{App.LLC4O} we will first prove a weaker version of Theorem \ref{Prasad.Conj} (see Theorem \ref{weakver.Prasad.Conj}) using Paul's results, and then define an LLC for real full even orthogonal groups by using theta lifts. By our construction, Theorem \ref{weakver.Prasad.Conj} is upgraded to Theorem \ref{Prasad.Conj} automatically.

\item Indeed, this theorem also holds for all generic $L$-parameters. In particular, it holds for all almost tempered $L$-parameters.
\end{enumerate}
\end{remark}

Now we have three component groups $\calS_{\phi}$, $\calS_{\vartheta(\phi)}$, and $\calS_{\phi_{\theta(\phi)}}$. The natural embeddings between them give us a commutative diagram
\begin{equation*}
	\begindc{\commdiag}[75]
	%\obj(-2,1)[a]{$U(V_{2n+1}^-)$}
	\obj(8,4)[b]{$\overline{\calS_{\phi_{\theta(\phi)}}}$}
	\obj(0,0)[x]{$\overline{\calS_{\vartheta(\phi)}}$}
	\obj(-8,4)[c]{$\calS_{\phi}$}
	%\obj(2,-1)[d]{$U(V_{2n-1}^+)$}
	%\mor{a}{x}{}[\atleft,\solidline]
	\mor{x}{b}{}[\atleft,\solidarrow]
	\mor{c}{x}{}[\atleft,\solidarrow]
	\mor{c}{b}{$\ell$}[\atleft,\solidarrow]
	%\mor{d}{x}{}[\atleft,\solidline]
	\enddc.
\end{equation*}
Notice that 
\[
\phi_{\theta(\phi)}=\left(|\cdot|^{n-r}+\cdots+|\cdot|^{-1}\right)\chi_V+\vartheta(\phi)+\left(|\cdot|^{1}+\cdots+|\cdot|^{r-n}\right)\chi_V.
\]
Hence the natural map $\calS_{\vartheta(\phi)}\lra\calS_{\phi_{\theta(\phi)}}$ is indeed an isomorphism. The component group $\overline{\calS_{\vartheta(\phi)}}$ will serve as a springboard and make our proofs much easier.

\begin{proof}[Proof of Proposition \ref{Sub-Diagram.Local.Compare}]
	Let $\sigma$ be an irreducible representation in the $L$-packet $\Pi_{\phi_{\theta(\phi)}}^L(H)$, corresponding to the character $\eta_\sigma\in\wh{\overline{\calS_{\phi_{\theta(\phi)}}}}$. Let
	\[
	\eta_{\sigma_0}=\eta_\sigma\Big|_{\overline{\calS_{\vartheta(\phi)}}},
	\]
	and $\sigma_0\in\Pi_{\vartheta(\phi)}^L(H_0)$ the irreducible tempered representation corresponding to $\eta_{\sigma_0}$. Then by the LLC for $H$, there is a parabolic subgroup of $H$, say $Q$, with Levi component
	\begin{equation*}
		L\simeq \GL_{1}(E)\times\cdots\times \GL_{1}(E)\times H_0,
	\end{equation*}
	so that $\sigma$ is the unique irreducible quotient of the standard module 
	\begin{equation*}
		\Ind_Q^{H}\left(\chi_V|\det|^{r-n}\boxtimes\cdots\boxtimes\chi_V|\det|^{1}\boxtimes\sigma_0\right).
	\end{equation*}
	According to the Prasad's conjecture Theorem \ref{Prasad.Conj}, there exists an unique pure inner form $G'$ of $G$, such that the theta lift $\pi$ of $\sigma_0$ to $G'$ is non-zero. Moreover, $\pi$ is in the $L$-packet $\Pi_\phi^L(G')$, corresponding to the character
	\[
	\eta_\pi=\eta_{\sigma_0}\Big|_{\calS_\phi}.
	\]
	We claim that the theta lift of $\pi$ to the group $H$ is just $\sigma$. The assertion of Proposition \ref{Sub-Diagram.Local.Compare} then follows from our claim.

	Now we prove our claim. Indeed, when $F$ is non-Archimedean, this claim follows from \cite[Prop.3.2]{MR3502978} directly. When $F$ is Archimedean, by the persistence principle (see \cite[Prop.4.1]{kudla1996notes}) % \cite[Cor.3.21]{MR1346217}, \cite[Thm.4.5.5]{MR1658091}, and \cite[Sect.5.2]{MR2175409}, 
 we know that there is a non-zero map 
	\[
	\Theta(\pi) \longrightarrow \Ind_{Q'}^H\left(\chi_V|\det|^{\frac{n-1-r}{2}}\boxtimes \sigma_0\right),
	\]
	where $\Theta(\pi)$ is the big theta lift of $\pi$ to $H$, and $Q'$ is the standard parabolic subgroup of $H$ with Levi component $\GL_{r-n}(E)\times H_0$. Since the dual-pair $(G',H)$ is in the stable range, it follows from \cite[Thm.A]{MR3305312} that $\Theta(\pi)=\theta(\pi)$ is irreducible. Therefore the map above is injective. Apply both the MVW functor and the contragredient functor to the above map, we deduce that 
	$\theta(\pi)$ is a quotient of $ \Ind_{Q'}^H\left(\chi_V|\det|^{\frac{r-n+1}{2}}\boxtimes \sigma_0\right)$, which is also the unique quotient of the the standard module 
	\begin{equation*}
		\Ind_Q^{H}\left(\chi_V|\det|^{r-n}\boxtimes\cdots\boxtimes\chi_V|\det|^{1}\boxtimes\sigma_0\right).
	\end{equation*}
	This completes the proof of our claim. 
\end{proof}

\subsection{Globalization}
To complete the proof of Theorem \ref{Local-Compare}, we need to appeal to some global methods. Let $\phi$ be an $L$-parameter of good parity for $G=G(V)$. Write 
\begin{equation}\label{E:phiSUM}
\phi=\sum_i \phi_i 
\end{equation}
for some (not necessarily distinct) $n_i$-dimensional irreducible (conjugate) self-dual representations of $L_E$. We shall globalize $\phi$ to a generic elliptic $A$-paramater $\dot \phi$, by globalizing each $\phi_i$ separately. 

We fisrt consider the case when $\phi_i$ is a (conjugate) self-dual character. 
\begin{lemma}\label{Globalization character}
Let $(\dot{F},\dot{E})$ be a pair of number fields, and $u$ be a place of $\dot{F}$, such that $(\dot{F}_u,\dot{E}_u)\simeq (F,E)$. In particular $\dot{E}=\dot{F}$ in Case ${\rm O}$ and $\dot E/\dot F$ is a quadratic extension in Case ${\rm U}$. Let $T$ be a finite set of places of $\dot{F}$ that contain $u$. In Case ${\rm U}$, we also require that $\dot E$ is not split at all $v\in T$. Fix $\kappa=\pm 1$. For each place $v\in T$, let $\chi_v$ be a character of $\GL_1(\dot E_v)$, which is
	\[
	\begin{cases}
		\textit{quadratic}\quad&\textit{in Case ${\rm O}$};\\
		\textit{conjugate self-dual of parity $\kappa$}\quad&\textit{in Case ${\rm U}$}.
	\end{cases}
	\]
	Then there exists an automorphic character $\dot\chi$ of $\GL_1(\AAA_{\dot E})$, which is also (conjugate) self-dual of the same parity as each $\chi_v$, such that $\dot\chi_v=\chi_v$ for all $v\in T$. 
\end{lemma}
\begin{proof}
	For Case O, this is a special case of the Grunwald-Wang theorem (cf. \cite[Chap.X, Thm.5]{MR2467155}). For Case U, this is proved in \cite[Lem.8.8]{MR4372665}.
\end{proof}

Next, we consider the case where $\dim \phi_i=m \geq 2$. 
\begin{lemma}\label{Golobalization higner dimension}
	Let $(\dot{F},\dot{E})$ be a pair of number fields, and $u$ be a place of $\dot{F}$, such that $(\dot{F}_u,\dot{E}_u)\simeq (F,E)$. In particular $\dot{E}=\dot{F}$ in Case ${\rm O}$ and $\dot E/\dot F$ is a quadratic extension in Case ${\rm U}$. Let $T$ be a finite set of places of $\dot{F}$ containing $u$. In Case ${\rm U}$, we also require that $\dot E$ is not split at all $v\in T$. Let $m\geq 2$ be a positive integer and fix $\kappa=\pm 1$. For each place $v\in T$, let $\phi_v$ be an irreducible $m$-dimensional
	\[
	\begin{cases}
		\textit{othogonal representation of $L_{\dot E_v}$}\quad&\textit{in Case ${\rm O}$};\\
		\textit{conjugate self-dual representation of $L_{\dot E_v}$ of parity $\kappa$}\quad&\textit{in Case ${\rm U}$}.
	\end{cases}
	\]
	Then there exists an irreducible cuspidal automorphic representation $\dot\phi$ of $\GL_m(\AAA_{\dot E})$, which is also (conjugate) self-dual of the same parity as each $\phi_v$, such that $\dot\phi_v=\phi_v$ for all $v\in T$. Here, we regard $\phi_v$ as an irreducible representation of $\GL_m(E_v)$ by the LLC for $\GL_m$. 
\end{lemma}
\begin{proof}
%For case O, this is proved in \cite[Lemma 6.7]{ishimoto2024endoscopic}. For case $U$, this is proved in \cite[Lemma 4.3.1]{kaletha2014endoscopic}. 
%The argument we use here is similar to that of \cite[Sect.6.4]{MR3573972}. 
%\iffalse
%{\color{red}\textit{Need to discuss: 1. Shin's results need assumptions, see his paper p.98 and Sect.4.3; 2. What if m is odd? In Case O do we also need to use symplectic group?}} 
A similar statement is proved in \cite[Lem.6.7]{ishimoto2024endoscopic} for Case O and \cite[Lem.4.3.1]{kaletha2014endoscopic} for Case U. Since our setting is slightly different from theirs, we provide a proof here.
%The result will follow an application of \cite[Thm.5.13]{MR3004076}, \cite{MR3135650}, \cite{kaletha2014endoscopic} and \cite{MR3338302}. 
We first consider the case when $m$ is even in Case O or $\kappa=(-1)^{m-1}$ in Case U. Let $\dot V^+$ be the unique $m$-dimensional $c$-Hermitian space satisfying the following properties:
\begin{enumerate}
	\item In Case O, for all $v\in S$, the quadratic character associated to $\dot V^+_v$ is the same as $\det\phi_v$;
	\item For all places $v$ of $\dot F$, the Hasse-Witt invariant (resp. sign) of $\dot V^+_v$ is $+1$, i.e. $\epsilon(\dot V^+_v)=1$.
\end{enumerate}
The existence of such a space is guaranteed by the local-global principle for orthogonal and Hermitian spaces. Let $\dot G=G(\dot V^+)^0$ be the identity component of the isometry group of $\dot V^+$, i.e. $\dot G=\SO(\dot V^+)$ in Case O, and $\dot G=\U(\dot V^+)$ in Case U. Since the space $\dot V^+$ is maximally split, the group $\dot G$ is quasi-split. For each $v\in T$, we pick up an irreducible square-integrable representation $\pi_v \in \Pi^L_{\phi_v}(\dot G_v)$. We also pick up a finite place $w$ of $\dot F$ outside $T$, such that
\[
	\begin{cases}
	\dot V^+ \textit{ is unramified (see \cite[Sect.2.3]{MR3708200})} \quad & \textit{in Case O};\\
	\dot E \textit{ is split at $w$} \quad & \textit{in Case U}. 
	\end{cases}
\]
In the notion of \cite[Thm.5.13]{MR3004076}, we take $S=T\cup \{w\}$ and choose two finite places $v_1,v_2$ of $\dot F$ not contained in $S$. Then by \cite[Thm.5.13]{MR3004076}, there exists an irreducible cuspidal automorphic representation $\dot\pi$ of $\dot G$, such that: 
\begin{enumerate}
	\item For all $v\in T$, $\dot\pi_v=\pi_v$;
	\item In Case O, $\dot\pi_w$ is tempered unramified;
	\item In Case U, $\dot\pi_w$ is a supercuspidal representation of $\dot G_w\simeq \GL_m(\dot F_w)$.
\end{enumerate}
Based on the results of Arthur \cite{MR3135650} and Mok \cite{MR3338302}, the cuspidal representation $\dot\pi$ has an elliptic $A$-parameter $\dot\phi$. The conditions we put on the place $w$ guarantee that $\dot\phi$ is generic. Moreover, we know that
\[
	\dot \phi_v=\phi_v
\]
for all $v\in S$, since the localization $\dot\pi_v=\pi_v$ lies in both $\Pi_{\dot \phi_v}^L(\dot G_v)$ and $\Pi_{\phi_v}^L(\dot G_v)$, and different local $L$-packets are disjoint with each other. From the irreducibility of $\phi_v$, we can further conclude that $\dot \phi$ is simple. Therefore, $\dot\phi$ is an irreducible cuspidal automorphic representation of $\GL_m(\AAA_{\dot E})$ that satisfies all our requirements. 

Next we consider the case when $m$ is odd in Case O or $\kappa=(-1)^{m}$ in Case U. In Case U, we choose a conjugate symplectic character $\chi_v$ at each $v\in T$. Then we can globalize $\phi_v$ by globalizing both $\phi_v\chi_v$ (which is of parity $\kappa=(-1)^{m-1}$ now) and $\chi_v$ (using \Cref{Globalization character}). Similarly in Case O, we can twist $\phi_v$ by a quadratic character and assume  $\det \phi_v=1$ for all $v\in T$. Then a similar argument for the previous case works here by replacing $\dot V^+$ by $\dot W^+$, where $\dot W^+$ is the unique $(m-1)$-dimensional symplectic space. 
%\begin{enumerate}
%	\item In Case O, for all $v\in S$, the quadratic character associated to $\dot V^+_v$ is the same as $\det\phi_v$;
%	\item In case $U$, for all places $v$ of $\dot F$, the Hasse-Witt invariant (resp. sign) of $\dot V^+_v$ is $+1$, i.e. $\epsilon(\dot V^+_v)=1$.
%\end{enumerate}
This completes the proof of this lemma.
%\fi 
\end{proof}

As an application of the previous two lemmas, we are now able to globalize an $L$-parameter $\phi$ of good parity. 
\begin{corollary}\label{Globalization-2}
	Let $\phi$ be an $L$-parameter of good parity for $G=G(V)$. Then there exists a tuple of data $(\dot{F},\dot{E},\dot{V},\dot{\phi},u_1,u_2,w)$, where
	\begin{itemize}
		\item $\dot{F}$ is a number field, and $\dot{E}$ is either $\dot F$ itself or a quadratic extension of $\dot{F}$, according to our cases; if we are in Case ${\rm U}$ and $F$ is non-Archimedean, we may choose $\dot{F}$ to be totally imaginary; 
		\item $\dot{V}$ is a vector space over $\dot{E}$, equipped with a non-degenerate Hermitian $c$-sesquilinear form;
		\item $\dot{\phi}$ is a generic elliptic $A$-parameter for $\dot{G}=G(\dot{V})$;
		\item $u_1,u_2,w$ are places of $\dot F$, and $w$ is finite; 
	\end{itemize}
	such that the following conditions hold:
	\begin{enumerate}
		\item $(\dot{F}_{u_1},\dot{E}_{u_1},\dot{V}_{u_1},\dot{\phi}_{u_1})\simeq(\dot{F}_{u_2},\dot{E}_{u_2},\dot{V}_{u_2},\dot{\phi}_{u_2})\simeq(F,E,V,\phi)$;
		\item If we are in the Case ${\rm U}$, then $\dot E_w/\dot F_w$ is a quadratic field extension;
		\item $\dot\phi_w$ is a discrete $L$-parameter for $\dot G_w=G(\dot V_w)$; furthermore, we may choose $\dot\phi_w$ so that it does not contain a given character $\chi_w$ of $L_{\dot{E}_{w}}$;
		\item The localization maps $\calS_{\dot{\phi}}\lra\calS_{\dot{\phi}_{u_1}}$ and $\calS_{\dot{\phi}}\lra\calS_{\dot{\phi}_{u_2}}$ agree, and they are surjections;
		\item At the place $w$, the localization map $\calS_{\dot{\phi}}\lra\calS_{\dot{\phi}_{w}}$ is an isomorphism.
	\end{enumerate}
\end{corollary}
\begin{proof}
Firstly we choose a pair of number fields $(\dot E,\dot{F})$, together with three places $u_1,u_2,w$ of $\dot F$, satisfying the conditions below: 
\begin{enumerate}
	\item $(\dot E_{u_1},\dot F_{u_1})\simeq(\dot E_{u_2},\dot F_{u_2})\simeq(E,F)$; 
	\item If we are in Case O, then $\dot F_w$ is a finite extension of $\QQ_2$ with a sufficiently big residue field; % (this condition guarantees we will have enough irreducible orthogonal representations of $W_{\dot F_w}$; see Appendix \ref{self.dual.Gal.rep});
	\item If we are in Case U, then $\dot F_w$ is a finite extension of $\QQ_p$ for some $p\neq2$ with a sufficiently big residue field, and $\dot E_w$ is a ramified quadratic field extension of $\dot F_w$; if further that $F$ is non-Archimedean, then $\dot{F}$ is totally imaginary. %(this condition guarantees we will have enough irreducible conjugate self-dual representations of $W_{\dot E_w}$, with any given parity; see Appendix \ref{conj.self.dual.gal.rep})
\end{enumerate} 
The existence of such pair of number fields will be proved in Appendix \ref{Existence.NumberField}.
	We write $\phi=\sum_i\phi_i$ as in (\ref{E:phiSUM}). By the results in Appendix \ref{self.dual.Gal.rep} and Appendix \ref{conj.self.dual.gal.rep}, there are sufficiently many irreducible (conjugate) self-dual representations of $W_{\dot E_w}$. Hence we can pick up irreducible (conjugate) self-dual representations $\phi_{w,i}$ of $L_{\dot E_w}$ of the same dimension and parity as $\phi_i$ for each $i$, such that $\phi_{w,i}\not\simeq\phi_{w,i'}$ whenever $i\neq i'$. Moreover, we can further require that all $\phi_{w,i}$ are different from a given character $\chi_w$ of $L_{\dot{E}_{w}}$. 
	
	Let $u=u_1$ and $T=\{u_1,u_2,w\}$. By Lemma \ref{Globalization character} and Lemma \ref{Golobalization higner dimension}, we can globalize each $\phi_i$ to an irreducible cuspidal representation $\dot\phi_i$ of $\GL_{n_i}(\AAA_{\dot E})$, which is (conjugate) self-dual of appropriate parity, such that 
	\begin{enumerate}
		\item $(\dot\phi_i)_{u_1}=(\dot\phi_i)_{u_2}=\phi_{i}$;
		\item $(\dot\phi_i)_{w}=\phi_{w,i}$. 
	\end{enumerate}
	Let
	\[
	\dot\phi=\sum_i\dot\phi_i.
	\]
	It follows from the local-global principle for orthogonal and Hermitian spaces that there exists %{\color{red}{\textit{why? In Case O need $\det\phi=\chi_{\dot V}$...}}} 
 a $\dot V$ over $\dot E$, equipped with a non-degenerate Hermitian $c$-sesquilinear form, such that
	\begin{enumerate}
		\item $\dot V_{u_1}\simeq\dot V_{u_2}\simeq V$;
		\item $\dot\phi$ is a generic elliptic $A$-parameter for $\dot G=G(\dot V)$.
	\end{enumerate}
	Then the tuple of data $(\dot{F},\dot{E},\dot{V},\dot{\phi},u_1,u_2,w)$ satisfies all our requirements.
\end{proof}

\subsection{Multiplicity-freeness}\label{SS:Multi-free}
Let $\phi$ be an $L$-parameter of good parity for $G$ and $\theta(\phi)$ be the local $A$-parameter for $H$ as in (\ref{thetaphi}). Recall that Theorem \ref{T:APack-MultiFree} asserts that the local $A$-packet $\Pi_{\theta(\phi)}(H)$ is multiplicity-free, and it has been proved by M{\oe}glin and many others. In this subsection, we shall give an independent and self-contained proof by using global methods and theta lifts.

Let $(\dot{F},\dot{E},\dot{V},\dot{\phi},u_1,u_2,w)$ be as given in Corollary \ref{Globalization-2}. In Case U, we fix a trace zero element $\delta\in \dot E^\times$. Let $\dot W$ be the
\begin{equation*}%\label{Wsplit}
	\begin{cases}
		2r\textit{-dimensional}\quad &\textit{Case ${\rm O}$};\\
		(2r+1)\textit{-dimensional}\quad &\textit{Case $\U_0$};\\
		(2r+2)\textit{-dimensional}\quad &\textit{Case $\U_1$}\\
	\end{cases}
\end{equation*}
$c$-skew-Hermitian space over $\dot E$ as in (\ref{W-split}), and $\dot H=H(\dot W)$. We also let $\dot W_0$ be the
\begin{equation*}%\label{W0}
	\begin{cases}
		\dim V\textit{-dimensional}\quad &\textit{Case ${\rm O}$};\\
		(\dim V+1)\textit{-dimensional}\quad &\textit{Case ${\rm U}$ and $F$ is non-Archimedean};\\
		(\dim V-1)\textit{-dimensional}\quad &\textit{Case ${\rm U}$ and $F$ is real}
	\end{cases}
\end{equation*}
$c$-skew-Hermitian space over $\dot E$ in the Witt tower containing $\dot W$. Let $W_0=\dot{W}_{0,u_1}$. We put $\dot H_0=H(\dot W_0)$ and $H_0=H(W_0)$. We will use symbols $\dot G'$, $\dot H'_0$, and $H'_0$ to denote pure inner forms of $\dot G$, $\dot H_0$, and $H_0$ respectively.

Let $\psi_{\dot F}$ be a nontrivial additive character of $\mathbb A/\dot F$, such that $\psi_{\dot F,u}$ is in the $F^{\times 2}$-orbit of $\psi_F$ for $u\in\left\{u_1,u_2\right\}$. We shall use the additive charater $\psi_{\dot F}$ to fix Whittaker data $\scrW$ and $\scrW'$ of $\dot G$ and $\dot H$ (and also $\dot H_{0}$), as we have described in Section \ref{Whittaker.Data}. We also globalize characters $(\chi_V,\chi_W)$ of $E^\times$ to a pair of characters of $\dot E^\times\backslash\AAA_{\dot E}^\times$ as follows:
\[
\chi_{\dot V}=\begin{cases}
	\textit{the quadratic character associated to }\dot V\quad &\textit{Case ${\rm O}$};\\
	\textit{a character such that }\chi_{\dot V}|_{\AAA_{\dot F}^\times}=\omega_{\dot E/\dot F}^{\dim V}\quad &\textit{Case ${\rm U}$},
\end{cases}
\]
and
\[
\chi_{\dot W}=\begin{cases}
	\textit{the trivial character of }{\dot F}^\times\backslash\AAA_{\dot F}^\times\quad &\textit{Case ${\rm O}$};\\
	\textit{a character such that }\chi_{\dot W}|_{\AAA_{\dot F}^\times}=\omega_{\dot E/\dot F}^{\dim W}\quad &\textit{Case ${\rm U}$}.
\end{cases}
\]  
We will consider the theta lift between $(\dot G',\dot H)$, which is in the stable range case, with respect to $(\psi_{\dot F},\chi_{\dot V},\chi_{\dot W})$. According to Corollary \ref{Globalization-2}, we may globalize $\phi$ suitably so that $\dot\phi_w$ does not contain the character $\chi_{\dot W,w}$, and we will henceforth assume this. Besides, we will also consider the theta lift between $(\dot G',\dot H'_0)$, which is in the almost equal rank case, with respect to some auxiliary data $(\psi_{\dot F},\chi'_{\dot V},\chi'_{\dot W})$. In Case O, there is no flexibility of choosing such data. However in Case U, we shall choose $(\chi'_{\dot V},\chi'_{\dot W})$ suitably in later proofs depending on our needs. Indeed, the flexibility of choosing $(\chi'_{\dot V},\chi'_{\dot W})$ is a key point in our proof for Case U. To distinguish notations, we use the symbol $\theta$ to denote the theta lift between $(\dot G',\dot H)$, and use the symbol $\vartheta$ to denote the theta lift between $(\dot G',\dot H'_0)$.
We first show the following:
\begin{lemma}\label{Global-MultiOne}
	Let $\dot G'$ be a pure inner form of $\dot G$, and $\dot\pi$ be an irreducible representation of $\dot G'(\AAA)$, such that: 
	\begin{enumerate}
		\item the $L$-parameter of $\dot\pi_v$ is $\dot\phi_v$ for almost all $v$;
		\item $\dot\pi_w$ is in the $L$-packet $\Pi_{\dot\phi_w}^L(\dot G'_w)$. 
	\end{enumerate}
	Then $m_{\disc}(\dot\pi)\leq 1$. Moreover, if $m_{\disc}(\dot\pi)=1$, then we have $\dot\pi_{u_1}\in\Pi_{\phi}^L(\dot G'_{u_1})$. 
\end{lemma}
\begin{proof}
	If $m_{\disc}(\dot\pi)=0$, then our conclusions hold. So we may assume $m_{\disc}(\dot\pi)\geq 1$ in the rest of the proof. Since $\dot \phi$ is generic, by Proposition \ref{cusp.realization}, we know that
	\[
	m_{\cusp}(\dot\pi)=m_{\disc}(\dot\pi)=m(\dot\pi),  
	\]
	and any realization $\calV$ of $\dot\pi$ in $\calA(\dot G')$ lies in $\calA_{\cusp}(\dot G')$. We will prove this lemma by considering (automorphic) theta lifts between $(\dot G', \dot H_0')$ for pure inner forms $\dot H'_0$ of $\dot H_0$, with respect to the datum $(\psi_{\dot F},\chi'_{\dot V},\chi'_{\dot W})$. We do it case by case, and the choice of $(\chi'_{\dot V},\chi'_{\dot W})$ will be specified later in each case. \\
	
	\underline{Case O:} In this case, there is no other pure inner form of $\dot H_0$, neither the flexibility of choosing the auxiliary datum. By the local theta correspondence for unramified representations, we know that $\vartheta(\dot\pi_v)\neq 0$ for almost all place $v$ of $\dot F$. Hence 
	\[
	T'=\left\{v\textit{ place of }\dot F~\big|~\vartheta(\dot\pi_v)=0\right\}
	\]
	is a finite set. Let
	\[
	T=\begin{cases}
		T'\quad&\textit{if $|T'|$ is even};\\
		T'\cup\{w\}\quad&\textit{if $|T'|$ is odd}
	\end{cases}
	\]
	and
	\[
	\dot\pi\otimes{\det}_T=\dot\pi\otimes\left(\bigotimes_{v\in T}{\det}_v\right).
	\]
	Then  
	\[
	m_{\cusp}(\dot\pi\otimes{\det}_T)=m_{\cusp}(\dot\pi)\geq 1. 
	\]
	Indeed, any realization $\calV$ of $\dot\pi$ in $\calA_{\cusp}(\dot G')$ gives a realization
	\[
	\calV\otimes{\det}_T=\left\{f\otimes\left(\bigotimes_{v\in T}{\det}_v\right)~\Bigg|~f\in\calV\right\}
	\]
	of $\dot\pi\otimes{\det}_T$ in $\calA_{\cusp}(\dot G')$, and vice versa. It follows from the conservation relation \cite{MR3369906} that $\vartheta\left(\dot\pi_v\otimes{\det}_v\right)\neq0$ for all places $v\in T'$. Moreover, since $\mathbbm 1\not\subset \dot \phi_w$, it follows from Theorem \ref{Prasad.Conj}(2) that both $\vartheta\left(\dot\pi_w\right)$ and $\vartheta\left(\dot\pi_w\otimes{\det}_w\right)$ are non-zero. Hence 
	\begin{equation*}%\label{smallthetanonzero}
		\vartheta\left(\left(\dot\pi\otimes{\det}_T\right)_v\right)\neq0 \quad\textit{for all places $v$ of $\dot F$}.
	\end{equation*}
	For any realization $\calV$ of $\dot\pi\otimes{\det}_T$ in $\calA_{\cusp}(\dot G')$, consider the automorphic theta lift $\vartheta^{aut}(\calV)$ of $\calV$ to $\dot H_0$. Again, since $\mathbbm 1\not\subset \dot \phi_w$, we know that $\vartheta(\pi_w)$ and $\vartheta(\pi_w\otimes\det_w)$ are the first occurrence of $\pi_w$ and $\pi_w\otimes\det_w$. Hence globally, $\vartheta^{aut}(\calV)$ is either zero or the first occurrence of $\calV$ in the Witt tower containing $\dot W_0$. This implies that $\vartheta^{aut}(\calV)$ is cuspidal. We would like to show that $\vartheta^{aut}(\calV)$ is indeed non-zero. To show this, we investigate the $L$-function $L\left(s,\dot\pi\otimes{\det}_T\right)$.

	By a result of Jacquet-Shalika \cite{MR432596}, the (full) $L$-function
\[
	L\left(s,\dot\phi\right)
\]
	is holomorphic and non-zero at $s=1$. Since $\dot \phi_v$ is almost tempered for every places $v$ of $\dot F$, we know that the local $L$-factors $L\left(s,\dot \phi_v\right)$ are holomorphic when $\Re s\geq 1/2$. Hence the partial $L$-functions
	\begin{equation*}
		L^S\left(s,\dot\pi\otimes{\det}_T\right)=L^S\left(s,\dot\phi\right)
	\end{equation*}
	are also holomorphic and non-zero at $s=1$. Here $S$ is a sufficiently large finite set of places of $\dot F$, and $L^S\left(s,\dot\pi\otimes{\det}_T\right)$ is the partial $L$-function of $\dot\pi\otimes{\det}_T$ relative to the standard representation of $\Lgp{\dot G'}$. On the other hand, since the local $L$-factors $L\left(s,\dot \pi_v\right)$ and $L\left(s,\dot \pi_v\otimes \det_v\right)$ can never have a zero, we know that the complete $L$-function $L\left(s,\dot\pi\otimes{\det}_T\right)$ is also non-zero at $s=1$. Finally we claim that this $L$-function must be holomorphic at $s=1$. Suppose on the contrary that it has a pole at $s=1$. Then \cite[Thm.10.1]{MR3211043} asserts that $\calV$ has non-zero automorphic theta lift to some symplectic group $H(\dot W_-)$, where $\dot W_-$ is a symplectic space of dimension strictly less than $\dot W_0$. This contradicts with the fact that $\vartheta(\pi_w)$ and $\vartheta(\pi_w\otimes\det_w)$ are the first occurrence of $\pi_w$ and $\pi_w\otimes\det_w$. Hence we know that  
	\[
	L\left(s,\dot\pi\otimes{\det}_T\right)
	\]
	is holomorphic and non-zero at $s=1$. It then follows from the Rallis inner product formula \cite[Thm.1.3]{MR3279536} \cite[Thm.10.3]{MR3211043} that the automorphic theta lift $\vartheta^{aut}(\calV)$ is non-vanishing. 

	Let $\dot\sigma_0=\vartheta^{abs}(\dot\pi\otimes{\det}_T)$ be the abstract theta lift to $\dot H_0'$. Then by the multiplicity preservation \cite[Prop.2.6]{MR2402681}, we have
	\[
	m_{\disc}(\dot\sigma_0)\geq m_{\cusp}(\dot\sigma_0)\geq m_{\cusp}(\dot\pi\otimes{\det}_T).
	\]
	Also, it follows from the local theta correspondence for unramified representations that $\dot\sigma_0$ is an irreducible summand of $L_{\vartheta(\dot\phi)}^2(\dot H_0)$, where
	\[
	\vartheta(\dot\phi)=\dot\phi(\chi'_{\dot W})^{-1}\chi'_{\dot V}+\chi'_{\dot V}.
	\]
	Since $\vartheta(\dot\phi)$ is generic, the AMF Theorem \ref{H-ThmB} for $\dot H_0$ implies that
	\[
	m_{\disc}(\dot\sigma_0)=1.
	\]
	Thus, combining these (in)equalities, we get $m_{\disc}(\dot \pi)=1$. Moreover, by the AMF Theorem \ref{H-ThmB}, we also have  $\dot\sigma_{0,u_1}\in\Pi_{\vartheta(\dot\phi)_{u_1}}^L(\dot H'_{0,u_1})$. It then follows from Theorem \ref{Prasad.Conj} and Remark \ref{LLC}(1) that $\dot\pi_{u_1}\in\Pi_{\phi}^L(\dot G'_{u_1})$.\\
	
	\underline{Case $\U_0$, and $F$ is non-Archimedean:} Recall that in Corollary \ref{Globalization-2}, $\dot F$ is chosen to be a totally imaginary number field in this case. We let 
	\[
	(\chi'_{\dot V},\chi'_{\dot W})=(\chi_{\dot V},\chi_{\dot W}). 
	\]
	For each place $v$ of $\dot F$, by the conservation relation \cite{MR3369906}, there is a skew-Hermitian space $W'_{0,v}$ of the same dimension as $\dot W_{0,v}$, such that 
	\[
	\vartheta(\dot\pi_v)\neq0.
	\]  
	Here $\vartheta(\dot\pi_v)$ is the theta lift of $\dot\pi_v$ to $H'_{0,v}$ with respect to $(\psi_{F,v}, \chi'_{\dot V,v},\chi'_{\dot W,v})$. Since $\chi'_{\dot W,w}\not\subset\dot\phi_w$, it follows from Theorem \ref{Prasad.Conj} that we will have two choices of the skew-Hermitian space $W'_{0,w}$ at the place $w$. The flexibility at the place $w$ allows us to pick these local skew-Hermitian spaces coherently such that they form a global skew-Hermitian space $\dot W'_0$ over $\dot E$. Let $\dot H_{0}'=H(\dot W'_0)$.

	The rest of the proof in this case is similar to Case O. Let $\dot\sigma_0=\vartheta^{abs}(\dot\pi)$ be the abstract theta lift to $\dot H_0'$. Then one can show that 
	\[
	m_{\disc}(\dot\sigma_0)\geq m_{\cusp}(\dot\sigma_0)\geq m_{\cusp}(\dot\pi), 
	\]
	and $\dot\sigma_0$ is an irreducible summand of $L_{\vartheta(\dot\phi)}^2(\dot H_0)$, where
	\[
	\vartheta(\dot\phi)=\dot\phi(\chi'_{\dot W})^{-1}\chi'_{\dot V}+\chi'_{\dot V}.
	\]
	Applying Proposition \ref{AMF-U_1}, we get
	\[
	m_{\disc}(\dot\sigma_0)=1.
	\]
	Thus, combining these (in)equalities, we get $m_{\disc}(\dot\pi)=1$. Moreover, by Proposition \ref{AMF-U_1}, we also have  $\dot\sigma_{0,u_1}\in\Pi_{\vartheta(\dot\phi)_{u_1}}^L(\dot H'_{0,u_1})$. It then follows from Theorem \ref{Prasad.Conj} that $\dot\pi_{u_1}\in\Pi_{\phi}^L(\dot G'_{u_1})$.\\

	\underline{Case $\U_1$, and $F$ is non-Archimedean:} In this case, $\dot F$ is a totally imaginary number field. Then the lemma follows from the Proposition \ref{AMF-U_1} directly.\\
	
	\underline{Case U, and $F$ is real:} In this case, we prove by induction on $\dim V$. Recall that $\phi$ is an $L$-parameter of good parity for $G$, so it must be of the form
	\[
	\phi=m_1\chi_1+\cdots+m_r\chi_r,
	\]
	where $\chi_i$ is a conjugate self-dual character of $L_{\CC}=\CC^\times$, and $m_i$ is some positive integer. Recall that in the proof of Corollary \ref{Globalization-2}, to globalize the $L$-parameter $\phi$, we globalize each irreducible constituent of $\phi$ separately, and then added them together. Hence $\dot\phi$ is a summation of one dimensional automorphic characters in this case. We pick up a character $\dot\chi$ such that $\dot\chi\subset\dot\phi$, and set 
	\[
	\left(\chi'_{\dot V},\chi'_{\dot W}\right)=\left(\chi_{\dot V},\dot\chi\right). 
	\]
	Since $\dot\phi$ is generic and $\chi'_{\dot W}\subset\dot\phi$, we know that
	\[
	L^S\left(s,\dot\pi\times(\chi'_{\dot W})^{-1}\right)=L^S\left(s,\dot\phi(\chi'_{\dot W})^{-1}\right)
	\]
	is holomorphic when $\Re(s)>1$ and has a pole at $s=1$, where $S$ is a sufficiently large finite set of places of $\dot F$, and $L^S\left(s,\dot\pi\times(\chi'_{\dot W})^{-1}\right)$ is the partial $L$-function associated to $\dot\pi$ and $(\chi'_{\dot W})^{-1}$. Hence the complete $L$-function $L\left(s,\dot\pi\times(\chi'_{\dot W})^{-1}\right)$ is also holomorphic when $\Re(s)>1$ and has a pole at $s=1$, because the local $L$-factors $L\left(s,\dot\pi_v\times(\chi'_{\dot W,v})^{-1}\right)$ are holomorphic and non-zero when $\Re(s)>1/2$. It then follows from the Rallis inner product formula \cite[Thm.7.2.5]{MR1289491} \cite[Thm.10.1]{MR3211043} that there exists a pure inner form $\dot H'_0=H\left(\dot W'_0\right)$ of $\dot H_0$, such that for any realization $\calV$ of $\dot\pi$ in $\calA_{\cusp}(\dot G')$, we have
	\[
	\vartheta^{aut}(\calV)\neq0,
	\]
	where $\vartheta^{aut}(\calV)$ is the automorphic theta lift of $\calV$ to $\dot H'_0$. Moreover, $\vartheta^{aut}(\calV)$ is the first occurrence of $\calV$ in the Witt tower containing $\dot W'_0$. This implies that $\vartheta^{aut}(\calV)$ is cuspidal. 

	Let $\dot\sigma_0=\vartheta^{abs}(\dot\pi)$ be the abstract theta lift to $\dot H'_0$. Then by the multiplicity preservation \cite[Prop.2.6]{MR2402681}, we have
	\[
	m_{\disc}(\dot\sigma_0)\geq m_{\cusp}(\dot\sigma_0)\geq m_{\cusp}(\dot\pi).
	\]
	Also, it follows form the local theta correspondence for unramified representations that $\dot\sigma_0$ is an irreducible summand of $L_{\vartheta(\dot\phi)}^2(\dot H_0)$, where
	\[
	\vartheta(\dot\phi)=\left(\dot\phi-\chi'_{\dot W}\right)(\chi'_{\dot W})^{-1}\chi'_{\dot V}.
	\]
	Since $\dot H'_0$ is an unitary group of $\left(\dim V-1\right)$-variables, by the induction hypothesis, the lemma holds for $\dot H'_0$. Hence, we have 
\[
	m_{\disc}(\dot\sigma_0)=1
\]
	Thus, combining these (in)equalities, we get $m_{\disc}(\dot{\pi})=1$. Moreover, by the induction hypothesis, we also have $\dot\sigma_{0,u_1}\in\Pi_{\vartheta(\dot\phi)_{u_1}}^L(\dot H'_{0,u_1})$. It then follows from Theorem \ref{Prasad.Conj} that $\dot\pi_{u_1}\in\Pi_{\phi}^L(\dot G'_{u_1})$.
\end{proof}

For any irreducible unitary representation $\sigma$ of $H$ and any character $\eta$ of $\overline{\calS_{\theta(\phi)}}$, we define the multiplicity $m(\sigma,\eta)$ by
\[
m(\sigma,\eta)=\dim\Hom_{\overline{\calS_{\theta(\phi)}}\times H}\Big(\eta\boxtimes\sigma, \Pi_{\theta(\phi)}(H)\Big).
\]
\begin{proposition}\label{Multi-Free.N.Inj}~
	\begin{enumerate}
		\item Let $\sigma$ be an irreducible unitary representation $\sigma$ of $H$. Then, for any character $\eta$ of $\overline{\calS_{\theta(\phi)}}$, we have
		\[
		m(\sigma,\eta)\leq 1,
		\]
		with equality for at most one $\eta$. Hence $\Pi_{\theta(\phi)}(H)$ is multiplicity-free.
		\item The theta lift between $(G',H)$ for all pure inner forms $G'$ of $G$ defines an injection
		\[
		\theta:\Pi_{\theta(\phi)}(H)\lra\bigsqcup_{G'}\Pi^L_\phi(G'),
		\] 
		where the disjoint union runs over all pure inner forms of $G$.
	\end{enumerate}
\end{proposition}
\begin{proof}
	Assume that $m(\sigma,\eta)>0$ for some $\eta$. Let $\dot \phi, \chi_{\dot V},\chi_{\dot W}, 
	\dot G$ and $\dot H$ be as given at the beginning of this subsection, and 
	\[
	\theta(\dot\phi)=\dot\phi\chi_{\dot W}^{-1}\chi_{\dot V}+\chi_{\dot V}\boxtimes S_{2r-2n+1}
	\]
	be an elliptic $A$-parameter for $\dot H$. Since $\dot\phi$ is generic, it follows from Lemma \ref{Compare-epsilon} and equation (\ref{epsilon-generic}) that $\epsilon_{\theta(\dot\phi)}$ is trivial. We define an abstract irreducible representation $\dot\sigma=\bigotimes_v\dot\sigma_v$ of $\dot H(\AAA)$ as follows:
	\begin{itemize}
		\item $\dot\sigma_{u_1}=\dot\sigma_{u_2}=\sigma$;
		\item at place $v\notin\{u_1,u_2\}$, $\dot\sigma_v$ is the irreducible representation in the $L$-packet $\Pi_{\phi_{\theta(\dot\phi)_v}}^L(\dot H_v)$ associated to the trivial character of $\overline{\calS_{\phi_{\theta(\dot\phi)_v}}}$.
	\end{itemize}
	By Theorem \ref{H-ThmB}, we have an embedding
	\[
	\left(\bigoplus_{\eta\in\wh{\overline{\calS_{\theta(\phi)}}}}\bigg(m(\sigma,\eta)\sigma\otimes m(\sigma,\eta)\sigma\bigg)\right)\otimes\left(\bigotimes_{v\notin\{u_1,u_2\}}\dot\sigma_v\right)\xhookrightarrow[~]{~~~~} L_{\theta(\dot\phi)}^2(\dot H).
	\]
	In particular
	\[
	m_{\disc}(\dot\sigma)\geq\sum_{\eta\in\wh{\overline{\calS_{\theta(\phi)}}}}m(\sigma,\eta)^2>0.
	\]
	Moreover, it follows from Remark \ref{Multi-Preserve-PIF} that there exists a pure inner form $\dot G'$ of $\dot G$, and an irreducible summand $\dot\pi$ of $L_{\dot\phi}^2(\dot G')$, such that $\dot\sigma=\theta^{abs}(\dot\pi)$, and
	\[
	m_{\disc}(\dot\pi)=m_{\disc}(\dot\sigma).
	\]
	It follows from our construction and Proposition \ref{Sub-Diagram.Local.Compare} that $\dot\pi_w\in\Pi_{\dot\phi_w}^L(\dot G'_w)$. Hence we have 
	\[
	m_{\disc}(\dot \pi)\leq 1
	\]
	by Lemma \ref{Global-MultiOne}. Thus, combining these (in)equalities, we obtain
	\[
	1\geq\sum_{\eta\in\wh{\overline{\calS_{\theta(\phi)}}}}m(\sigma,\eta)^2.
	\]
	Hence the first statement hold.
	
	For the second statement, notice that if $\sigma\in\Pi_{\theta(\phi)}(H)$, Lemma \ref{Global-MultiOne} also asserts that
	\[
	\theta(\sigma)=\dot\pi_{u_1}\in\Pi_{\phi}^L(\dot G'_{u_1}).
	\]
	Hence it follows from the the conservation relation \cite{MR3369906} that the theta lift between $(G',H)$ for all pure inner forms $G'$ of $G$ gives a well-defined map
	\[
	\theta:\Pi_{\theta(\phi)}(H)\lra\bigsqcup_{G'}\Pi^L_\phi(G'),
	\] 
	where the disjoint union runs over all pure inner forms of $G$. By the Howe duality, this map is an injection. This completes the proof.
\end{proof}

\subsection{The last jigsaw piece}\label{thelastjig}
We retain the notations in the last subsection. After proving Proposition \ref{Multi-Free.N.Inj}, we know that $\Pi_{\phi}^{\theta}(G)\subset \Pi_{\phi}^{L}(G)$ as sets. To finish the proof of Theorem \ref{Local-Compare}, we only need to show that:
\begin{proposition}\label{Final.Jigsaw}
	For any pure inner form $G'$ of $G$, and any irreducible representation $\pi$ in the $L$-packet $\Pi_{\phi}^L(G')$, the theta lift $\sigma$ of $\pi$ to $H$ lies in the $A$-packet $\Pi_{\theta(\phi)}(H)$. Moreover, we have
	\[
		\calJ_\scrW^L(\pi)= \ell^*\left(\calJ_{\scrW'}(\sigma)\right).
	\]
\end{proposition}

With the help of Proposition \ref{Sub-Diagram.Local.Compare}, we can first prove Proposition \ref{Final.Jigsaw} for a large class of $\pi\in \Pi_{\phi}^L(G')$. 
\begin{lemma}\label{Springboard}
	Let $G'$ be a pure inner form of $G$, and $\pi$ be an irreducible representation in the $L$-packet $\Pi_\phi^L(G')$. If the theta lift $\sigma_0$ of $\pi$ to $H_0$ is non-zero (with respect to the datum $(\psi_F,\chi_V,\chi_W)$), then Proposition \ref{Final.Jigsaw} holds for $\pi$. In particular, if we are in one of the following cases:
	\begin{itemize}
		\item Case ${\rm O}$;
		\item Case ${\rm U}$, and $F$ is non-Archimedean;
	\end{itemize}
	and $\chi_W\not\subset\phi$, then Proposition \ref{Final.Jigsaw} holds for any $\pi\in \Pi_{\phi}^L(G')$. 
\end{lemma}
\begin{proof}
	The first assertion can be proved exactly the same as Proposition \ref{Sub-Diagram.Local.Compare}. To prove that Proposition \ref{Final.Jigsaw} holds in the special cases listed above, one just needs to note that $\vartheta(\pi)\neq 0$ for any $\pi\in \Pi^L_{\phi}(G')$ in these special cases (see Theorem \ref{Prasad.Conj}). Hence we are done. 
\end{proof}
Based on this lemma, we now fill in the last jigsaw piece.
\begin{proof}[Proof of Proposition \ref{Final.Jigsaw}]
	We will prove this proposition case by case. To simplify notations, we let
	\[
		\eta_\pi=\calJ_\scrW^L(\pi) \quad \textit{and} \quad \eta_\sigma=\calJ_{\scrW'}(\sigma).
	\]

	\underline{Case O:} By Lemma \ref{Springboard}, we only need to consider the case when the theta lift of $\pi$ to $H_0$ is zero. It then follows from the conservation relation \cite{MR3369906} that the theta lift of $\pi\otimes\det$ to $H_0$ is non-zero. In particular, Proposition \ref{Final.Jigsaw} holds for $\pi\otimes\det$. Next we appeal to the global method to compare the theta lifts $\theta(\pi)$ and $\theta(\pi\otimes\det)$ of $\pi$ and $\pi\otimes\det$ to $H$.
	
	As in the proof of Proposition \ref{Multi-Free.N.Inj}, we define an abstract irreducible representation $\dot\sigma'=\bigotimes_v\dot\sigma'_v$ of $\dot H(\AAA)$ by setting:
	\begin{itemize}
		\item $\dot\sigma'_{u_1}=\dot\sigma'_{u_2}=\theta(\pi\otimes\det)$;
		\item at place $v\notin\{u_1,u_2\}$, $\dot\sigma'_v$ is the irreducible representation in the $L$-packet $\Pi_{\phi_{\theta(\dot\phi)_v}}^L(\dot H_v)$ associated to the trivial character of $\overline{\calS_{\phi_{\theta(\dot\phi)_v}}}$.
	\end{itemize}
	By Theorem \ref{H-ThmB}, Lemma \ref{Compare-epsilon} and equation (\ref{epsilon-generic}), $\dot\sigma'$ is a summand of $L^2_{\theta(\dot\phi)}(\dot H)$. Remark \ref{Multi-Preserve-PIF} then implies that there exists a pure inner form $\dot G'$ of $\dot G$, and an irreducible summand $\dot\pi'$ of $L_{\dot\phi}^2(\dot G')$, such that $\dot\sigma'=\theta^{abs}(\dot\pi')$. According to the construction, we must have
	\[
	\dot G'_{u_1}=\dot G'_{u_2}=G',
	\]
	and
	\[
	\dot\pi'_{u_1}=\dot\pi'_{u_2}=\pi\otimes\det.
	\]
	Now we define another abstract irreducible representation $\dot\pi$ of $\dot G'(\AAA)$ by setting $\dot \pi=\dot \pi'\otimes \det_T$ with $T=\{u_1,w\}$. More precisely, 
	\begin{itemize}
		\item at place $v\in\{u_1,w\}$, $\dot\pi_v=\dot\pi'_v\otimes\det_v$;
		\item at place $v\notin\{u_1,w\}$, $\dot\pi_v=\dot\pi'_v$.
	\end{itemize}
	Then $\dot\pi$ is also an irreducible summand of $L_{\dot\phi}^2(\dot G')$, such that
	\[
	\dot\pi_{u_1}=\pi.
	\]
	Let $\dot\sigma=\theta^{abs}(\dot\pi)$. We deduce from Corollary \ref{Multi-Preserve} that $\dot\sigma$ is a summand of $L^2_{\theta(\dot\phi)}(\dot H)$. Hence by the AMF Theorem \ref{H-ThmA}, we have
	\[
	\sigma=\dot\sigma_{u_1}\in\Pi_{\theta(\phi)}(H).
	\]
	This proves the first assertion of Proposition \ref{Final.Jigsaw} for Case O. 

	Next we prove the second assertion of Proposition \ref{Final.Jigsaw}. By Theorem \ref{H-ThmB}, Lemma \ref{Compare-epsilon} and equation (\ref{epsilon-generic}), we have
	\begin{equation}\label{transferAMF-2}
		\ell^*\left(\mathcal J_{\scrW'}(\dot \sigma)\right)=1,
	\end{equation}
	where 
	\begin{equation*}%\label{transferAMF}
		\ell^*\left(\mathcal J_{\scrW'}(\dot \sigma)\right)(x)= \prod_v \ell_v^*\left(\mathcal J_{\scrW_v'}(\dot \sigma_v)\right)(x_v)
	\end{equation*}
	for $x\in \overline{\mathcal S_{\dot \phi}}$. The character $\ell_v^*\left(\mathcal J_{\scrW_v'}(\dot \sigma_v)\right)$ can be computed explicitly as follows: 
	\begin{itemize}
		\item at place $v\notin \{u_1,u_2,w\}$, $\ell_v^*\left(\mathcal J_{\scrW_v'}(\dot \sigma_v)\right)$ is the trivial character of $\mathcal S_{\dot\phi_v}$;
		\item at the place $u_2$, we know that Proposition \ref{Final.Jigsaw} holds for $\pi\otimes\det$ by our hypothesis. Thus
		\begin{align*}\label{O.Equal.rk.Eqn}
			\ell_{u_2}^*\left(\mathcal J_{\scrW_{u_2}'}(\dot \sigma_{u_2})\right)=\mathcal J^L_{\scrW_{u_2}}(\pi\otimes \det)= \eta_\pi\cdot \kappa_{\phi}.
		\end{align*}
		Here we have made use of Remark \ref{LLC} in the last equality;
		\item at the place $w$, since $\mathbbm{1}\not\subset\dot\phi_w$, it follows from Lemma \ref{Springboard} that Proposition \ref{Final.Jigsaw}, hence Theorem \ref{Local-Compare} holds for $\dot\phi_w$. Thus
		\[
		\ell_w^*\left(\mathcal J_{\scrW_w'}\left(\dot \sigma_w\right)\right)=\mathcal J^L_{\scrW_w}(\dot \pi'_w\otimes {\det}_{w})=\kappa_{\dot\phi_w}.
		\]
		Here again we have made use of Remark \ref{LLC} in the last equality.
	\end{itemize}
	Hence 
	\begin{equation}\label{cal-1}
		\prod_v \ell_v^*\left(\mathcal J_{\scrW_v'}(\dot \sigma_v)\right)(x)= \eta_{\sigma}(x_{u_1})\cdot\eta_\pi(x_{u_2})\cdot \kappa_{\phi}(x_{u_2})\cdot\kappa_{\dot\phi_w}(x_w). 
	\end{equation}
	for all $x\in \overline{\mathcal S_{\dot \phi}}$. On the other hand, it is easy to check that 
	\begin{equation}\label{cal-2}
		\kappa_{\phi}\left(x_{u_2}\right)\cdot\kappa_{\dot\phi_w}\left(x_w\right)=1
	\end{equation}
	for all $x\in \overline{\mathcal S_{\dot \phi}}$. Combining these equalities (\ref{transferAMF-2}), (\ref{cal-1}) and  (\ref{cal-2}), we get 
	\[
	\eta_\pi(x_{u_2}) = \ell_{u_1}^*\left(\eta_\sigma\right)(x_{u_1}).
	\]
	for all $x\in \overline{\mathcal S_{\dot \phi}}$. Finally, since the localization maps $\mathcal S_{\dot \phi}\lra \mathcal S_{\dot \phi_{u_1}}$ and  $\mathcal S_{\dot \phi}\lra \mathcal S_{\dot \phi_{u_2}}$ agree and they are surjective, we deduce that 
	\[
	\eta_\pi= \ell_{u_1}^*\left(\eta_\sigma\right).
	\]
	This completes the proof in Case O.\\
	
	\underline{Case $\U_0$, and $F$ is non-Archimedean:} In this case, $\dot F$ is a totally imaginary field, $W_0$ is an $(2n+1)$-dimensional skew-Hermitian space over $\dot E$, and $H_0=\U(W_0)$. For a pure inner form $G'$ of $G$ and $\pi$ an irreducible tempered representation in the $L$-packet $\Pi_{\phi}^L(G')$, Lemma \ref{Springboard} asserts that, if the theta lift of $\pi$ to $H_0$ (with respect to the datum $(\psi_F,\chi_V,\chi_W)$) is non-zero, then Proposition \ref{Final.Jigsaw} holds for $\pi$. Next we appeal to the global method to reduce the general situation to this known situation.
	
	We pick up the pair of characters $\left(\chi'_{\dot V},\chi'_{\dot W}\right)$ such that $\chi'_{\dot W,v}\not\subset\dot\phi_v$ for $v\in\{u_1,w\}$. Firstly we use the almost equal rank theta lift to globalize the representation $\pi$. As in the proof of Lemma \ref{Global-MultiOne}, let
	\[
	\vartheta(\dot\phi)=\dot\phi(\chi'_{\dot W})^{-1}\chi'_{\dot V}+\chi'_{\dot V}.
	\]
	We define an irreducible automorphic subrepresentation $\dot\sigma_0=\bigotimes_v\dot\sigma_{0,v}$ of $L^2_{\vartheta(\dot\phi)}(\dot H_0)$ by setting:
	\begin{itemize}
		\item at place $v\notin\{u_1,w\}$, $\dot\sigma_{0,v}$ is the irreducible representation in the $L$-packet $\Pi_{\vartheta(\dot\phi)_v}^L(\dot H_{0,v})$ associated to the trivial character of $\overline{\calS_{\vartheta(\dot\phi)_v}}$;
		\item at the place $u_1$, $\dot\sigma_{0,u_1}=\vartheta(\pi)$ is the theta lift of $\pi$ to $\dot H_{0,u_1}$, which is non-zero by Theorem \ref{Prasad.Conj};
		\item at the place $w$, $\dot\sigma_{0,w}$ is the tempered representation in the $L$-packet $\Pi_{\vartheta(\dot\phi)_w}^L\left(\dot H_{0,w}\right)$ corresponding to the character $\eta_{0,w}$, determined by the formula
		\[
			\prod_v\eta_{0,v} = 1,
		\]
		where $\eta_{0,v}=\calJ_{\scrW'_v}^L(\dot\sigma_{0,v})$, and we regard $\prod_v\eta_{0,v}$ as a character of the global component group $\calS_{\vartheta(\dot\phi)}$ through the localization maps. 
	\end{itemize}
	By Theorem \ref{H-ThmB} and equation (\ref{epsilon-generic}), $\dot\sigma_0$ is a summand of $L^2_{\vartheta(\dot\phi)}(\dot H_0)$. Using the same argument as in Lemma \ref{Global-MultiOne}, we can show that there exists a pure inner form $\dot G'$ of $\dot G$ such that the abstract theta lift $\dot \pi=\vartheta^{abs}(\dot\sigma_0)$ of $\dot\sigma_0$ to $\dot G'$ is nonzero, and  
	\[
	m_{\disc}(\dot\pi)\geq m_{\cusp}(\dot\pi)\geq m_{\cusp}(\dot\sigma_0). 
	\]
	Hence $\dot\pi$ is a summand in $L^2_{\dot \phi}(\dot G')$, and $\dot\pi_v\in\Pi_{\dot\phi_v}^L(\dot G'_v)$ for all places $v$ of $\dot F$. By the conservation relation \cite{MR3369906}, we know that
	\[
		\left(\dot G'_{u_1},\dot\pi_{u_1}\right)\simeq\left(G',\pi\right).
	\] 
	Also, Theorem \ref{Prasad.Conj} implies that 
\begin{equation}\label{U.NA.Equal.rk.Eqn}
	\prod_v\eta_{v} (x_v)= 1
\end{equation}
	for all $x\in\calS_{\dot\phi}$, where $\eta_v=\calJ_{\scrW_v}^L(\dot\pi_v)$. This product is well-defined, since $\eta_v=1$ for all places $v\notin\{u_1,w\}$.
	
	Next we consider the stable range theta lift of $\dot\pi$ to extort some other information. Let $\dot\sigma=\theta^{abs}(\dot\pi)$. We deduce from J-S. Li's inequality (Theorem \ref{J-S inequality}) that $\dot\sigma$ is a summand of $L^2_{\theta(\dot\phi)}(\dot H)$. Hence 
	\[
	\sigma=\dot\sigma_{u_1}\in\Pi_{\theta(\phi)}(H).
	\]
	Also, it follows from Theorem \ref{H-ThmB}, Lemma \ref{Compare-epsilon} and equation (\ref{epsilon-generic}) that
	\begin{equation}\label{transferAMF-3}
		\ell^*\left(\mathcal J_{\scrW'}(\dot \sigma)\right)=1,
	\end{equation}
	where 
	\begin{equation*}%\label{transferAMF-4}
		\ell^*\left(\mathcal J_{\scrW'}(\dot \sigma)\right)(x)= \prod_v \ell_v^*\left(\mathcal J_{\scrW_v'}(\dot \sigma_v)\right)(x_v)
	\end{equation*}
	for $x\in \overline{\mathcal S_{\dot \phi}}$. For all places $v\not\in\{u_1,w\}$, we have 
	\begin{equation*}%\label{cal-4}	
		\ell_v^*\left(\mathcal J_{\scrW_v'}(\dot \sigma_v)\right)=\eta_v.
	\end{equation*}
	Indeed, if $\dot\phi_v$ is of good parity, then this equality follows from Theorem \ref{Prasad.Conj} and Lemma \ref{Springboard}; on the other hand, if $\dot\phi_v$ is not of good parity, then this equality follows from our induction hypothesis. At place $w$ of $F$, since $\chi_{\dot W,w}\not\subset\dot\phi_w$, it follows from Lemma \ref{Springboard} that Theorem \ref{Local-Compare} holds for $\dot\phi_w$. Hence
	\begin{equation}\label{cal-5}
		\ell_w^*\left(\mathcal J_{\scrW_w'}(\dot \sigma_w)\right)=\eta_{w}.
	\end{equation}
	Combining these equalities (\ref{U.NA.Equal.rk.Eqn}), (\ref{transferAMF-3}) and (\ref{cal-5}), we have 
	\begin{equation*}
		\eta_{\pi}(x_{u_1})=\ell^*_{u_1}(\eta_\sigma)(x_{u_1}).
	\end{equation*}
	Since the localization map $\mathcal S_{\dot \phi}\lra \mathcal S_{\dot \phi_{u_1}}$ is surjective, 
	the desired conclusion for $\pi$ holds.\\
	
	\underline{Case $\U_1$, and $F$ is non-Archimedean:} The method we used in the previous case can also apply to this case similarly. Indeed, this case is even easier, since Theorem \ref{AMF-U_1} allow as to globalize $\pi\in\Pi_\phi^L(G')$ directly. We omit the details.\\
	
	\underline{Case U, and $F$ is real:} The method we used in the previous two cases can not apply here since the almost equal rank theta lift of $\pi\in\Pi_\phi^L(G')$ to the split Witt tower may vanish. Therefore, we can not globalize it by using the AMF. In this case, we prove by induction on $\dim V$. Then we can use the AMF for some smaller unitary groups to help us do the globalization step. 
	
	When $\dim V=0$, this follows from Lemma \ref{Induction.Hypothesis}. Suppose that for all non-negative integer $m<\dim V$, Proposition \ref{Final.Jigsaw}, hence Theorem \ref{Local-Compare} holds for all real unitary groups of $m$-variables. Now we show that the desired conclusion also holds for all real unitary groups of $\dim V$-variables.
	
	Recall that $\phi$ is an $L$-parameter of good parity for $G$, so it must be of the form
	\[
	\phi=m_1\chi_1+\cdots+m_r\chi_r,
	\]
	where $\chi_i$ is a conjugate self-dual character of $L_{\CC}=\CC^\times$, and $m_i$ is some positive integer. Due to the same reason as in the proof of Lemma \ref{Global-MultiOne}, we may assume that $\dot\phi$ is a summation of automorphic characters. We pick up the pair of characters $\left(\chi'_{\dot V},\chi'_{\dot W}\right)$ such that $\chi'_{\dot W}\subset\dot\phi$. Let
	\[
	\vartheta(\dot\phi)=\left(\dot\phi-\chi'_{\dot W}\right)(\chi'_{\dot W})^{-1}\chi'_{\dot V}.
	\]
	For any $\pi\in \Pi_\phi^L(G')$, we first use the almost equal rank theta lift to globalize it. By Theorem \ref{Prasad.Conj}, there exists a pure inner form $H'_0=H(W'_0)$ of $H_0$, such that the theta lift $\sigma_0\coloneqq\vartheta(\pi)$ of $\pi$ to $H'_0$ (with respect to $\left(\psi_{F,v}, \chi'_{\dot V,v},\chi'_{\dot W,v}\right)$) is non-zero. Let $\dot W'_0$ be the unique $(\dim V-1)$-dimensional $c$-skew-Hermitian space, such that it is split at all places of $\dot F$ except $\{u_1,w\}$, and the localization of $\dot W'_0$ at the place $u_1$ is isometric to $W'_0$. Let $\dot H'_0=H(\dot W'_0)$. Now we define an abstract irreducible representation $\dot\sigma'_0=\bigotimes_v\dot\sigma'_{0,v}$ of $\dot H'_0(\AAA)$ by setting:
	\begin{itemize}
		\item at place $v\notin\{u_1,w\}$, $\dot\sigma'_{0,v}$ is the irreducible representation in the $L$-packet $\Pi_{\vartheta(\dot\phi)_v}^L(\dot H'_{0,v})$ associated to the trivial character of $\calS_{\vartheta(\dot\phi)_v}$.
		\item at the place $u_1$, $\dot\sigma'_{0,u_1}=\sigma_0$;
		\item at the place $w$, $\dot\sigma'_{0,w}$ is the irreducible representation in the $L$-packet $\Pi_{\vartheta(\dot\phi)_w}^L(\dot H'_{0,w})$ corresponding to the character $\eta'_{0,w}\in\wh{\calS_{\vartheta(\dot\phi)_w}}$, determined by the formula
		\[
			\prod_v\eta'_{0,v} = 1,
		\]
		where $\eta'_{0,v}=\calJ_{\scrW'_v}^L(\dot\sigma'_{0,v})$, and we regard $\prod_v\eta'_{0,v}$ as a character of the global component group $\calS_{\vartheta(\dot\phi)}$ through the localization maps. The existence of $\eta'_{0,w}$ is guaranteed by the local-global principle for skew-Hermitian spaces and the LLC for unitary groups. 	
	\end{itemize}
	Since $\dot H'_0$ is a unitary group of $(\dim V-1)$-variables, by our induction hypothesis, Theorem \ref{Local-Compare} holds for all localizations of $\dot H'_0$, hence Theorem \ref{ThmB} holds for $\dot H'_0$. It then follows that $\dot\sigma'_0$ is a summand of $L^2_{\vartheta(\dot\phi)}(\dot H'_0)$. For each place $v$ of $\dot F$, by the conservation relation \cite{MR3369906}, there is a Hermitian space $V'_{v}$ of the same dimension as $V$, such that 
	\[
	\vartheta(\dot\sigma'_{0,v})\neq0,
	\]  
	where $\vartheta(\dot\sigma'_{0,v})$ is the theta lift of $\dot\sigma'_{0,v}$ to $G'_v=G(V'_{v})$ (again with respect to $\left(\psi_{F,v}, \chi'_{\dot V,v},\chi'_{\dot W,v}\right)$). Since $\dot\phi_w$ is discrete, $\vartheta(\dot\phi)_w$ does not contain the character $\chi'_{\dot V,w}$. Thus we will have two choices of the skew-Hermitian space $V'_{w}$ at the place $w$. The flexibility at the place $w$ allows us to pick these local Hermitian spaces coherently such that they form a global Hermitian space $\dot V'$ over $\dot E$. We let $\dot G'=G(\dot V')$. Let $\dot\pi=\vartheta^{abs}(\dot\sigma_0')$ be the abstract theta lift to $\dot G'$. It follows from the same argument as in Lemma \ref{Global-MultiOne} that 
	\[
	m_{\disc}(\dot\pi)\geq m_{\cusp}(\dot\pi)\geq m_{\cusp}(\dot\sigma'_0). 
	\]
	Hence $\dot\pi$ is a summand in $L^2_{\dot \phi}(\dot G')$, and $\dot\pi_v\in\Pi_{\dot\phi_v}^L(\dot G'_v)$ for all places $v$ of $\dot F$. Then by the conservation relation \cite{MR3369906}, we know that
	\[
	\left(\dot G'_{u_1},\dot\pi_{u_1}\right)\simeq(G',\pi).
	\] 
	Also, Theorem \ref{Prasad.Conj} implies that 
\begin{equation}\label{U.A.Equal.rk.Eqn}
	\prod_v\eta_{v} (x_v)= 1
\end{equation}
	for all $x\in\calS_{\vartheta(\dot\phi)}$. Here $\eta_v=\calJ_{\scrW_v}^L(\dot\pi_v)$, and we regard $\calS_{\vartheta(\dot\phi)}$ as a subgroup of $\calS_{\dot\phi}$ via the natural embedding
	\[
	\calS_{\vartheta(\dot\phi)}\longrightarrow\calS_{\dot\phi}.
	\] 
	Since $\eta_v=1$ for all places $v\notin\{u_1,w\}$, this product is well-defined.
	 
	Next we consider the stable range theta lift of $\dot\pi$ to extort some other information. Let $\dot\sigma=\theta^{abs}(\dot\pi)$. We deduce from J-S. Li's inequality (Theorem \ref{J-S inequality}) that $\dot\sigma$ is a summand of $L^2_{\theta(\dot\phi)}(\dot H)$. Hence 
	\[
	\sigma=\dot\sigma_{u_1}\in\Pi_{\theta(\phi)}(H).
	\]
	Similar to previous cases, combining the equality (\ref{U.A.Equal.rk.Eqn}) and the AMF for $\dot H$, we get 
	\begin{equation*}%\label{final}
		\eta_\pi(x_{u_1})=\ell_{u_1}^*\left(\eta_{\sigma}\right)(x_{u_1})
	\end{equation*}
	for all $x\in \calS_{\vartheta(\dot\phi)}$. Since the localization map $\calS_{\vartheta(\dot\phi)}\rightarrow \calS_{\vartheta(\dot\phi)_{u_1}}$ is surjective, we have proved that $\eta_\pi$ and $\ell_{u_1}^*\left(\eta_{\sigma}\right)$ are equal on the image of the natural embedding $\calS_{\vartheta(\dot\phi)_{u_1}}\hookrightarrow\calS_{\phi}$. Certainly this embedding $\calS_{\vartheta(\dot\phi)_{u_1}}\hookrightarrow \calS_{\phi}$ is not necessarily surjective. But there is nothing to worry about, since we may vary the character $\chi'_{\dot W}$. When the character $\chi'_{\dot W}$ runs over all irreducible components of $\dot\phi$, the image of $\calS_{\vartheta(\dot\phi)_{u_1}}$ will exhaust all elements in $\calS_{\phi}$. Hence the desired conclusion holds for real unitary groups.
\end{proof}

So now, we have finished proving Theorem \ref{ThmB}.

\section{Encore: beyond the generic case}
Let $F$ be a number field, $G$ an even orthogoanl or unitary group over $F$ as in the setting of Section \ref{NEC&APara}. Let $\psi$ be an elliptic $A$-parameter for $G$ and
\[
	\theta(\psi)=\psi\chi_W^{-1}\chi_V+\chi_V\boxtimes S_{2r-2n+1}
\]
an elliptic $A$-parameter for $H$. In Section \ref{transfer.AMF}, we have transfered the AMF from $L_{\theta(\psi)}^2(H)$ to $L_\psi^2(G)$ when $\psi=\phi$ is generic. Recall that the key step is to show Proposition \ref{cusp.realization}, which implies that J-S. Li's inequality (Theorem \ref{J-S inequality}) is an equality in the generic case. In this section, we want to go one step further beyond the generic case. We would like to propose the following naive conjecture.

\begin{conjecture}\label{Conjectureequality}
	Let $G$ be an even orthogonal or unitary group, and
\begin{equation}\label{E:sumPhi}
	\psi=\sum_i\phi_i\boxtimes S_{d_i}
\end{equation}
	an elliptic $A$-parameter for $G$, where $\phi_i$ is a cuspidal representation of $\GL_{n_i}(\AAA_E)$. Let $\pi$ be an irreducible representation of $ G(\mathbb A)$ such that the $L$-parameter of $\pi_v$ is $\phi_{\psi_v}$ for almost all $v$. Then we have 
	\[
	m_{\disc}(\pi)=m(\pi).  
	\]
\end{conjecture} 

We have proved this conjecture in Proposition \ref{cusp.realization} when $\psi=\phi$ is generic. Indeed, it is easy to generalize this conjecture to a slightly more general case. 
\begin{assumption}\label{strange.assumption}
	Let $r_V$ be the Witt index of $V$. Suppose that in the expression (\ref{E:sumPhi}), for any $i$ such that $d_i>1$, we have $n_i>r_V$.
\end{assumption}
Under this weird assumption, we can prove Conjecture \ref{Conjectureequality} by using an argument similar to that of Proposition \ref{cusp.realization}. 
\begin{proposition}\label{strangepropo}
	Suppose that $G=G(V)$ and $\psi$ satisfy the Assumption \ref{strange.assumption}. Let $\pi$ be an irreducible representation of $G(\AAA)$, such that the $L$-parameter of $\pi_v$ is $\phi_{\psi_v}$ for almost all $v$. Then we have
	\[
	m_{\cusp}(\pi)=m_{\disc}(\pi)=m(\pi).
	\]
\end{proposition}
\begin{remark}
	% Observe that when $\psi$ is generic, it satisfies the Assumption \ref{strange.assumption}. We have proved this case in Proposition \ref{cusp.realization}. 
	A case worth noting is when $r_V=0$, i.e. $G=G(V)$ is anisotropic. In this case, Assumption \ref{strange.assumption} is automatically satisfied. Indeed, when $G$ is anisotropic, we have $\mathcal A_{\cusp}(G)=\mathcal A_{\disc}(G)=\mathcal A(G)$. Therefore, Proposition \ref{strangepropo} holds with no extra work needed. 
\end{remark}
\begin{proof}[Proof of Proposition \ref{strangepropo}]
	In the same spirit of \cite[Prop.4.1]{MR3866889}, it suffices to show that for any realization $\calV\subset\calA(G)$ of $\pi$, we have $\calV\subset\calA_{\cusp}(G)$. Suppose on the contrary that $\calV\not\subset\calA_{\cusp}(G)$ for some such $\calV$. By considering the constant term maps, it follows from \cite{langlands1979notion} that 
	\[
	\pi\subset\Ind_{P(\AAA)}^{G(\AAA)}\rho
	\]
	for some proper parabolic subgroup $P$ of $G$ with Levi component $M$, and some irreducible cuspidal automorphic representation $\rho$ of $M(\AAA)$. Suppose that
	\[
	M\simeq\prod_j\GL_{k_j}\times G_0
	\]
	for some $G_0=G(V_0)$, where $V_0$ is a space in the Witt tower containing $V$. Then $\rho$ is of the form
	\[
	\rho\simeq\left(\BIGboxtimes_j\tau_j\right)\boxtimes\pi_0
	\] 
	for some irreducible cuspidal automorphic representation $\tau_j$ and $\pi_0$ of $\GL_{k_j}$ and $G_0$, respectively. By Theorem \ref{ThmA}, $\pi_0$ has a weak transfer $\tau_0$ to $\GL_{n_0}(\AAA)$. Then $\pi$ has a weak transfer to $\GL_n(\AAA)$ of the form
	\begin{equation}\label{firstexpression}
		\left(\BIGboxplus_j(\tau_j\boxplus(\tau_j^c)^\vee)\right)\boxplus \tau_0.
	\end{equation}
	
	On the other hand, since the $L$-parameter of $\pi_v$ is $\phi_{\psi_v}$ for almost all $v$, it follows that $\pi$ has a weak transfer to $\GL_n(\AAA)$ of the form
	\begin{equation}\label{secondexpression}
		\BIGboxplus_i\left(\phi_i|\cdot|^{\frac{d_i-1}{2}}\boxplus\cdots\boxplus\phi_i|\cdot|^{-\frac{d_i-1}{2}}\right).
	\end{equation}
	By the strong multiplicity one theorem \cite{MR623137}, the two expressions (\ref{firstexpression}) and (\ref{secondexpression}) must agree. Hence $\tau_j$ in the first expression must have the form $\phi_{i_j}|\cdot|^{s_j}$ for some $i_j$ and $s_j\in \frac{1}{2}\mathbb Z$. Note that we have $k_j \leq r_V$. It then follows from Assumption \ref{strange.assumption} that 
	\[
	k_j <n_i
	\]
	for any $i$ such that $d_i>1$. Hence we must have $d_{i_j}=1, s_j=0$. This also implies that 
	\[
	\phi_{i_j}\boxtimes S_{d}
	\]
	is not contained in $\psi$ for any $d>1$. However, $\tau_j=\phi_{i_j}$ occurs with multiplicity at least $2$ in the expression (\ref{firstexpression}), whereas it occurs with multiplicity $1$ in the expression (\ref{secondexpression}). This is a contradiction. Hence we have $\calV\subset\calA_{\cusp}(G)$ as required.
\end{proof}

When the pair $(G,\psi)$ does not satisfy Assumption \ref{strange.assumption}, the realizations $\calV\subset\calA(G)$ of $\pi$ may not lie in $\calA_{\cusp}(G)$. To prove Conjecture \ref{Conjectureequality}, we need some extra inputs. Thanks to the square-integrability criterion \cite[I.4.11 Lem.]{MR1361168}, when $G=G(V)$ is of $F$-rank $1$, we are able to complete the proof.

\begin{proposition}
	Conjecture \ref{Conjectureequality} holds if $G=G(V)$ is of $F$-rank $1$, i.e. $r_V=1$. 
\end{proposition}
\begin{proof}
	Here we only prove Case O. The proof of Case U is similar. 

	Let $\mathcal V\subset \calA(G)$ be an automorphic realization of $\pi$. We need to show that $\mathcal V$ is contained in $\mathcal A^2(G)$. We may assume that $\mathcal V$ is not contained in $\mathcal A_{\cusp}(G)$, otherwise it is already contained in $\calA^2(G)$. Since $G$ is of $F$-rank $1$, the proper standard parabolic subgroup $P=MN$ of $G$ is unique, with the Levi component 
	\[
	M\simeq \GL_{1}\times G_0, 
	\]
	where $G_0$ is an anisotropic group.
	Let $\chi|\cdot|^s\boxtimes \pi_0$ be a cuspidal support of $\pi$ along $P$, where $\chi$ is an unitary automorphic character of $\GL_1$, and $\pi_0$ is a cuspidal automorphic representation of $G_0$. Then it follows from \cite{langlands1979notion} that  
	\begin{equation*}%\label{sub2}
		\pi\hookrightarrow \Ind_{P}^{G}\left(\chi|\cdot|^{s} \boxtimes \pi_0\right).    
	\end{equation*} 
	Considering the weak transfer of $\pi$ to $\GL_{2n}(\mathbb A)$. On the one hand, this weak transfer is represented by the elliptic $A$-parameter $\psi$; on the other hand, it also has an expression given by the above embedding. Similar to the proof of Proposition \ref{strangepropo}, by comparing these two expressions, we know that there exists some $i$ such that 
	\[
		\chi=\phi_i,\quad  d_i\geq 3\quad \textit{and}\quad s=\pm \frac{d_i-1}{2}. 
	\]
	Moreover, the cuspidal automorphic representation $\pi_0$ is in the NEC represented by the elliptic $A$-parameter 
	\begin{equation}\label{psi0}
		\psi_0=\psi-\chi\boxtimes S_{d_i}+ \chi\boxtimes S_{d_i-2}.	    
	\end{equation}
	If we can show that $s=-(d_i-1)/2$, then the square-integrability criterion \cite[I.4.11 Lem.]{MR1361168} will imply that $\calV\subset\calA^2(G)$, which will complete the proof. So next we shall prove this by contradiction. 

	Suppose on the contrary that $s=(d_i-1)/2$, then at every unramified place $v$, we have 
	\begin{equation}\label{sub3}
		\pi_v\hookrightarrow \Ind_{P_v}^{G_v}\left(\chi_v|\cdot|^{\frac{d_i-1}{2}} \boxtimes \pi_{0,v}\right).     
	\end{equation}
	Since $\chi_v|\cdot|^{\frac{d_i-1}{2}}$ is not self-dual, it follows from \cite[Lem.3.1.3]{MR3268853} that $\pi_v$ is the unique subrepresentation of $\Ind_{P_v}^{G_v}\left(\chi_v|\cdot|^{\frac{d_i-1}{2}} \boxtimes \pi_{0,v}\right)$. 
	Applying both the MVW and contragredient functors, we know that $\pi_v$ is also the unique quotient of $\Ind_{P_v}^{G_v}\left(\chi_v|\cdot|^{-\frac{d_i-1}{2}} \boxtimes \pi_{0,v}\right)$. Let $K_v$ be a special maximal compact subgroup of $G_v$ which has good position relative to $P_v$. Fix a representative $w_v\in K_v$ of the unique non-trivial element in $W_{M_v}=N_{G_v}(M_v)/M_v$. Let 
	\[
	\mathcal M\left(s,\chi_v\boxtimes \pi_{0,v},w_v\right): \Ind_{P_v}^{G_v}\left(\chi_v|\cdot|^{s} \boxtimes \pi_{0,v}\right)\lra\Ind_{P_v}^{G_v}\left(\chi_v|\cdot|^{-s} \boxtimes \pi_{0,v}^{w_v}\right)
	\]
	be the unnormalized intertwining operator given by (the meromorphic continuation of) the integral 
	\[
	\mathcal M\left(s,\chi_v\boxtimes \pi_{0,v},w_v\right) f(g)=\int_{N_v} f(w_v^{-1}ng) dn  
	\]
	for $f\in\Ind_{P_v}^{G_v}\left(\chi_v|\cdot|^{s} \boxtimes \pi_{0,v}\right)$, where $\pi_{0,v}^{w_v}$ is the representation of $G_{0,v}$ on the same space of $\pi_{0,v}$ with the action given by
	\[
	\pi_{0,v}^{w_v}(m)=\pi_{0,v}(w_v^{-1}mw_v).
	\]
	Let $f_{v,s}$ and $f_{v,-s}^{\prime}$ be the unramified vectors in $\Ind_{P_v}^{G_v}\left(\chi_v|\cdot|^{s} \boxtimes \pi_{0,v}\right)$ and $\Ind_{P_v}^{G_v}\left(\chi_v|\cdot|^{-s} \boxtimes \pi_{0,v}^{w_v}\right)$ respectively with the normalization
\[
	f_{v,s}\left(1_{G_v}\right)=f_{v,-s}^\prime\left(1_{G_v}\right)=1.
\]
	Then by the Gindikin-Karpelevich formula \cite[p.141, Thm.6.7]{MR2071722}, we have 
	\begin{equation*}%\label{GK-formula}
		\mathcal M\left(s,\chi_v\boxtimes \pi_0,w_v\right) f_{v,s}= \frac{L\left(s,\pi_{0,v}\times \chi_v\right)}{L\left(1+s,\pi_{0,v}\times \chi_v\right)}\cdot f_{v,-s}^\prime.
	\end{equation*}
%	where $L(s,\pi_{0,v}\times \chi_v)$ is the $L$-function of $\pi_{0,v}\boxtimes \chi_v$ relative to the standard representation of ${}^L M_v$.
	If we write 
\[
	\psi_{0,v}=\sum_j \mu_j\boxtimes S_{d_j}
\]
for some unramified characters $\mu_j$ (not necessarily unitary), then
\begin{equation}\label{normalizatin}
	\frac{L\left(s,\pi_{0,v}\times \chi_v\right)}{L\left(1+s,\pi_{0,v}\times \chi_v\right)}
	=\prod_j\frac{L\left(s-\frac{d_j-1}{2},\mu_j\chi_v\right)}{L\left(s+\frac{d_j+1}{2},\mu_j\chi_v\right)}.
\end{equation}
Since $\psi_{0,v}$ is the localization of the global elliptic A-parameter $\psi_0$, we know that for each $j$, $\mu_j$ can be decomposed as
\[
	\mu_j=\mu'_j|\cdot|^s
\]
for some unitary character $\mu'_j$ and real number $|s_j|<1/2$. By (\ref{psi0}), we have $\chi_v\boxtimes S_{d_i-2}\subset\psi_{0,v}$. This implies that the factor (\ref{normalizatin}) has a zero at $s=-(d_i-1)/2$. Hence 
	\begin{equation}\label{zero}
		\mathcal M\left(s,\chi_v\boxtimes \pi_0,w_v\right) f_{v,s}~\Big|_{s=-\frac{d_i-1}{2}}= 0. 
	\end{equation}
	Moreover, since $\pi_v$ is the unique quotient of 
	\[
		\Ind_{P_v}^{G_v}\left(\chi_v|\cdot|^{-\frac{d_i-1}{2}} \boxtimes \pi_{0,v}\right), 
	\]
	we know that $\Ind_{P_v}^{G_v}\left(\chi_v|\cdot|^{-\frac{d_i-1}{2}} \boxtimes \pi_{0,v}\right)$ is generated by the unramified vector $f_{v,s}$. By the equation (\ref{zero}) and the $G_v$-equivariance of $\mathcal M\left(s,\chi_v\boxtimes \pi_0,w_v\right)$, we deduce that $\mathcal M(s,\chi_v\boxtimes \pi_0,w_v)$ is holomorphic and zero at $s=-(d_i-1)/2$, which is impossible (see \cite[Thm.VI.1.1 \& Rmk.]{MR1989693}). This finishes the proof. 
\end{proof}

Then, totally the same as the proof of Corollary \ref{Multi-Preserve}, we deduce:
\begin{corollary}\label{anisotropic}
	Suppose that either: 
	\begin{enumerate}
		\item $(G,\psi)$ satisfies the Assumption \ref{strange.assumption}; 
		\item $G$ is of $F$-rank $1$, and $\psi$ is any elliptic $A$-parameter for $G$. 
	\end{enumerate}
	Suppose that
	\[
	L_{\psi}^2(G)=\bigoplus_\pi m_\pi\pi. 
	\]
	Then
	\[
	L_{\theta(\psi)}^2(H)=\left(\bigoplus_\pi m_\pi\theta^{abs}(\pi)\right)\oplus\left(\bigoplus'_\sigma m_\sigma\sigma\right),
	\]
	where the second summation on the RHS runs over all $\sigma$ with $A$-parameter $\theta(\psi)$ and not relevant to $G$.	
\end{corollary}
Recall that in Section \ref{transfer.AMF}, we have defined local packets $\Pi_{\psi_v}^\theta(G_v)$ for each place $v$ of $F$, as well as the global packet $\Pi_\psi^\theta(G,\epsilon_\psi)$. Combining all of these with Lemma \ref{Compare-epsilon}, we deduce: 
\begin{theorem}\label{Thm-C}
	Suppose that either: 
	\begin{enumerate}
		\item $(G,\psi)$ satisfies the Assumption \ref{strange.assumption}; 
		\item $G$ is of $F$-rank $1$, and $\psi$ is any elliptic $A$-parameter for $G$. 
	\end{enumerate}
	Then there is a decomposition
	\[
	L_\psi^2(G)=\bigoplus_{\pi\in\Pi_\psi^\theta(G,\epsilon_\psi)}\pi.
	\]
\end{theorem}
\begin{remark}
	In particular, when $G=G(V)$ is of $F$-rank less than or equal to $1$, we obtain a description of the whole $L^2_{\disc}(G)$. A case worth noting is when $G$ is an unitary group and $G_v\simeq \U_{1,n-1}$ at one real place $v$. In this case, the description of $L^2_{\disc}(G)$ might have some arithmetic applications to Shimura varieties of type $\U_{1,n-1}$. 
\end{remark}
These results also motivate us to study these local packets $\Pi_{\psi_v}^\theta(G_v)$ at each local place $v$ of $F$. In particular, we want to show that: 
\begin{itemize}
	\item These local packets $\Pi_{\psi_v}^\theta(G_v)$ do not depend on the choice of the auxiliary group $H_v=H\left(W_v\right)$.
	\item If $G_v$ is quasi-split, then
	\[
		\Pi_{\psi_v}^\theta(G_v)=\Pi_{\psi_v}^A(G_v)
	\]
	 as representations of $\calS_\psi\times G_v$, where $\Pi_{\psi_v}^A(G_v)$ is the local $A$-packet defined by Arthur \cite{MR3135650} and Mok \cite{MR3338302}. 	 
\end{itemize}
In \cite{chen2021theta}, we will prove these when $v$ is a non-Archimedean place of $F$. 

\begin{remark}~
\begin{enumerate}
\item In Case O, a large part of these local comparison results have been proved by M{\oe}glin already in a much more general context. In \cite{MR2767522}, she has constructed a packet $\Pi_{\psi_v}^M(G_v)$ explicitly for each $\psi_v$ when $v$ is non-Archimedean. Moreover, she showed that $\Pi_{\psi_v}^M(G_v)$ is multiplicity-free, and 
\[
	\Pi_{\psi_v}^M(G_v)=\Pi_{\psi_v}^A(G_v)
\]
as sets if $G_v$ is quasi-split. Using her explicity construction, she studied the Adams' conjecture in \cite{MR2906916} and obtained many results. It follows from her results that
\[
	\Pi_{\psi_v}^\theta(G_v)=\Pi_{\psi_v}^M(G_v)
\]
as sets. Hence, comparing to her results, the new thing in \cite{chen2021theta} is that we also compare the ``labelings'', i.e. the map $\calJ^\theta$.
\item In Case U, when $v$ is real and the $A$-parameter $\psi_v$ is Adams-Johnson, these local comparison results have been proved in \cite{cossutta2009theta} already. We expect that the method in \cite{cossutta2009theta} can also apply to Case O.
\end{enumerate}
\end{remark}

%For an irreducible unitary representation $\sigma$ of $H$, we say $\sigma$ is relevant with $G$, if there is an irreducible unitary representation $\pi$ of $G$ and a character of $E^1$, such that 
%\[
%	\sigma\simeq\theta(\pi)\otimes\chi
%\]

\appendix
%\iffalse%%%%%%%%%%%%%%%%%%%%%%%%%%%%%%%%%%%%%%%%%%%%%%%%%%%%%%%%%%%%%%%%%%%%%%%%%%%%%%%%%%%%%%%%%%%%%%%%%%%%%%%%%%%%%%%%%%%%%%%%%%%%%%%
\section{Prasad's conjecture: real even orthogonal-symplectic case} % (fold)
\label{App.LLC4O}

In this appendix, we shall use theta lifts to establish an LLC for real full even orthogonal groups, based on Paul's results \cite[Thm.15]{MR2175409}. %to full even orthogonal groups.

We fix some notation. For $\alpha\in\frac{1}{2}\ZZ$, we denote by $\chi_{2\alpha}$ the character
\[
z\mapsto \left(z/\bar{z}\right)^\alpha
\]
of $L_\CC$, and by $\calD_{2\alpha}$ the $2$-dimensional representation of $L_\RR$ induced from the character $\chi_{2\alpha}$ of $L_\CC$. Note that $\calD_{2\alpha}$ is irreducible unless $\alpha=0$. We also retain the notations in Section \ref{S:Asubdiagram}. So now $V$ is a $2n$-dimensional orthogonal space over $\RR$ with isometry group $G\left(V\right)$, and $W_0$ is a $2n$-dimensional symplectic space over $\RR$ with isometry group $H\left(W_0\right)$. Pure inner forms of $G\left(V\right)$ will be typically denoted by $G\left(V'\right)$ for some orthogonal space $V'$ (see the beginning of Section \ref{state.main.results} for the classification of these $V'$). We also denote the special even orthogonal group associated to $V$ by $G^0\left(V\right)$, which is an index two subgroup in $G(V)$. By some classical Clifford theory, there is a canonical bijection
\[
	\iota:\Irr\left(G\left(V\right)\right)/\sim_{\det} \lra \Irr\left(G^0\left(V\right)\right)/\sim_\varepsilon
\]
given by restriction, where the LHS of the bijection is the set of equivalence classes of irreducible representations of $G\left(V\right)$ up to the determinant twist, and the RHS is the set of equivalence classes of irreducible representations of $G^0\left(V\right)$ up to the action of the outer automorphism. Given $\pi\in\Irr\left(G\left(V\right)\right)$, we shall use $\left[\pi\right]_{\det}$ to denote the equivalence class in $\Irr\left(G\left(V\right)\right)/\sim_{\det}$ containing $\pi$.

Since $G^0\left(V\right)$ and $H\left(W_0\right)$ are connected reductive groups over $\RR$, by the work of Langlands \cite{MR1011897} we have the LLC for these two groups, as we have recalled in Section \ref{SS:LLC}. However, instead of original $L$-parameters for $G^0\left(V\right)$, we prefer to use the so-called weak $L$-parameters which we now describe. By modulo the action of the outer automorphism, we obtain a finite to one surjective map
\begin{equation}\label{weak LLC G0}
	\calL^0: \bigsqcup_{V'}\Irr\left(G^0\left(V'\right)\right)/\sim_\varepsilon \lra \Phi^+\left(G^0\left(V\right)\right)/\sim_\varepsilon,
\end{equation}
from the original LLC for $G^0\left(V\right)$, where the disjoint union on the LHS runs over all $2n$-dimensional orthogonal spaces $V'$ of the same discriminant as $V$, and the RHS of the surjection is the set of equivalence classes of $L$-parameters of $G^0\left(V\right)$ up to $\Lgp{G(V)}={\rm O}_{2n}\left(\CC\right)$ conjugation. %\zjl{discuss here} Here 
By composing with the standard representation
\[
	{\rm O}_{2n}\left(\CC\right) \lra \GL_{2n}\left(\CC\right),
\] 
the set $\Phi^+\left(G^0\left(V\right)\right)/\sim_\varepsilon$ can be identified with the set 
\[
	\Phi^+\left(G\left(V\right)\right)\coloneqq \Big\{\phi:L_\RR\lra \GL_{2n}\left(\CC\right)~\big|~\phi \textit{ is semi-simple, orthogonal and } \det\left(\phi\right)=\chi_V\Big\}.
\]
This is the set of weak $L$-parameters we will make use of.

Consider the theta lift between $\left(G\left(V\right) ,H\left(W_0\right)\right)$, with respect to a non-trivial additive character $\psi_\RR$ and splitting characters $\left(\chi_V,\chi_W\right)$ as in Section \ref{ThetaLifts}. Our first goal here is to establish a weaker version of Theorem \ref{Prasad.Conj}, which partially describes this theta lift in terms of (weak) $L$-parameters for $G^0\left(V\right)$ and $H\left(W_0\right)$. 

\begin{theorem}\label{weakver.Prasad.Conj}
Let $\phi$ be a tempered weak $L$-parameter for $G^0\left(V\right)$, and $\pi$ an irreducible tempered representation of $G\left(V\right)$, such that %{\color{red}\textit{(specify Whittaker data!!)}}
\[
	\calL^0\left(\iota\left(\left[\pi\right]_{\det}\right)\right) = \phi.
\]
Suppose that the theta lift $\vartheta\left(\pi\right)$ of $\pi$ to $H\left(W_0\right)$ is non-zero. Then $\vartheta\left(\pi\right)$ is tempered and lies in the $L$-packet associated to
\[
	\vartheta\left(\phi\right) = \phi\chi_W^{-1}\chi_V+\chi_V.
\]
Moreover, if $\chi_W\not\subset\phi$, then both 
\[
		\begin{cases}
			\textit{the theta lift }\vartheta(\pi)\textit{ of $\pi$ to } H\left(W_0\right) ,\\
			\textit{the theta lift }\vartheta(\pi\otimes\det)\textit{ of $\pi\otimes\det$ to } H\left(W_0\right)
		\end{cases}
\]
are non-zero.

\end{theorem}

\begin{proof}
We first consider the case that $\phi$ is a discrete weak $L$-parameter, i.e. $\phi$ is multiplicity-free. We can write it as 
\[
	\phi = \calD_{2\alpha_1} + \calD_{2\alpha_2} + \cdots +\calD_{2\alpha_n}
\]
for some non-negative integers $\alpha_1>\alpha_2>\cdots>\alpha_n\geq 0$. Following Paul \cite[Sect.3.2]{MR2175409}, if the Harish-Chandra parameter of $\pi$ is of the form
\[
	\lambda = \left(a_1,a_2,\cdots;b_1,b_2,\cdots\right),
\]
where $a_i$ and $b_j$ are non-negative integers, $a_1>a_2>\cdots$ and $b_1>b_2>\cdots$, then the infinitesimal character of $\pi$ is precisely the orbit of $\lambda$ under the Weyl group action. Since the LLC preserves the infinitesimal character, we have
\[
	\left\{\alpha_1,\alpha_2,\cdots,\alpha_n\right\} = \left\{a_1,a_2,\cdots,b_1,b_2,\cdots\right\}.
\]
By \cite[Thm.15]{MR2175409}, $\vartheta\left(\pi\right)$ is a limit of discrete series, and the Harish-Chandra parameter of $\vartheta\left(\pi\right)$ is of the form
\[
	\lambda' = \left(a_1,a_2,\cdots;\cdots -b_2,-b_1\right).
\]
Note that by definition any $L$-parameter of $H\left(W_0\right)$ has trivial determinant. Then again by considering the infinitesimal character, one can see immediately that the $L$-parameter of $\pi$ must be $\vartheta\left(\phi\right)$ as predicted.
%\zjl{
%Write explicitly 
%\[
%\vartheta\left(\phi\right)= \calD_{2\alpha_1} + \calD_{2\alpha_2} + \cdots +\calD_{2\alpha_n}\oplus \chi_{\bC/\bR}^n
%\]
%?
%}

Next we consider the case that $\phi$ is tempered but not discrete. It follows from the LLC for $G^0\left(V\right)$ that any irreducible representation of $G^0\left(V\right)$ in $\iota\left(\left[\pi\right]_{\det}\right)$ is tempered but not a discrete series, and hence so is $\pi$. It is well known that $\pi$ can be embedded into a parabolic induction of a discrete series representation. Then by the induction principle of local theta correspondence \cite[Sect.5.2]{MR2175409} and the result in discrete series case, we know that the $L$-parameter of $\vartheta\left(\pi\right)$ is $\vartheta\left(\phi\right)$.

Finally we prove the last statement of this theorem. Let $W_{00}$ be the $(2n-2)$-dimensional symplectic space over $\RR$, and $H\left(W_{00}\right)$ be the corresponding symplectic group. If $\vartheta\left(\pi\right)=0$, then by the conservation relation \cite{MR3369906} the theta lift of $\pi\otimes \det$ to $H\left(W_{00}\right)$ is non-zero. By using the same argument as above, we can show that $\chi_W\subset\phi$. Likewise $\vartheta\left(\pi\otimes\det\right)=0$ also yields $\chi_W\subset\phi$. This completes the proof.
\end{proof}

With this theorem at hand, next we extend the LLC for real special even orthogonal groups to full even orthogonal groups using the same idea as in \cite[Sect.5]{zou2020local}. 
As mentioned before, the LLC depends on the choices of a Whittaker datum $\scrW$ of the quasi-split pure inner form of $G\left(V\right)$. Since the construction will also involve the LLC for $H\left(W_0\right)$, we need to choose a Whittaker datum $\scrW'$ of $H\left(W_0\right)$ as well. We shall make these choices according to the additive character $\psi_\RR$ (and some other auxiliary data) as explicated in Section \ref{Whittaker.Data}.

Let us deal with tempered representations first. Recall that by the weak LLC for $G^0\left(V\right)$ we have a finite to one surjective map $\calL^0$ as in (\ref{weak LLC G0}).
We define a map
\[
	\calL:\bigsqcup_{V'}\Irr_{temp}\left(G\left(V'\right)\right) \lra \Phi\left(G\left(V\right)\right)
\] 
by setting $\calL\left(\pi\right) = \calL^0\left(\iota\left(\left[\pi\right]_{\det}\right)\right)$ for $\pi\in\Irr_{temp}\left(G\left(V'\right)\right)$. It then follows from the properties of $\calL^0$ that $\calL$ is a finite to one surjective map. 
For each $L$-parameter $\phi\in\Phi\left(G\left(V\right)\right)$, to give a parametrization of the fibers
\[
	\Pi_\phi\left(G\left(V'\right)\right) \coloneqq \calL^{-1}\left(\phi\right)\cap \Irr\left(G\left(V'\right)\right),
\]
we appeal to the theta lift. Let $\pi\in\Pi_\phi\left(G\left(V'\right)\right)$. Consider the theta lift $\vartheta\left(\pi\right)$ of $\pi$ to $H\left(W_0\right)$. There are two possibilities:

\begin{itemize}
	\item \textit{Case 1}: $\vartheta\left(\pi\right)$ is non-zero. Then by Theorem \ref{weakver.Prasad.Conj}, we know that $\vartheta\left(\pi\right)\in\Pi_{\vartheta\left(\phi\right)}\left(H\left(W_0\right)\right)$, and $\calS_\phi$ can be regarded as a subgroup of $\calS_{\vartheta\left(\phi\right)}$. In this case we set
	\[
		\eta_\pi \coloneqq \eta_{\vartheta(\pi)}\Big|_{\calS_\phi},
	\]
	where $\eta_{\vartheta(\pi)}\in\widehat{\overline{\calS_{\vartheta(\phi)}}}$ is the character associated to $\vartheta\left(\pi\right)$ by the LLC for symplectic groups.

	\item \textit{Case 2}: $\vartheta\left(\pi\right)$ is zero. Then by the conservation relation \cite{MR3369906}, the theta lift $\vartheta\left(\pi\otimes\det\right)$ of $\pi\otimes \det$ to $H\left(W_0\right)$ is non-zero. Since $\left[\pi\right]_{\det}=\left[\pi\otimes\det\right]_{\det}$, it is easy to see from the definition that $\calL\left(\pi\right) = \calL\left(\pi\otimes\det\right)$. In the previous case we have already attached a character $\eta_{\pi\otimes\det}\in\wh{\calS_\phi}$ to $\pi\otimes\det$. In this case we set
	\[
	\eta_\pi \coloneqq \eta_{\pi\otimes\det}\cdot\kappa_\phi,
	\]
	where $\kappa_\phi\in\wh{\calS_\phi}$ is the character defined in equation (\ref{determitwis}).
\end{itemize}
Combining these two cases we obtain a map
\[
	\calJ^L_{\scrW}: \bigsqcup_{V'}\Pi_{\phi}\left(G\left(V'\right)\right) \lra \wh{\calS_\phi}
\]
by setting $\calJ^L_{\scrW}\left(\pi\right) = \eta_\pi$ for $\pi\in\Pi_\phi\left(G\left(V'\right)\right)$. Similar to \cite[Prop.5.10]{zou2020local}, it follows from the Howe duality and the conservation relation \cite{MR3369906} that the map $\calJ^L_{\scrW}$ is indeed a bijection. 

\begin{remark}\label{R:CompatibleRealEvenOvsSO}
The following two properties of the map $\calJ^L_{\scrW}$ are worth noting.
\begin{enumerate}
	\item For any $\pi\in\Pi_\phi\left(G\left(V'\right)\right)$, we have
	\begin{equation*}%\label{E:App.det.twist}
		\calJ^L_{\scrW}\left(\pi\otimes\det\right) = \calJ^L_{\scrW}\left(\pi\right) \cdot \kappa_\phi.
	\end{equation*}  
	Indeed, in the case that $\vartheta\left(\pi\right)=0$ or $\vartheta\left(\pi\otimes\det\right)=0$, this equality is a direct consequence of the construction. The proof of our main local result Theorem \ref{Local-Compare} (more precisely the proof of Proposition \ref{Final.Jigsaw}) will only involve this special case. In the case that both $\vartheta\left(\pi\right)$ and $\vartheta\left(\pi\otimes\det\right)$ are non-zero, we can appeal to the strength of Theorem \ref{ThmB} as follows. Similar to Section \ref{SS:Multi-free}, when the $L$-parameter $\phi$ is of good parity, one can suitably globalize $\pi$ %\zjl{need discrete} 
 to a cuspidal automorphic representation $\dot{\pi}$ with generic $A$-parameter $\dot{\phi}$, such that:
	\begin{itemize}
		\item at a place $v$, the localizations of $\dot{\pi}$ and $\dot{\phi}$ are $\pi$ and $\phi$;

		\item at an auxiliary finite place $w$, the localization map $\calS_{\dot{\phi}}\lra\calS_{\dot{\phi}_w}$ is an isomorphism.
	\end{itemize}
	Let $\dot{\pi}'$ be the cuspidal automorphic representation obtained from $\dot{\pi}$ by replacing $\dot{\pi}_v$ and $\dot{\pi}_w$ by $\dot{\pi}_v\otimes\det$ and $\dot{\pi}_w\otimes\det$. Then applying Theorem \ref{ThmB} to $\dot{\pi}$ and $\dot{\pi}'$, one gets
	\[
		\eta_{\dot\pi_v}\cdot\eta_{\dot\pi_w} = \eta_{\dot\pi_v\otimes\det}\cdot\eta_{\dot\pi_w\otimes\det}.
	\]
	Since the desired equality holds for $\dot\pi_w$ (see Remark \ref{LLC}(1)), it also holds for $\pi=\dot\pi_v$. For general $\phi$ the desired conclusion follows from the compatibility of the LLC with parabolic inductions.

	\item A priori, this map $\calJ^L_{\scrW}$ depends on the choice of the additive character $\psi_\RR$. However, as suggested by the notation it only depends on the choice of the Whittaker datum $\scrW$ of the quasi-split pure inner form of $G\left(V\right)$ but not on $\psi_\RR$. This is a consequence of the scaling property of the Weil representation. Similar to \cite[II.Cor.6.2 \& IV.Prop.1.9]{kudla1996notes}, an easy computation shows that for any $\pi\in\Irr\left(G\left(V\right)\right)$, we have
	\[
		\vartheta_{\psi_{\RR,a}}\left(\pi\right) \simeq \vartheta_{\psi_\RR}\left(\pi\right)^{\delta_a}.
	\]
	Here $a\in\RR^\times$, $\psi_{\RR,a}\coloneqq\psi_\RR\left(a\cdot-\right)$ and $\delta_a$ is an element in $\GL\left(W_0\right)$ such that
	\[
	\left\langle \delta_a\left(v\right),\delta_a\left(v'\right)\right\rangle_{W_0} = a\cdot\left\langle v, v'\right\rangle_{W_0}
	\]
	for any $v, v'\in W_0$. The subscripts ``$_{\psi_{\RR,a}}$'' and ``$_{\psi_\RR}$'' indicate the additive characters used in the definition of theta lifts. Let $\scrW'_a$ be the Whittaker datum of $H\left(W_0\right)$ determined by $\psi_{\RR,a}$ as in Section \ref{Whittaker.Data}. Then it follows from \cite[Thm.4.3]{MR3194648} that
	\[
		\calJ^L_{\scrW'}\left(\vartheta_{\psi_\RR}\left(\pi\right)\right) = \calJ^L_{\scrW'_a}\left(\vartheta_{\psi_{\RR,a}}\left(\pi\right)\right),
	\]
	and in particular these two characters have the same restriction to $\calS_\phi$.
\end{enumerate}	
\end{remark}

After furnishing tempered representations with the maps $\calL$ and $\calJ^L_{\scrW}$, we can extend these maps to all irreducible representations in a standard manner similar to \cite{MR3271238}. Although full even orthogonal groups are disconnected, the Langlands' classification is still valid by \cite[Sect.3.2]{MR2175409}. To be more precise, for any irreducible non-tempered representation $\pi\in \Irr\left(G\left(V\right)\right)$, there is a standard module 
\begin{equation*}
	\Ind_{P}^{G\left(V\right)}\left(\tau_1|\cdot|^{s_1}\boxtimes\cdots\boxtimes\tau_r|\cdot|^{s_r}\boxtimes\pi_{0}\right)
\end{equation*}
of $G\left(V\right)$, where 
\begin{itemize}
	\item $P$ is a parabolic subgroup of $G\left(V\right)$, with a Levi component
	\begin{equation*}
		L\simeq \GL_{d_1}(\RR)\times\cdots\times \GL_{d_m}(\RR)\times G\left(V_0\right);
	\end{equation*}
	here $V_0$ is some orthogonal space in the Witt tower containing $V$;
	\item $\tau_i$ is an irreducible (limit of) discrete series of $\GL_{d_i}(\RR)$, and $s_i$ is a positive real number; 
	\item $\left\{\tau_i|\cdot|^{s_i}\right\}_i$ is ordered such that 
	\[
	s_1\geq \cdots\geq s_r>0;
	\]
	\item $\pi_0$ is an irreducible tempered representation of $G\left(V_0\right)$;
\end{itemize}
such that $\pi$ is the unique irreducible quotient of this standard module. Let $\phi_{\tau_i}$ be the $L$-parameter of $\tau_i$, and $\phi_0=\calL\left(\pi_0\right)$. We set
\[
	\calL\left(\pi\right) = \left(\phi_{\tau_1}|\cdot|^{s_1} + \cdots + \phi_{\tau_r}|\cdot|^{s_r}\right) + \phi_0 + \left(\phi_{\tau_1}|\cdot|^{s_1} + \cdots + \phi_{\tau_r}|\cdot|^{s_r}\right)^\vee.
\]
Then as explicated in \cite[Sect.8]{MR3202556}, there is a natural isomorphism $\calS_\phi\simeq\calS_{\phi_0}$. Let $\eta_0 = \calJ^L_{\scrW}\left(\pi_0\right)$. Under this identification of component groups, we define
\[
	\calJ^L_{\scrW}\left(\pi\right) = \eta_0.
\]
Since these $P$, $\tau_i$, $s_i$ and $\pi_0$ are uniquely determined by $\pi$, the $L$-parameter $\calL\left(\pi\right)$ and the character $\calJ^L_{\scrW}\left(\pi\right) \in\wh{\calS_\phi}$ are well-defined. We conclude the above discussions as follows.

\begin{theorem}\label{T:LLC4RealfullO}
There is a finite to one surjective map
\[
\calL:\bigsqcup_{V'}\Irr\left(G\left(V'\right)\right) \lra \Phi^+\left(G\left(V\right)\right),
\]
where the disjoint union runs over all $2n$-dimensional orthogonal spaces $V'$ of the same discriminant as $V$. For each $L$-parameter $\phi\in\Phi^+\left(G\left(V\right)\right)$, we denote
\[
\Pi_\phi\left(G\left(V'\right)\right) \coloneqq \calL^{-1}\left(\phi\right)\cap \Irr\left(G\left(V'\right)\right),
\]
and we call it the $L$-packet of $G'$ associated to $\phi$. There is a bijection (depends on the choice of the Whittaker datum $\scrW$)
\begin{equation*}
	\calJ_{\scrW}^L:\bigsqcup_{V'}\Pi_{\phi}\left(G\left(V'\right)\right) \lra \wh{\calS_\phi}.
\end{equation*}
%Furthermore, these maps $\calL$ and $\calJ^L_{\scrW}$ are compatible with the LLC for special even orthogonal groups in the sense of Remark \ref{R:CompatibleRealEvenOvsSO}.
\end{theorem}

From our construction, one can see immediately that:

\begin{corollary}
Under the LLC provided by Theorem \ref{T:LLC4RealfullO} for real full even orthogonal groups, Theorem \ref{Prasad.Conj} holds for dual pairs $\left(G\left(V'\right), H\left(W_0\right)\right)$.
\end{corollary}

Finally, recall that Arthur has already established a tempered LLC for quasi-split real full even orthogonal groups, namely a finite to one surjective map
\[
\calL^A: \Irr_{temp}\left(G\left(V^+\right)\right) \lra \Phi\left(G\left(V^+\right)\right),
\]
together with a bijection (depends on the choice of a Whittaker datum $\scrW$)
\[
	\calJ_{\scrW}^A: \Pi_{\phi}\left(G\left(V^+\right)\right) \lra \wh{\overline{\calS_{\phi}}}
\]
on each fiber $\Pi_{\phi}\left(G\left(V^+\right)\right)$ of $\phi\in\Phi\left(G\left(V^+\right)\right)$.
We can justify our construction of the LLC for real full even orthogonal groups by comparing it with Arthur's. Since Arthur's LLC is compatible with the LLC for $G^0\left(V^+\right)$, for any $\pi\in\Irr_{temp}\left(G\left(V^+\right)\right)$ we have
\[
	\calL^A\left(\pi\right) = \calL^0\left(\iota\left(\left[\pi\right]_{\det}\right)\right) = \calL\left(\pi\right).
\]
For any $\phi\in\Phi\left(G\left(V^+\right)\right)$ and $\pi\in\Pi_{\phi}\left(G\left(V^+\right)\right)$, to compare $\calJ_{\scrW}^A\left(\pi\right)$ and $\calJ_{\scrW}^L\left(\pi\right)$, again we appeal to the global method. When the $L$-parameter $\phi$ is of good parity, using the same argument as in Section \ref{SS:Multi-free}, we can suitably globalize $\pi$ to a cuspidal automorphic representation $\dot{\pi}$ of a globally quasi-split even orthogonal group $G(\dot{V}^+)$, with generic $A$-parameter $\dot{\phi}$, such that at a place $v$, the localizations of $\dot{\pi}$ and $\dot{\phi}$ are $\pi$ and $\phi$. Then comparing the original Arthur's multiplicity formula and our version (Theorem \ref{ThmB}) for $L^2_{\dot\phi}\left(G(\dot{V}^+)\right)$, we deduce that
\[
	\calJ_{\scrW}^A\left(\pi\right) = \calJ_{\scrW}^L\left(\pi\right)
\]
For general $\phi$ the desired equality follows from the compatibility of the LLC with parabolic inductions.

% section prasad_s_conjecture_even_orthogonal_symplectic_case (end)

\section{An irreducibility result of some induced representations}\label{App.Irr}
In this appendix, we sketch a proof of Lemma \ref{Irr-nonStdMod} in the case when $F=\RR$. We have $E=\RR$ in Case O and $E=\CC$ in Case U. We will prove it in a more general context.

Recall that an irreducible representation of $L_E$ is said to be almost tempered and positive if it is of the form $\phi|\cdot|^s$, where $\phi$ is a representation of $L_E$ with bounded image, and $0< s< 1/2$ is a real number. Let $\psi$ be a local $A$-parameter for $H$. We assume 
\[
\psi=\varphi+\psi_0+\left(\varphi^c\right)^\vee, 
\]
where
\begin{itemize}
	\item $\psi_0$ is a local $A$-parameter for $H_0=H(W_0)$, which is of good parity; here $W_0$ is a $c$-skew-Hermitian space in the Witt tower containing $W$;
	\item $\varphi$ is a $k$-dimensional representation of $L_E$ whose irreducible summands are either almost tempered and positive, or tempered but not (conjugate) self-dual with the same parity as that of $\psi$. 
\end{itemize}
Let $\tau$ be the irreducible representation of $\GL_k(E)$ associated to $\varphi$, and $Q$ the standard parabolic subgroup of $H$ with Levi component $L\simeq \GL_k(E)\times H_0$. We shall prove that:
\begin{theorem}\label{A.Irr-nonStdMod}
	For any irreducible unitary representation $\sigma_0$ in the $A$-packet $\Pi_{\psi_0}(H_0)$, the induced representation $\Ind_Q^H(\tau\boxtimes\sigma_0)$ is irreducible.
\end{theorem}
In \cite{MR3907735}, Gan-Ichino proved a similar statement for odd orthogonal groups. Mimicking their proof, we briefly describe the strategies to prove the theorem. Readers may consult \cite[Sect.3I]{MR3907735} for a treatment in full details. The ingredient of the proof is the normalized intertwining operators. Recall that for a real reductive group $G$, a parabolic subgroup $P$ of $G$, and an irreducible representation $\pi$ of the Levi component of $P$, the induced representation $\Ind_P^G\left(\pi\right)$ is called a standard module for $G$ if $\pi$ satisfies certain positivity conditions. For such an induced representation, one can define a (normalized) intertwining operator
\[
R_{\overline{P}|P}(\pi):\Ind_P^G\left(\pi\right)\lra\Ind_{\overline{P}}^G\left(\pi\right),
\]
where $\overline{P}$ is the parabolic subgroup of $G$ opposite to $P$, such that the image of $R_{\overline{P}|P}(\pi)$ is the unique irreducible quotient of $\Ind_P^G\left(\pi\right)$. In the proof of Theorem \ref{A.Irr-nonStdMod}, we shall realize the representation $\Ind_Q^H(\tau\boxtimes\sigma_0)$ as the image of a standard module for $H$.

Firstly, we decompose the representation $\tau$. It follows from our assumptions that we may write $\varphi$ as a summation of some subrepresentations
\[
\varphi=\varphi_1+\cdots+\varphi_r,
\]
satisfying the following conditions:
\begin{itemize}
	\item[-]Suppose we are in Case O, 
	\begin{itemize}
		\item[$\bullet$] Each $\varphi_i$ is of the form $\phi_i|\cdot|^{\nu_i}$, where $\phi_i$ is either $\sgn^{\delta_i}$ for some $\delta_i\in\mathbb{Z}/2\mathbb{Z}$, or $\calD_{2\alpha_i}$ for some $\alpha_i\in \frac{1}{2}\ZZ\backslash \{0\}$; and $\nu_i$ is a complex number;
		\item[$\bullet$] If $\nu_i=0$ for some $i$, then we have $\varphi_i=\calD_{2\alpha_i}$ for some $\alpha_i\in\ZZ+\frac{1}{2}$.
	\end{itemize}
	\item[-]Suppose we are in Case U, 
	\begin{itemize}
		\item[$\bullet$] Each $\varphi_i$ is of the form $\chi_{2\alpha_i}|\cdot|^{\nu_i}$ for some $\alpha_i\in\frac{1}{2}\ZZ$ and $\nu_i\in\CC$;
		\item[$\bullet$] If $\nu_i=0$ for some $i$, then we have $\alpha_i\in\ZZ+\frac{\dim W}{2}$.
	\end{itemize}
\end{itemize}
In both cases, the summation can be ordered such that
		\[
		\frac{1}{2}>\Re(\nu_1)\geq\cdots\geq\Re(\nu_r)\geq 0.
		\]
Let $k_i=\dim\varphi_i$, $\tau_i$ be the irreducible representation of $\GL_{k_i}(E)$ corresponding to $\varphi_i$ by the LLC for general linear groups, and $\tau_\varphi=\tau_1\boxtimes\cdots\boxtimes\tau_r$. It is easy to see that there is a parabolic subgroup $Q_\varphi$ of $\GL_k(E)$, with Levi component
\[
L_\varphi\simeq \GL_{k_1}(E)\times\cdots\times \GL_{k_r}(E),
\]
such that 
\[
\tau\simeq\Ind_{Q_\varphi}^{\GL_k(E)}\left(\tau_\varphi\right).
\]
Let $Q_1$ be the parabolic subgroup of $H$ with Levi component $L_1\simeq L_\varphi\times H_0$, such that $Q_1\subset Q$, and $Q_1\cap L=Q_\varphi\times H_0$. Then by induction in stages, we have
\[
\Ind_Q^H\left(\tau\boxtimes\sigma_0\right)\simeq\Ind_{Q_1}^H\left(\tau_\varphi\boxtimes\sigma_0\right).
\]

Next we deal with the irreducible representation $\sigma_0$. By the Langlands' classification, we know that $\sigma_0$ is the unique irreducible quotient of a standard module (for $H_0$)
\begin{equation*}
	\Ind_{Q_{\psi_0}}^{H_0}\left(\tau'_1\boxtimes\cdots\boxtimes\tau'_m\boxtimes\sigma_{00}\right),
\end{equation*}
where 
\begin{itemize}
	\item $Q_{\psi_0}$ is a parabolic subgroup of $H_{0}$, with Levi component
	\begin{equation*}
		L_{\psi_0}\simeq \GL_{d_1}(E)\times\cdots\times \GL_{d_m}(E)\times H_{00};
	\end{equation*}
	here $H_{00}=H(W_{00})$ for some space $W_{00}$ in the Witt tower containing $W_0$;
	\item $\tau'_i$ is an irreducible essentially (limit of) discrete series of $\GL_{d_i}(E)$, which is of the form
	\[
	\tau'_i=\tau''_i|\cdot|^{s_i}
	\]
	for some irreducible (limit of) discrete series $\tau''_i$ of $\GL_{d_i}(E)$ and $s_i>0$; since $\sigma_0$ lies in the local $A$-packet $\Pi_{\psi_0}(H_0)$, we can further conclude %{\color{red}\textit{Using infinitestimal character? reference?}} 
 that $s_i\in\frac{1}{2}\ZZ$;
	\item $\{\tau'_i\}_i$ is ordered such that 
	\[
	s_1\geq \cdots\geq s_m>0;
	\]
	\item $\sigma_{00}$ is a tempered %limit of discrete series
 representation of $H_{00}$ %{\color{red}\textit{(why limit of discrete series?? I guess should be tempered)}} \zjl{I think it is temepered}.
\end{itemize}
Moreover, if we let $\sigma_{0,\psi_0}=\tau'_1\boxtimes\cdots\boxtimes\tau'_m\boxtimes\sigma_{00}$, and
\begin{equation}\label{eq.inner.intertwin}
	R_{\overline{Q}_{\psi_0}|Q_{\psi_0}}(\sigma_{0,\psi_0}):\Ind_{Q_{\psi_0}}^{H_0}\left(\sigma_{0,\psi_0}\right)\lra\Ind_{\overline{Q}_{\psi_0}}^{H_0}\left(\sigma_{0,\psi_0}\right)
\end{equation}
be the (normalized) intertwining operator, then $\sigma_0$ is just the image of this operator $R_{\overline{Q}_{\psi_0}|Q_{\psi_0}}(\sigma_{0,\psi_0})$.

Now we come to the key step. Applying the functor $\Ind_{Q_1}^H\left(\tau_\varphi\boxtimes~\cdot~\right)$ to the intertwining map (\ref{eq.inner.intertwin}), we get
\[
\Ind_{Q_1}^H\left(\mathbbm{1}_{\tau_\varphi}\boxtimes R_{\overline{Q}_{\psi_0}|Q_{\psi_0}}(\sigma_{0,\psi_0})\right):\Ind_{Q_1}^H\left(\tau_\varphi\boxtimes\Ind_{Q_{\psi_0}}^{H_0}\left(\sigma_{0,\psi_0}\right)\right)\lra\Ind_{Q_1}^H\left(\tau_\varphi\boxtimes\Ind_{\overline{Q}_{\psi_0}}^{H_0}\left(\sigma_{0,\psi_0}\right)\right).
\]
Let $Q_2$ be the parabolic subgroup of $H$ with Levi component $L_2\simeq L_\varphi\times L_{\psi_0}$, such that $Q_2\subset Q_1$, and $Q_2\cap L_1=L_\varphi\times Q_{\psi_0}$. Similarly, let $Q'_2$ be the parabolic subgroup of $H$ with the same Levi component as $Q_2$, such that $Q'_2\subset Q_1$, and $Q'_2\cap L_1=L_\varphi\times \overline{Q}_{\psi_0}$. Then, by induction in stages and properties of (normalized) intertwining operators, we can rewrite this intertwining map as 
\[
R_{Q'_2|Q_2}\left(\tau_\varphi\boxtimes\sigma_{0,\psi_0}\right):\Ind_{Q_2}^H\left(\tau_\varphi\boxtimes\sigma_{0,\psi_0}\right)\lra\Ind_{Q'_2}^H\left(\tau_\varphi\boxtimes\sigma_{0,\psi_0}\right).
\]
Since the functor $\Ind_{Q_1}^H\left(\tau_\varphi\boxtimes~\cdot~\right)$ is exact, to show that $\Ind_{Q_1}^H\left(\tau_\varphi\boxtimes\sigma_0\right)$ is irreducible, it is sufficient to show that the image of $R_{Q'_2|Q_2}\left(\tau_\varphi\boxtimes\sigma_{0,\psi_0}\right)$ is irreducible. Let $Q_3$ be the parabolic subgroup of $H$ with the same Levi component $L_3$ as $Q_2$ and $Q'_2$, such that 
\[
	\Re\left(\omega_{\varphi,\psi_0}\right)\in\overline{\gotha}_{Q_3}^{*,+},
\]
where $\omega_{\varphi,\psi_0}$ is the central character of $\tau_\varphi\boxtimes\sigma_{0,\psi_0}$, and $\overline{\gotha}_{Q_3}^{*,+}$ is as defined in \cite[Sect.3H]{MR3907735}. Then by the properties of (normalized) intertwining operators, we have a commutative diagram
		\[
		\begin{CD}
			\Ind_{Q_2}^H\left(\tau_\varphi\boxtimes\sigma_{0,\psi_0}\right) @>R_{Q'_2|Q_2}>> \Ind_{Q'_2}^H\left(\tau_\varphi\boxtimes\sigma_{0,\psi_0}\right)\\
			@VV R_{Q_3|Q_2} V @VV R_{\overline{Q}_3|Q'_2} V\\
			\Ind_{Q_3}^H\left(\tau_\varphi\boxtimes\sigma_{0,\psi_0}\right) @>R_{\overline{Q}_3|Q_3}>> \Ind_{\overline{Q}_3}^H\left(\tau_\varphi\boxtimes\sigma_{0,\psi_0}\right)
		\end{CD}.
		\]
Similar to \cite[Lem.3.10]{MR3907735}, we have the following lemma.

\begin{lemma}\label{exchange}
	The (normalized) intertwining operators $R_{Q_3|Q_2}\left(\tau_\varphi\boxtimes\sigma_{0,\psi_0}\right)$ and $R_{\overline{Q}_3|Q'_2}\left(\tau_\varphi\boxtimes\sigma_{0,\psi_0}\right)$ are isomorphisms.
\end{lemma}

\begin{proof}
Same as the proof of \cite[Lem.3.10]{MR3907735}, the intertwining operator $R_{Q_3|Q_2}\left(\tau_\varphi\boxtimes\sigma_{0,\psi_0}\right)$ can be decomposed as the composition of a sequence of intertwining operators
\[
	R_{Q_3|Q_2} = R_{R_t|R_{t-1}}\circ \cdots \circ R_{R_2|R_1}\circ R_{R_1|R_0},
\]
where $R_0= Q_2, R_1, \cdots, R_t=Q_3$ are parabolic subgroups of $H$. For each $R_{R_k|R_{k-1}}$ on the RHS, there exists some $1\leq i\leq r$ and $1\leq j\leq m$, such that $R_{R_k|R_{k-1}}$ is essentially (a parabolic induction of) the intertwining operator 
\[
\Ind_{P_{i,j}}^{\GL_{k_i+d_j}(E)}\left(\tau_i\boxtimes\tau'_j\right)\lra\Ind_{\overline{P}_{i,j}}^{\GL_{k_i+d_j}(E)}\left(\tau_i\boxtimes\tau'_j\right),
\]
where $P_{i,j}$ is a parabolic subgroup of $\GL_{k_i+d_j}(E)$ with the Levi component $M_{i,j}\simeq \GL_{k_i}(E)\times \GL_{d_j}(E)$. It follows from \cite[Thm.6.19]{MR590291} and the conditions on $\tau_i$, $\tau'_j$ that $\Ind_{P_{i,j}}^{\GL_{k_i+d_j}(E)}\left(\tau_i\boxtimes\tau'_j\right)$ is irreducible. Then one can conclude that each $R_{R_k|R_{k-1}}$ is an isomorphism, hence so is their composition $R_{Q_3|Q_2}\left(\tau_\varphi\boxtimes\sigma_{0,\psi_0}\right)$. Similarly, one can prove that $R_{\overline{Q}_3|Q'_2}\left(\tau_\varphi\boxtimes\sigma_{0,\psi_0}\right)$ is also an isomorphism.
\end{proof}

Therefore, up to an isomorphism, the image of $R_{Q'_2|Q_2}\left(\tau_\varphi\boxtimes\sigma_{0,\psi_0}\right)$ is the same as the image of $R_{\overline{Q}_3|Q_3}\left(\tau_\varphi\boxtimes\sigma_{0,\psi_0}\right)$. Notice that $\Ind_{Q_3}^H\left(\tau_\varphi\boxtimes\sigma_{0,\psi_0}\right)$ is already very close to a standard module for $H$. In general, by ``collapsing the same exponents'', we can find a parabolic subgroup $Q_4$ of $H$ with Levi component $L_4$, such that $Q_3\subset Q_4$, $L_3\subset L_4$, and 
\[
\Re(\omega_{\sigma_\psi})\in\gotha_{Q_4}^{*,+},
\]
where $\omega_{\sigma_\psi}$ is the central character of the irreducible representation $\sigma_\psi\coloneqq\Ind_{Q_3\cap L_4}^{L_4}\left(\tau_\varphi\boxtimes\sigma_{0,\psi_0}\right)$. It follows that $\Ind_{Q_4}^H\left(\sigma_\psi\right)$ is a standard module for $H$, and up to an isomorphism, the image of $R_{\overline{Q}_3|Q_3}\left(\tau_\varphi\boxtimes\sigma_{0,\psi_0}\right)$ is the same with the image of 
\[
R_{\overline{Q}_4|Q_4}(\sigma_\psi):\Ind_{Q_4}^H\left(\sigma_\psi\right)\lra\Ind_{\overline{Q}_4}^H\left(\sigma_\psi\right),
\]
hence is irreducible. This completes the proof.
%\fi%%%%%%%%%%%%%%%%%%%%%%%%%%%%%%%%%%%%%%%%%%%%%%%%%%%%%%%%%%%%%%%%%%%%%%%%%%%%%%%%%%%%%%%%%%%%%%%%%%%%%

\section{On irreducible self-dual Galois representations}\label{self.dual.Gal.rep}
In this appendix, we consider irreducible self-dual representations of the Weil group of a local field. The results in this appendix will serve as some supplements to our proof of Corollary \ref{Globalization-2} in Case O.

Let $F$ be a non-Archimedean local field with characteristic $0$. Let $W_F$ be the Weil group of $F$. We use $\Irr^{\star,d}(F)$ (resp. $\Irr^{S,d}(F)$, resp. $\Irr^{O,d}(F)$) to denote the irreducible self-dual (resp. symplectic, resp. orthogonal) representations of $W_F$ of dimension $d$. 
\begin{lemma}
Suppose that $F$ is a finite extension of $\QQ_p$. Let $\pi$ be a uniformizer of $F$, and $k$ be the residue field of $F$. Then 
\begin{equation*}
F^\times\simeq\pi^\ZZ\times k^\times\times\mu_{p^\infty}(F)\times\ZZ_p^d,
\end{equation*}
where $\mu_{p^\infty}(F)$ is the group of roots of unity of $p$-power order in $F$, and $d$ is the degree of $F$ over $\QQ_p$.
\end{lemma}
\begin{proof}
See \cite[Chap.2.5]{MR1697859}.
\end{proof} 
\begin{lemma}
Let $E$ be a finite unramified extension of $F$, and $s=\Frob_F\in W_F$. Then $E$ is a cyclic extension over $F$, and 
\begin{equation*}
W_F/W_E\simeq\langle\overline{s}\rangle,
\end{equation*}
where $\overline{s}$ is the image of $s$ in $\Gal(E/F)$. Moreover, the following diagram
\begin{equation*}
\begin{CD}
W_E @>{Ad(s)}>> W_E\\
@VV{r_E}V @VV{r_E}V\\
E^\times @>{\overline{s}}>> E^\times
\end{CD}
\end{equation*}
commutes, where $r_E$ is the reciprocity law homomorphism of class field theory.
\end{lemma}
\begin{proof}
See \cite[p.4]{MR546607}.
\end{proof}
\begin{theorem}
For any positive integer $n$, there exists infinitely many symplectic (resp. orthogonal) irreducible representations of $W_F$ of dimension $2n$.
\end{theorem}
\begin{proof}
Let $E_0$ be the unique unramified extension of $F$ of degree $n$, and $E$ be the unique unramified quadratic extension of $E_0$. Then $E/F$ is also unramified. We have
\begin{equation*}
E^\times/\Nm_{E/E_0}(E^\times)\simeq\left(\ZZ/2\ZZ\right)\times\left(k_E^\times/k_{E_0}^\times\right)\times\big(\mu_{p^\infty}(E)/\mu_{p^\infty}(E_0)\big)\times\left(\ZZ_p^{dn}\times T\right),
\end{equation*}
where $T$ is some finite torsion group. Fix a primitive element $x$ in $k_E^\times$, and a primitive $(q^n+1)$-th root of unity $\zeta$. Let
\begin{align*}
X_\zeta:k_E^\times/k_{E_0}^\times&\lra\CC^\times\\
\overline{x}&\longmapsto\zeta.
\end{align*}
Then for any character $\chi$ of $\ZZ_p^{dn}\times T$,
\begin{equation*}
\sgn\boxtimes X_\zeta\boxtimes\mathbbm{1}\boxtimes\chi\quad\quad(\textit{resp. } \mathbbm{1}\boxtimes X_\zeta\boxtimes\mathbbm{1}\boxtimes\chi)
\end{equation*}
gives a character $\wt\chi$ of $E^\times$, which satisfies:
\begin{enumerate}
\item $\wt\chi|_{E_0^\times}=\omega_{E/E_0}$ (resp. $\mathbbm{1}$);
\item for any $1\leq i<2n$, we have $\wt\chi^{\overline{s}^{i}}\neq\wt\chi$.
\end{enumerate}
Therefore $\Ind_{W_E}^{W_F}\wt\chi$ is symplectic (resp. orthogonal) and irreducible of dimension $2n$. This construction gives an injection
\begin{equation*}
\Big(\widehat{\ZZ_p^{dn}}/\sim\Big)\lra\Irr^{S,2n}(F).~~(\textit{resp. }\Irr^{O,2n}(F))
\end{equation*}
Here we regard a character of $\ZZ_p^{dn}$ as a character of $\ZZ_p^{dn}\times T$ which is trivial on the torsion group $T$, and define $\chi_1\sim\chi_2$ if there is some $1\leq i\leq2n$, such that $\wt\chi_1^{\overline{s}^{i}}=\wt\chi_2$. Since $\widehat{\ZZ_p^{dn}}/\sim$ is an infinite set, we are done.
\end{proof}

Now we consider irreducible self-dual representations of $W_F$ with arbitrary dimension.
\begin{proposition}
Suppose $F$ is a finite extension of $\QQ_p$ and $p\neq2$. Then there is no irreducible self-dual representation of $W_F$ with odd dimension greater than $1$.
\end{proposition}
\begin{proof}
See \cite[Prop.4]{MR1670568}.
\end{proof}
We now assume that $F$ is a finite extension of $\QQ_2$, with the residue field $k_F$. Let $d$ be the degree of $F$ over $\QQ_2$, and $d_u$ be the degree of $k_F$ over $\FF_2$.
\begin{theorem}
Let $N$ be an arbitrary positive integer. Suppose 
\begin{equation*}
2^d>N.
\end{equation*}
Then, we have
\begin{equation*}
|\Irr^{\star,1}(F)|>N.
\end{equation*}
\end{theorem} 
\begin{proof}
Just notice that
\begin{equation*}
F^\times/(F^\times)^2\simeq \left(\ZZ/2\ZZ\right)\times\Big(\mu_{2^\infty}(F)/\left(\mu_{2^\infty}(F)\right)^2\Big)\times\left(\ZZ/2\ZZ\right)^d.
\end{equation*}
\end{proof}
\begin{theorem}\label{App.B.6}
Fix a positive integer $n$. Let $N$ be an arbitrary positive integer. Suppose 
\begin{equation*}
\frac{2^{d_u}-1}{n}>N.
\end{equation*}
Then, we have
\begin{equation*}
|\Irr^{\star,n}(F)|>N.
\end{equation*}
\end{theorem}
\begin{proof}
Let $E$ be the unique unramified extension of $F$ of degree $n$. Then 
\begin{equation*}
E^\times\simeq\pi^\ZZ\times k_E^\times\times U_E^1,
\end{equation*}
and 
\begin{equation*}
U_E^1/U_E^2\simeq k_E,
\end{equation*}
where
\begin{equation*}
U_E^{i}=1+\pi^{i}\calO_E,
\end{equation*}
and $\calO_E$ is the ring of integers of $E$. Fix $x\in k_E$ such that $\{x,\overline{s}(x),\overline{s}^2(x),\cdots\}$ is a basis of $k_E$ over $k_F$. Then
\begin{equation*}
k_E=k\cdot x+k\cdot\overline{s}(x)+\cdots+k\cdot\overline{s}^{n-1}(x),
\end{equation*}
and the Pontryagin dual of $k_E$ can be identified with the set of $n$-tuples of characters of $k_F$ by
\begin{equation*}
(\chi_1,\cdots,\chi_n)\longmapsto\Big(\lambda\cdot\overline{s}^{i}(x)\mapsto\chi_i(\lambda)~,~\lambda\in k\Big).
\end{equation*}
Under this identification, $\overline{s}$ acts on the Pontryagin dual of $k_E$ by
\begin{equation*}
\overline{s}:(\chi_1,\chi_2,\cdots,\chi_n)\longmapsto(\chi_n,\chi_1,\cdots,\chi_{n-1}).
\end{equation*}
We have injectives maps
\begin{equation*}
\wh {k_F}\backslash\{\mathbbm{1}\}\lra\wh {k_E}\lra \wh{E^\times},
\end{equation*}
where the first map is given by
\begin{equation*}
\chi\mapsto(\chi,\mathbbm{1},\cdots,\mathbbm{1}),
\end{equation*}
and the second map is induced by the natural projection
\begin{equation*}
E^\times\lra U_E^1\lra U_E^1/U_E^2\simeq k_E.
\end{equation*}
We denote the image of $\chi$ in $\wh{E^\times}$ by $\wt\chi$. By our construction, $\wt\chi$ satisfies:
\begin{enumerate}
\item $\wt\chi$ is quadratic;
\item for any $1\leq i<n$, we have $\wt\chi^{\overline{s}^{i}}\neq\wt\chi$.
\end{enumerate}
Therefore $\Ind_{W_E}^{W_F}\wt\chi$ is self-dual and irreducible of dimension $n$. This construction gives an injection
\begin{equation*}
\Big(\widehat{k_F}\backslash\{\mathbbm{1}\}\Big)/\sim~\lra\Irr^{\star,n}(F).
\end{equation*}
Here we define $\chi_1\sim\chi_2$ if there is some $1\leq i\leq n$, such that $\wt\chi_1^{\overline{s}^{i}}=\wt\chi_2$. Notice that there are at most $n$ elements in each equivalence class. Hence the LHS of this injection has at least
\begin{equation*}
\frac{2^{d_u}-1}{n}
\end{equation*} 
elements. By our assumption, we are done.
\end{proof} 

\section{On irreducible conjugate self-dual Galois representations}\label{conj.self.dual.gal.rep}
In this appendix, we consider irreducible conjugate self-dual representations of the Weil group of a local field. The results in this appendix will serve as some supplements to our proof of Corollary \ref{Globalization-2} in Case U.

Let $F$ be a non-Archimedean local field with characteristic $0$, and $E$ be a quadratic field extension of $F$. Let $W_F$ and $W_E$ be the Weil group of $F$ and $E$ respectively. We fix an element $s\in W_F\backslash W_E$. We use $\Irr_F^{S,d}(E)$ (resp. $\Irr_F^{O,d}(E)$) to denote the irreducible conjugate symplectic (resp. conjugate orthogonal) representation of $W_E$ of dimension $d$.

\begin{lemma}
Suppose that:
\begin{enumerate}
\item $F$ is a finite extension of $\QQ_p$ and $p\neq2$;
\item $E/F$ is ramified.
\end{enumerate}
Then for any positive integer $n$, $E_n\coloneqq E\otimes_F F_n$ is a ramified quadratic field extension of $F_n$, where $F_n$ is the unique degree $n$ unramified extension of $F$.
\end{lemma}
\begin{proof}
By our assumptions, we can choose an uniformizer $\pi$ of $E$, such that $\pi^2$ is a uniformizer of $F$. Then we have
\[
	E\simeq F[x]/\left(x^2-\pi^2\right).
\]
Notice that for each positive integer $n$, $\pi^2$ is also the uniformizer of $F_n$, hence the polynomial $x^2-\pi^2$ is also irreducible in $F_n$. It then follows that $E_n\coloneqq E\otimes_F F_n$ is a ramified quadratic field extension of $F_n$.
\end{proof}

In the rest of this appendix, we assume that the local fields $F$ and $E$ satisfy the conditions in this lemma. We also retain the notations in the proof of this lemma. Let $\Gamma_n$ be the Galois group of $F_n/F$, and $k_n$ the residue field of both $F_n$ and $E_n$. We denote by $d_{u}$ the degree of $k_1$ over $\FF_p$. In spirit of Theorem \ref{App.B.6}, we prove the following:
\begin{theorem}
Fix a positive integer $n$. Let $N$ be an arbitrary positive integer. Suppose
\[
	\frac{p^{d_u}-1}{n}>N.
\]
Then, for $\star\in\{S,O\}$, we have
\[
	|\Irr_F^{\star,n}|>N.
\]
\end{theorem}
\begin{proof}
By the structure theorem of local fields, we have
\[
	E_n^\times\simeq\pi^\ZZ\times k_n^\times\times U_{E_n}^1,
\]
and 
\[
	\Nm_{E_n/F_n}\left(E_n^\times\right)\simeq\pi^{2\ZZ}\times \left(k_n^\times\right)^2\times U_{F_n}^1,
\]
where $U_{E_n}^1$ and $U_{F_n}^1$ are as in the proof of Theorem \ref{App.B.6}. These two isomorphisms are indeed $\Gamma_n$-equivariant. They induce another $\Gamma_n$-equivariant isomorphism
\[
	E_n^\times/\Nm_{E_n/F_n}\left(E_n^\times\right)\simeq \left(\ZZ/2\ZZ\right)\times \left(k_n^\times/\left(k_n^\times\right)^2\right)\times \left(U_{E_n}^1/U_{F_n}^1\right).
\]
Notice that as $\Gamma_n$-modules, $U_{F_n}^1\subset U_{E_n}^2$. Hence we obtain a $\Gamma_n$-equivariant surjection
\[
	U_{E_n}^1/U_{F_n}^1\lra U_{E_n}^1/U_{E_n}^2\simeq k_{E_n}.
\]
By playing the same trick as in the proof of Theorem \ref{App.B.6}, for any non-trivial character $\chi$ of $k_E$, we can produce a character $\wt\chi$ of $E_n^\times$ as follows: firstly we fix some $x\in k_{E_n}$ such that $\{\gamma(x)~|~\gamma\in\Gamma_n\}$ is a basis of $k_{E_n}$ over $k_E$, which allow us to identify the Pontryagin dual of $k_{E_n}$ with the set of $n$-tuples of characters of $k_E$ by
\begin{equation*}
(\chi_\gamma)_{\gamma\in\Gamma_n}\longmapsto\Big(\lambda\cdot\gamma(x)\mapsto\chi_\gamma(\lambda)~,~\lambda\in k_E\Big);
\end{equation*}
then, under this identification, $\chi$ can be regarded as a character of $k_{E_n}$ via
\begin{equation*}
\chi\mapsto(\chi,\mathbbm{1},\cdots,\mathbbm{1}),
\end{equation*}
which we shall still denote by $\chi$; finally we pull back the character 
\[
	\mathbbm{1}\boxtimes \sgn \boxtimes \chi\quad\quad(\textit{resp. } \mathbbm{1}\boxtimes \mathbbm{1} \boxtimes \chi)
\]
of $\left(\ZZ/2\ZZ\right)\times \left(k_n^\times/\left(k_n^\times\right)^2\right)\times k_{E_n}$ along the natural projections
\begin{equation*}
E_n^\times\lra E_n^\times/\Nm_{E_n/F_n}\left(E_n^\times\right)\lra \left(\ZZ/2\ZZ\right)\times \left(k_n^\times/\left(k_n^\times\right)^2\right)\times k_{E_n}.
\end{equation*}
We denote the image of $\chi$ in $\wh{E_n^\times}$ by $\wt\chi$. By our construction, $\wt\chi$ satisfies:
\begin{enumerate}
\item $\wt\chi$ is conjugate symplectic (resp. conjugate orthogonal) with respect to $F_n$;
\item for any $1\neq\gamma\in \Gamma_n$, we have $\wt\chi^{\gamma}\neq\wt\chi$.
\end{enumerate}
Therefore $\Ind_{W_{E_n}}^{W_E}\wt\chi$ is conjugate symplectic (resp. conjugate orthogonal) and irreducible of dimension $n$. This construction gives an injection
\begin{equation*}
\Big(\widehat{k_E}\backslash\{\mathbbm{1}\}\Big)/\sim~\lra\Irr_F^{S,n}(E)\quad\quad\bigg(\textit{resp. } \Big(\widehat{k_E}\backslash\{\mathbbm{1}\}\Big)/\sim~\lra\Irr_F^{O,n}(E)\bigg).
\end{equation*}
Here we define $\chi_1\sim\chi_2$ if there is some $1\neq\gamma\in \Gamma_n$, such that $\wt\chi_1^{\gamma}=\wt\chi_2$. Notice that there are at most $n$ elements in each equivalence class. Hence the LHS of this injection has at least
\begin{equation*}
\frac{p^{d_u}-1}{n}
\end{equation*} 
elements. By our assumption, we are done.
\end{proof}

\section{Existence of certain number fields} % (fold)
\label{Existence.NumberField}

In this appendix, we prove the existence of certain number fields. The results in this appendix will be used in the proof of Corollary \ref{Globalization-2}. We first start with a well known general result.

\begin{theorem}
Let $\dot{F}$ be a number field, and $v_1,v_2,\cdots,v_r$ be inequivalent places of $\dot{F}$. Let $F_i=\dot{F}_{v_i}$, and $K_i$ a finite extension of $F_i$ of degree $d_i$. We set 
\[
	d = \max\left\{d_i ~\big|~1\leq i \leq r\right\}.
\]
Then there exists a degree $d$ extension $\dot{K}$ of $\dot{F}$, and places $v'_i$ of $\dot{K}$ above $v_i$, such that
\[
	\dot{K}_{v'_i} \simeq K_i
\]
as extensions of $F_i$ for all $i=1,2,\cdots,r$. 
\end{theorem}

\begin{proof}
This is a simple application of Krasner's lemma and the weak approximation theorem. Since we don't know a convenient reference, for sake of the completeness we briefly sketch the proof here.

Since we are considering characteristic $0$ fields, any finite extension is simple. For each $i=1,2,\cdots,r$, let $\alpha_i\in K_i$ be an element such that $K_i=F_i\left(\alpha_i\right)$, and $g'_i$ the minimal polynomial of $\alpha_i$ over $F_i$. Then 
\[
	F_i[x]/\left(g'_i\right) \simeq K_i.
\]
Let $\beta'_{i,1},\beta'_{i,2},\cdots,\beta'_{i,d-d_i}\in F_i$, such that they are distinct, and $g'_i\left(\beta'_{i,j}\right) \neq 0$ for all $1\leq j\leq d-d_i$. We put
\[
	f_i(x) = \left(x-\beta'_{i,1}\right)\left(x-\beta'_{i,2}\right)\cdots\left(x-\beta'_{i,d-d_i}\right)\cdot g'_i(x).
\]
By the weak approximation theorem, we can take a monic polynomial $f\in\dot{F}[x]$ of degree $d$, such that for all $1\leq i\leq r$, coefficients of $f$ are arbitrarily close to coefficients of $f_i$ (with respect to the valuation $v_i$). Then by Krasner's lemma and some classical analysis, we can take $f$ so close such that for all $1\leq i\leq r$, we have:
\begin{itemize}
	\item $f$ can be decomposed as
\[
	f(x) = \left(x-\beta_{i,1}\right)\left(x-\beta_{i,2}\right)\cdots\left(x-\beta_{i,d-d_i}\right)\cdot g_i(x)
\]
for some $\beta_{i,1},\beta_{i,2},\cdots,\beta_{i,d-d_i}\in F_i$ and $g_i\in F_i[x]$;

\item there is an isomorphism
\[
	F_i[x]/\left(g_i\right) \simeq F_i[x]/\left(g'_i\right)
\]
as $F_i$-algebras.
\end{itemize}
Note that there exists some $i_0\in\left\{1,2,\cdots,r\right\}$, such that $d_{i_0}=d$. It follows that $g'_{i_0}$ is an irreducible polynomial of degree $d$ in $F_i[x]$. Consequently, $f$ is also irreducible in $\dot{F}[x]$. Therefore 
\[
	\dot{K} \coloneqq \dot{F}[x]/\left(f\right)
\]
is a field, and one can easily check that $\dot{K}$ satisfies all our requirements. 
\end{proof}

Now let $F$ be a local field, and $E$ either $F$ it self or a quadratic field extension of $F$. Using the theorem above we can prove the existence of the pair of number fields claimed in the proof of Corollary \ref{Globalization-2}.

\begin{corollary}
Given a positive integer $d$ and a prime number $p$. There exists a pair of number fields $(\dot E,\dot{F})$, together with three places $u_1,u_2,w$ of $\dot F$, satisfying the conditions below: 
\begin{enumerate}
	\item $(\dot E_{u_1},\dot F_{u_1})\simeq(\dot E_{u_2},\dot F_{u_2})\simeq(E,F)$; 
	\item $\dot F_w$ is a finite extension of $\QQ_p$, and the degree of the residue field $k_w$ of $\dot F_w$ over $\mathbb{F}_p$ is greater than $d$; % (this condition guarantees we will have enough irreducible orthogonal representations of $W_{\dot F_w}$; see Appendix \ref{self.dual.Gal.rep});
	\item If $E$ is a quadratic extension of $F$, then $\dot E_w$ is a ramified quadratic field extension of $\dot F_w$; if further that $F$ is non-Archimedean, then $\dot{F}$ is totally imaginary. %(this condition guarantees we will have enough irreducible conjugate self-dual representations of $W_{\dot E_w}$, with any given parity; see Appendix \ref{conj.self.dual.gal.rep})
\end{enumerate} 
\end{corollary}

\begin{proof}
We shall construct the desired number fields from $\mathbb{Q}$. Let 
\[
	v = \begin{cases}
		\ell \quad &\textit{if $F$ is a finite extension of $\mathbb{Q}_\ell$};\\

		\infty \quad &\textit{if $F$ is Archimedean}. 
	\end{cases}
\]
Let $\dot F_0$ to be a finite extension of $\mathbb{Q}$, together with three places $u'_1, u'_2, w$, such that $\dot F_{0,u'_1}$ and $\dot F_{0,u'_2}$ are subfields of $F$, and $w'$ is above $p$. Such $\dot F_0$ clearly exists: we can take $\dot F_{0}$ to be a quadratic, or biquadratic extension of $\mathbb{Q}$, depending on $v$ equals to $p$ or not, such that $v$ is totally split. Let $F_w$ be a finite extension of $\dot F_{0,w'}$, such that
\[
	[k_w:\mathbb{F}_p] >d,
\]
where $k_w$ is the residue field of $F_w$. Applying the theorem above, we know that there exists a number field $\dot F$, together with three places $u_1, u_2, w$ above $u'_1, u'_2, w'$, such that
\[
	\dot F_{u_1}\simeq \dot F_{u_2} \simeq F, \quad \textit{and} \quad \dot F_{w}\simeq F_w.
\]
If further that $F$ is non-Archimedean, we can take $\dot F$ to be totally imaginary. Indeed, if $\dot F$ is not totally imaginary, let $R$ be the set of all real places of $\dot F$. Applying the theorem above, we obtain a quadratic extension $\dot F'$ of $\dot F$, such that $u_1, u_2, w$ split in $\dot F'$, and 
\[
	\dot F'_u \simeq \CC
\]
for all $u\in R$. Then $\dot F'$ is a totally imaginary field satisfies our requirements, and we may replace $\dot F$ by $\dot F'$. Finally, if $E$ is a quadratic extension of $F$, once again it follows from the theorem above that there exists a number field $\dot E$, such that
\[
	\dot E_{u_1}\simeq \dot E_{u_2} \simeq E,
\]
and $\dot E_w$ is a ramified extension of $\dot F_w$.
\end{proof}

\bibliographystyle{alpha}
%\nocite{*}
\bibliography{ref}

\end{document}